\pdfoutput=1
\nonstopmode

\documentclass[leqno,a4paper,twoside]{book}
\usepackage{verse}
\usepackage{calc}
\usepackage{amsmath, amssymb, amsthm, amscd}
\usepackage[numbers]{natbib}
\usepackage{url}
\usepackage{currvita}
\usepackage{graphicx}
\usepackage{listings}
\usepackage{pdflscape}
\usepackage[polutonikogreek,latin,ngerman,english]{babel}
\usepackage{psfrag}
\usepackage[notightpage]{pst-pdf}
\usepackage{epsfig}
\usepackage[all]{xy}
\newdir{ >}{{}*!/-10pt/@{>}}
\newdir^{ (}{{}*!/-5pt/@^{(}}
\newdir_{ (}{{}*!/-5pt/@_{(}}
\def\cgaps#1{}
\def\Cgaps#1{}

\def\undersetbrace#1\to#2{\underbrace{#2}_{#1}}								 
\def\oversetbrace#1\to#2{\overbrace{#2}^{#1}}
\def\AMSunderset#1\to#2{\underset{#1}{#2}}
\def\AMSoverset#1\to#2{\overset{#1}{#2}}

\def\East#1#2{\overset{#1}{\longrightarrow}}
\def\norm#1{\left\|{#1}\right\|}
\def\adj#1{\on{Adj}({#1})}

\newtheorem{ass}{Assumption}[subsection]
\newtheorem*{ass*}{Assumption}

\swapnumbers

\newtheorem*{prop*}{Proposition}

\newtheorem*{thm*}{Theorem}

\newtheorem*{lem*}{Lemma}

\newtheorem*{cor*}{Corollary}

\numberwithin{equation}{subsection}
\numberwithin{equation}{section}
\newenvironment{demo}[1]{{\textit{#1.}}}{\par\smallskip}

\def\ign#1{}             
\def\o{\circ}
\def\X{\mathfrak X}
\def\al{\alpha}
\def\be{\beta}

\def\ep{\varepsilon}
\def\ze{\zeta}
\def\et{\eta}

\def\la{\lambda}

\def\si{\sigma}

\def\ph{\varphi}

\def\om{\omega}
\def\Ga{\Gamma}
\def\De{\Delta}

\def\La{\Lambda}

\def\Om{\Omega}
\def\i{^{-1}}
\def\x{\times}
\def\p{\partial}
\let\on=\operatorname
\def\D{T} 
\def\L{\mathcal L}
\def\grad{\on{grad}}%
\def\AMSonly#1{}

\def\Id{\on{Id}}
\def\R{\mathbb R}
\def\Tr{\on{Tr}}
\def\vol{{\on{vol}}}
\def\Vol{{\on{Vol}}}
\def\Imm{{\on{Imm}}}
\def\Emb{{\on{Emb}}}
\def\Diff{{\on{Diff}}}
\def\g{\overline{g}}
\def\grad{\on{grad}}
\def\dist{{\on{dist}}}
\def\Nor{{\on{Nor}}}
\def\sym{{\on{sym}}}
\def\alt{{\on{alt}}}

\def\Hor{{\on{Hor}}}
\def\hor{{\on{hor}}}
\def\vert{{\on{vert}}}

\begin{document}

\frontmatter
\title{Sobolev metrics on shape space of surfaces}
\author{Philipp Harms}

\begin{titlepage}
\vspace*{-2cm}  
\begin{flushright}
    \includegraphics{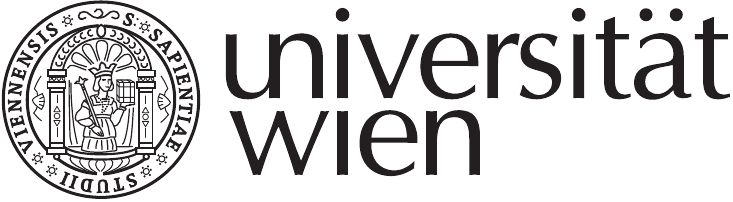}
\end{flushright}
\vspace{1cm}

\begin{center}  
    \Huge{\textbf{\textsf{\MakeUppercase{
        Dissertation
    }}}}
    \vspace{1.5cm}

    \large{\textsf{  
    Titel der Dissertation 
    }}
    \vspace{.1cm}

    \LARGE{\textsf{  
        Sobolev metrics on shape space of surfaces
    }}
    \vspace{2cm}

    \large{\textsf{  
        Verfasser
    }}

    \Large{\textsf{  
        Philipp Harms
    }}
    \vspace{2cm}

    \large{\textsf{
        angestrebter akademischer Grad  
    }}

    \Large{\textsf{  
       Doktor der Naturwissenschaften (Dr. rer. nat.)
    }}
\end{center}
\vspace{1.5cm}

\noindent\textsf{Wien, Dezember 2010}  
\vfill

\noindent\begin{tabular}{@{}ll}
\textsf{Studienkennzahl lt.\ Studienblatt:}
&
\textsf{A 091 405}  
\\
\textsf{Dissertationsgebiet lt.\ Studienblatt:}
&
\textsf{Mathematik}  
\\
\textsf{Betreuer:}
&
\textsf{Ao. Univ.-Prof. Dr. Peter W. Michor}  
\end{tabular}

\end{titlepage}

\cleardoublepage

\long\def\symbolfootnote[#1]#2{\begingroup%
\def\thefootnote{\fnsymbol{footnote}}\footnotetext[#1]{#2}\endgroup}

\mbox{}%
\symbolfootnote[1]{The persistence of memory by Salvador Dali. 
Picture taken from
\url{http://en.wikipedia.org/wiki/File:The_Persistence_of_Memory.jpg}, November 2010.}
\symbolfootnote[2]{Proemium of Ovid's metamorphoses.}

\begin{figure}[h]
\centering
\includegraphics[width=\textwidth-1pt]{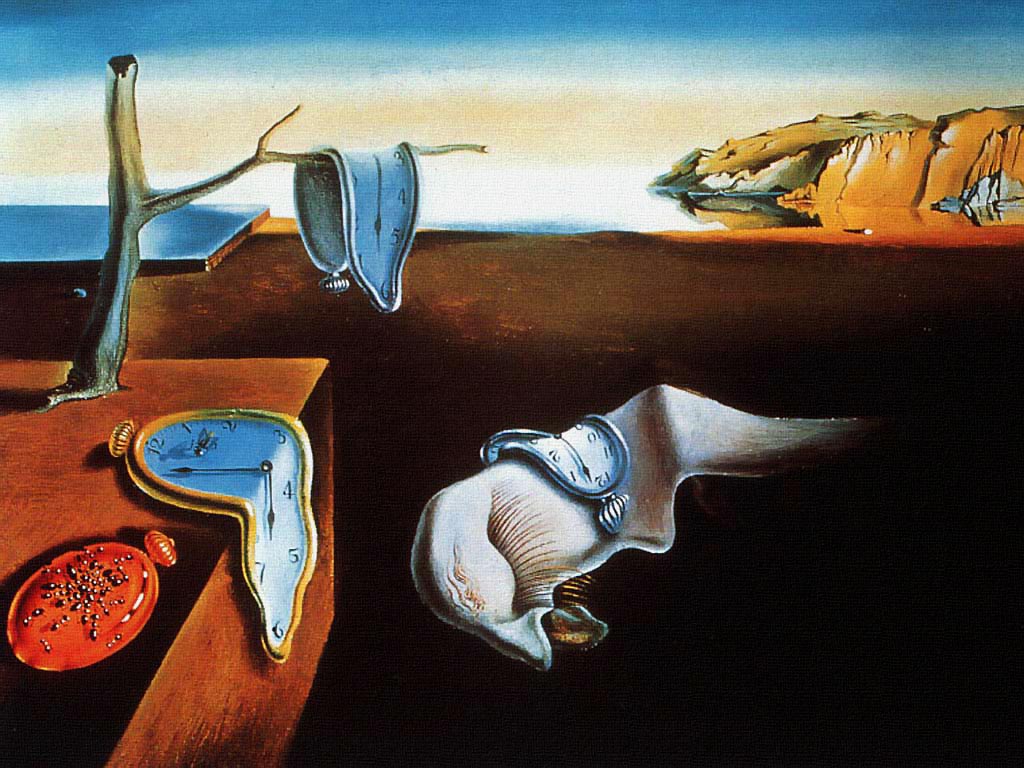}

\vspace{3cm}

\centering

\settowidth{\versewidth}{corpora; di, coeptis (nam vos mutastis et illas)}
\begin{verse}[\versewidth]
\selectlanguage{latin}
In nova fert animus mutatas dicere formas
\\
corpora; di, coeptis (nam vos mutastis et illas)
\\
adspirate meis primaque ab origine mundi
\\
ad mea perpetuum deducite tempora carmen!
\selectlanguage{english}
\end{verse}

\end{figure}

\cleardoublepage
\addcontentsline{toc}{chapter}{Preface}
\markboth{}{}
\addcontentsline{toc}{section}{Abstract}

\noindent
{\sc Abstract.}
Many procedures in science, engineering and medicine produce data in the form of geometric shapes.
Mathematically, a shape can be modeled as an un-parameterized immersed sub-manifold, which is the
notion of shape used here. Endowing shape space with a Riemannian metric opens up the world of
Riemannian differential geometry with geodesics, gradient flows and curvature. Unfortunately, the
simplest such metric induces vanishing path-length distance on shape space. This discovery by Michor
and Mumford was the starting point to a quest for stronger, meaningful metrics that should be able to
distinguish salient features of the shapes. Sobolev metrics are a very promising approach to that. They
come in two flavors: Outer metrics which are induced from metrics on the diffeomorphism group of
ambient space, and inner metrics which are defined intrinsically to the shape. In this work, Sobolev
inner metrics are developed and treated in a very general setting. There are no restrictions on the
dimension of the immersed space or of the ambient space, and ambient space is not required to be
flat. It is shown that the path-length distance induced by Sobolev inner metrics does not vanish. The
geodesic equation and the conserved quantities arising from the symmetries are calculated, and
well-posedness of the geodesic equation is proven. Finally examples of numerical solutions to the
geodesic equation are presented.

\cleardoublepage
\addcontentsline{toc}{section}{Acknowledgment}

\noindent
{\sc Acknowledgment.}
I am very thankful for the support I have received from so many sides during my time
as a Ph.D. student. 
First of all I want to thank my very dear advisor Peter Michor. 
I would also like to thank Alain Trouv\'e, Bianca Falcidieno, Silvia Biasotti, 
Darryl Holm and Andrea Mennucci who have invited me to research visits at their institutes; 
David Mumford, Stefan Haller, Johannes Wallner and Hermann Schichl 
who have helped me in mathematical questions; 
Peter Gruber and Josef Teichmann for their 
ongoing caring interest in my mathematical and personal welfare;
Andreas Kriegl, Dietrich Burde, Armin Rainer and Andreas Nemeth for their merry company 
at so many meals and coffee breaks; 
Martin Bauer for bearing with me during all those years as friends and colleagues;
my friends, siblings and parents for being such good friends, siblings and parents!

\addcontentsline{toc}{section}{Contents}
\setcounter{tocdepth}{2}
\tableofcontents

\mainmatter
\chapter{Introduction}\label{in}
\section{The Riemannian setting for shape analysis}

From very early on, shapes spaces have been analyzed in a Riemannian setting. 
This is also the setting that has been adopted in this work.
The Riemannian setting is well-suited to shape analysis for several reasons. 
\begin{itemize}
\item It formalizes an \emph{intuitive notion of similarity} of shapes: Shapes that
differ only by a small \emph{deformation} are similar to each other. To compare shapes, it
is thus necessary to measure deformations. This is exactly what is accomplished by 
a Riemannian metric. A Riemannian metric measures continuous deformations of shapes, that is, 
paths in shape space. 
\item Riemannian metrics on shape space have been used successfully in 
\emph{computer vision} for a long time, often without any mention of the 
underlying metric. Gradient flows for shape smoothing are an example.  
An underlying metric is needed for the definition of a gradient.
Often, the metric that has been used implicitly was the $L^2$-metric
that has however turned out to be too weak.
\item The exponential map that is induced by a Riemannian metric permits to
\emph{linearize shape space}: When shapes are represented as initial velocities of 
geodesics connecting them to a fixed reference shape, one effectively works
in the linear tangent space over the reference shape. Curvature will play an 
essential role in quantifying the deviation of curved shape space from its 
linearized approximation. 
\item The linearization of shape space by the exponential map
allows to do \emph{statistics} on shape space. 
\end{itemize}
A disadvantage of the Riemannian approach is that shapes can be compared with each other
only when there is a deformation between them. 

\section{Related work}

In mathematics and computer vision, shapes have been represented in many ways.
Point clouds, meshes, level-sets, graphs of a function, currents and measures are
but some of the possibilities.
Furthermore, the resulting shape spaces have been endowed with many different metrics. 
Approaches found in the literature include: 
\begin{itemize}
\item Inner metrics on shape space of unparametrized immersions.
	These metrics are induced from metrics on parametrized immersions. 
	Since this is the approach studied in this work, more references will be given later. 
\item Outer metrics on various shape spaces 
	(images, embedded surfaces, landmarks, measures and currents)
	that the diffeomorphism group of ambient space is acting on. See for example
	\cite{Bajcsy1983,Beg2005,Michor107,Trouve2002,
		Bookstein1997,Michor121,Kendall1984,
		Glaunes2008
		}. 
\item Metamorphosis metrics. See for example \cite{Trouve2005,Holm2009}.
\item The Wasserstein metric or Monge-Kantorovic metric on shape space of probability measures. 
	See for example \cite{Ambrosio2004,Ambrosio2008,Benamou2000,Benamou2002}.
\item The Weil-Peterson metric on shape space of planar curves. 
	See for example \cite{Mumford2004,Mumford2006,Kushnarev2009}.
\item Current metrics. See for example \cite{Vaillant2005, Trouve2008b,Trouve2009}.
\end{itemize}

More references can be found in the review papers 
\cite{Alt1996,Arman1993,Besl1985,Loncaric1998,Pope1994,Veltkamp1999}.

\section{Inner versus outer metrics}

\begin{figure}[p]
\centering
\includegraphics[width=.7\textwidth]{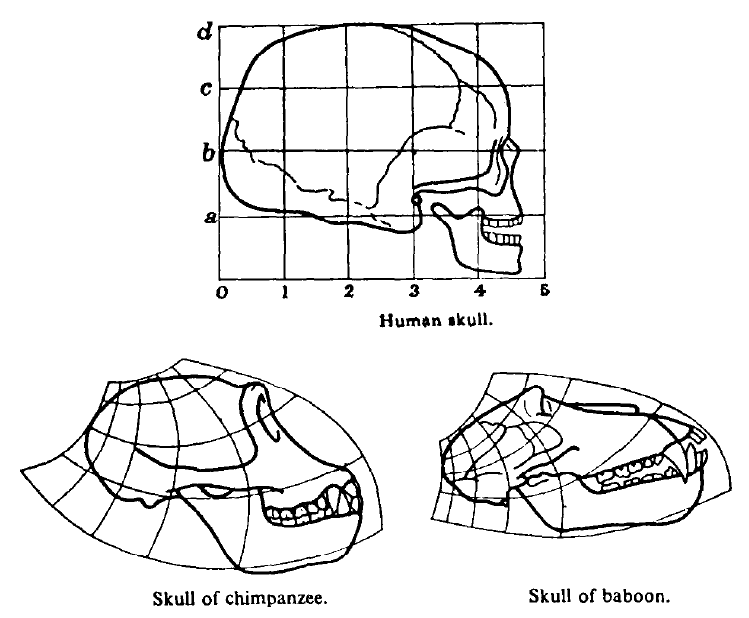}
\caption{Transformation of the ambient space carrying with it 
a transformation from human skull to chimpanzee and baboon. From 
\cite[p.~318--319]{Thompson1942}. }
\label{in:sk}
\vspace{.3in}
\includegraphics[width=.7\textwidth]{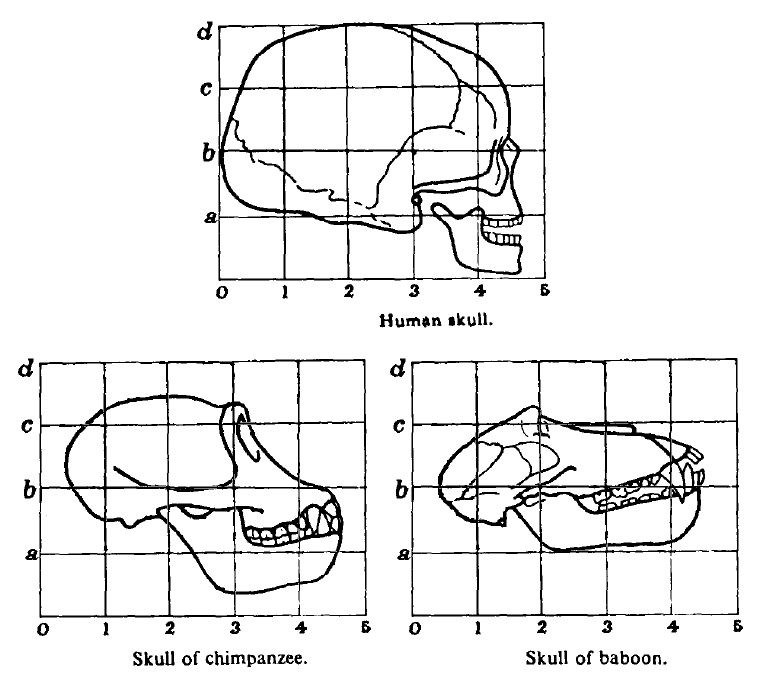}
\caption{The same transformation as in figure~\ref{in:sk}, but
only the shape is transformed, not the ambient space.
Modification of \cite[p.~318--319]{Thompson1942}. }
\label{in:sk2}
\end{figure}

\emph{Outer metrics} measure how much ambient space has to be deformed 
in order to yield the desired deformation of the shape. 
This concept has been introduced by the Scottish biologist, mathematician 
and classics scholar d'Arcy Thompson in 1942, already. 
As Thompson declared in the epilogue of his book ``On Growth and Form'' \cite{Thompson1942}, 
his aim was to show that ``a certain mathematical aspect of morphology is essential to (the) proper
study of Growth and Form''.
In the chapter ``On the comparison of related forms'' of this book, Thompson 
pictures transformations of the ambient 
space by a Cartesian grid. A transformation of the ambient space and the grid then
results in a deformation of the embedded shape, too. An example is given in figure~\ref{in:sk}. 

Thompson's notion of shape transformation and the concept of outer metrics 
is fundamentally different from the notion of shape transformation that underlies this work. 
In this work, so called \emph{inner metrics} are treated. 
Inner metrics measure how much the shape itself is deformed. A deformation of the shape 
does not carry with it a deformation of the ambient space.
The ambient space is fixed, and the shape
moves within it. This is illustrated in figure~\ref{in:sk2}.

Riemannian metrics on shape space measure \emph{infinitesimal deformations}. 
An infinitesimal deformation of ambient space is a vector field on ambient 
space and could be pictured as a small arrow attached to every point in ambient space. 
This is the kind of deformation that is measured by outer metrics. 
In contrast to this, an infinitesimal deformation of the shape itself is a (normal or horizontal)
vector field along the shape. It could be pictured as a small arrow attached to every point
of the shape. Figure~\ref{in:inout} illustrates the two kinds of deformation.

\begin{figure}[h]
\centering
\includegraphics[width=.75\textwidth]{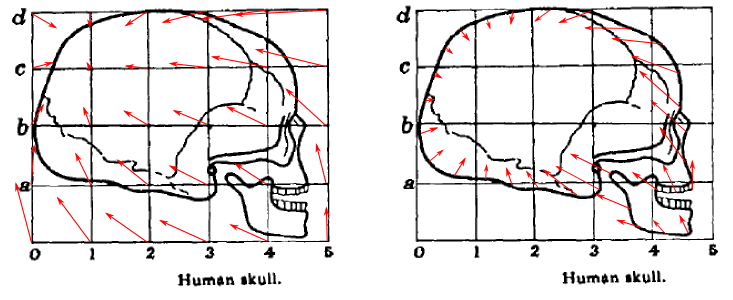}
\caption{Infinitesimal transformation of ambient space (left) as measured
by an outer metric and infinitesimal transformation of the shape itself (right)
as measured by an inner metric.}
\label{in:inout}
\end{figure}

The deformations that are measured by inner metrics have much smaller (lower dimensional)
support than those measured by outer metrics. This is an advantage in \emph{numerics}. 
A disadvantage is that the differential operator governing an inner metric usually depends 
on the shape, which is not the case for outer metrics.

Section~\ref{sh:inout} explains how the difference between inner and outer metrics
becomes manifest in the mathematical theory. Sections~\ref{sh:na} and \ref{sh:rish} explain
in mathematical terms what an infinitesimal deformation of a shape is.

\section{Where this work comes from}

This work continues the study of \emph{inner metrics} on
shape space of \emph{unparametrized immersions} 
of a fixed manifold $M$ into a Riemannian manifold $N$. 
$M$ fixes the topology of the shapes. 
$N$ is the ambient space, and it is endowed with a fixed metric. 
For example, $M$ could be a simple geometric object like a sphere, 
and $N$ could be some $\R^n$. 

It came as a big surprise when Michor and Mumford found out in \cite{Michor98,Michor102} 
that the simplest and most natural Riemannian metric on this shape space (the $L^2$-metric)
induces \emph{vanishing geodesic distance} on shape space. 
More precisely the result is that any two shapes can be 
deformed into each other by a deformation that is arbitrarily small
when measured with respect to this metric. The result also holds for $L^2$-metrics
on diffeomorphism groups and for outer $L^2$-metrics on shape space. 
The discovery of this degeneracy was the starting point to a quest for stronger metrics.
These metrics should be able to distinguish salient features of the shapes.
The meaning of salient of course depends on the application. 

One approach to strengthen the $L^2$ inner metric is by weighting it by a function depending
on the mean curvature and/or the volume \cite{Michor98,Michor102,Michor118,Michor120,Bauer2010}. 
Such metrics have been called \emph{almost local metrics}. 

Another approach, and the approach taken in this work, 
is to incorporate derivatives 
of the deformation vector fields in the definition of the metric.
This yields \emph{Sobolev inner metrics} on shape space. 
Sobolev inner metrics have first been defined on shape space of planar curves
\cite{Michor107,Michor111,Younes1998,TrouveYounes2000,MennucciYezzi2008}.
In this work and in \cite{Michor119}, these metrics are generalized 
to higher dimensional shape spaces and to a possibly curved ambient space. 

Some parts of this work have been written in collaboration with Martin Bauer 
and can also be found in his thesis~\cite{Bauer2010}. This concerns mainly 
sections \ref{no}, \ref{sh} and \ref{va} covering a lot of background material.
 
\section{Content of this work}

This work progresses from a very general setting to a specific one in three steps. 
In the beginning, a framework for general inner metrics is developed. 
Then the general concepts carry over to more and more specific inner metrics. 
\begin{itemize}
\item First, shape space is endowed with a \emph{general inner metric}, i.e with a metric
that is induced from a metric on the space of immersions, but that is unspecified otherwise.
It is shown how various versions of the geodesic equation can be expressed using gradients of the 
metric with respect to itself and how conserved quantities arise from symmetries. 
(This is section~\ref{sh}.)
\item Then it is assumed that the inner metric is defined via an elliptic pseudo-differential operator. 
Such a metric will be called a \emph{Sobolev-type metric}. 
The geodesic equation is formulated in terms of the operator, and
existence of horizontal paths of immersions within each equivalence class of paths is proven. 
(This is section~\ref{so}.)
Then estimates on the path-length distance are derived. Most importantly it is shown that
when the operator involves high enough powers of the Laplacian, then the metric does not have 
the degeneracy of the $L^2$-metric. (This is section~\ref{ge}.)
\item Motivated by the previous results it is assumed that the elliptic pseudo-differential operator 
is given by the \emph{Laplacian} and powers of it.
Again, the geodesic equation is derived. The formulas that are obtained are ready to be implemented numerically. 
(This is section~\ref{la}.)
\end{itemize}
The remaining sections cover the following material: 
\begin{itemize}
\item Section~\ref{no} treats some \emph{differential geometry of surfaces} that is needed throughout 
this work. It is also a good reference for the notation that is used.
The biggest emphasis is on a rigorous treatment of the covariant derivative. 
Some material like the adjoint covariant derivative is not found in standard text books. 
\item Section~\ref{va} contains formulas for the \emph{variation} of the metric, volume form, 
covariant derivative and Laplacian with respect to the immersion inducing them.
These formulas are used extensively later. 
\item Section~\ref{su} covers the special case of \emph{flat ambient space}. 
The geodesic equation is simplified and conserved momenta for the Euclidean motion group
are calculated. Sobolev-type metrics are compared to the Fr\'echet metric which is 
available in flat ambient space. 
\item Section~\ref{di} treats \emph{diffeomorphism groups} of compact manifolds as a 
special case of the theory that has been developed so far.
\item In section~\ref{nu} it is shown in some examples that the geodesic equation 
on shape space can be solved \emph{numerically}. 
\end{itemize}

\section{Contributions of this work.}

\begin{itemize}
\item This work is the first to treat Sobolev inner metrics on spaces of immersed
surfaces and on higher dimensional shape spaces. 

\item It contains the first description of how the geodesic equation can be 
formulated in terms of gradients of the metric with respect to itself 
when the ambient space is not flat. 
To achieve this, a covariant derivative on some bundles over immersions is defined. 
This covariant derivative is induced from the Levi-Civita covariant derivative on ambient space. 

\item The geodesic equation is formulated in terms of this covariant derivative. 
Well-posedness of the geodesic equation is shown under some
regularity assumptions that are verified for Sobolev metrics. 
Well-posedness also follows for the geodesic equation on 
diffeomorphism groups, where this result has not yet been obtained in that full generality.

\item To derive the geodesic equation, a variational formula for the Laplacian operator is developed.
The variation is taken with respect to the metric on the manifold where the Laplacian is 
defined. This metric in turn depends on the immersion inducing it. 

\item It is shown that Sobolev inner metrics separate points in shape space when the 
order of the differential operator governing the metric is high enough.
(The metric needs to be as least as strong as the $H^1$-metric.)
Thus Sobolev inner metrics overcome the degeneracy of the $L^2$-metric. 

\item The path-length distance of Sobolev inner metrics is compared to the Fr\'echet distance. 
It would be desirable to bound F\'echet distance by some Sobolev distance. 
This however remains an open problem. 

\item Finally it is demonstrated in some examples that the geodesic equation 
for the $H^1$-metric on shape space of surfaces in $\R^3$ can be solved numerically. 
\end{itemize}

\chapter[Notation]{Differential geometry of surfaces and notation}\label{no}

In this section the differential geometric tools 
that are needed to deal with immersed surfaces
are presented and developed. 
The most important point is a rigorous treatment of the covariant derivative 
and related concepts. 

The notation of \cite{MichorH} is used. 
Some of the definitions can also be found in \cite{Kobayashi1996a}.
A similar exposition in the same notation is \cite{Michor118, Michor119}.
This section has been written in collaboration with Martin Bauer
and is the same as section~1.1 of his Ph.D. thesis \cite{Bauer2010}, up 
to slight modifications.

\section{Basic assumptions and conventions}\label{no:as}

\begin{ass*}
It is always assumed that $M$ and $N$ are connected manifolds 
of finite dimensions $m$ and $n$, respectively. 
Furthermore it is assumed that $M$ is compact, 
and that $N$ is endowed with a Riemannian metric $\g$.
\end{ass*}

In this work, \emph{immersions} of $M$ into $N$ will be treated, i.e. 
smooth functions $M \to N$ with injective tangent mapping at every point.
The set of all such immersions will be denoted by $\Imm(M,N)$.
It is clear that only the case $\dim(M) \leq \dim(N)$ is of interest 
since otherwise $\Imm(M,N)$ would be empty.

Immersions or paths of immersions are usually denoted by $f$. 
Vector fields on $\Imm(M,N)$ or vector fields along $f$ will be called $h,k,m$, for example. 
Subscripts like $f_t = \p_t f = \p f/\p t$ denote differentiation 
with respect to the indicated variable, but subscripts are also used to indicate the 
foot point of a tensor field.

\section{Tensor bundles and tensor fields}\label{no:te}

The \emph{tensor bundles}
\begin{equation*}\xymatrix{
T^r_s M \ar[d] & 
T^r_s M \otimes f^*TN \ar[d] \\
M & M 
}\end{equation*}
will be used. Here $T^r_sM$ denotes the bundle of 
$\left(\begin{smallmatrix}r\\s\end{smallmatrix}\right)$-tensors on $M$, i.e.
$$T^r_sM=\bigotimes^r TM \otimes \bigotimes^s T^*M,$$
and $f^*TN$ is the pullback of the bundle $TN$ via $f$, see \cite[section~17.5]{MichorH}. 
A \emph{tensor field} is a section of a tensor bundle. Generally, when $E$ is a bundle, 
the space of its sections will be denoted by $\Ga(E)$. 

To clarify the notation that will be used later, 
some examples of tensor bundles and tensor fields are given now.
$S^k T^*M = L^k_{\on{sym}}(TM; \R)$ and
$\La^k T^*M = L^k_{\on{alt}}(TM; \R)$
are the bundles of symmetric and alternating 
$\left(\begin{smallmatrix}0\\k\end{smallmatrix}\right)$-tensors, respectively. 
$\Om^k(M)=\Ga(\La^k T^*M)$ is the space of differential forms,  
$\X(M)=\Ga(TM)$ is the space of vector fields, and 
$$\Ga(f^*TN) \cong \big\{ h \in C^\infty(M,TN): \pi_N \o h = f \big\}$$ 
is the space of \emph{vector fields along $f$}.


\section{Metric on tensor spaces}\label{no:me}

Let $\g \in \Gamma(S^2_{>0} T^*N)$ denote a fixed Riemannian metric on $N$. 
The \emph{metric induced on $M$ by $f \in \Imm(M,N)$} is the pullback metric 
\begin{align*}
g=f^*\g \in \Gamma(S^2_{>0} T^*M), \qquad g(X,Y)=(f^*\g)(X,Y) = \g(Tf.X,Tf.Y),
\end{align*}
where $X,Y$ are vector fields on $M$.
The dependence of $g$ on the immersion $f$ should be kept in mind.
Let $$\flat = \check g: TM \to T^*M \quad \text{and} \quad \sharp=\check g\i: T^*M \to TM.$$ 
$g$ can be extended to the cotangent bundle $T^*M=T^0_1M$ by setting
$$g\i(\alpha,\beta)=g^0_1(\alpha,\beta)=\alpha(\beta^\sharp)$$
for $\alpha,\beta \in T^*M$. 
The product metric 
$$g^r_s = \bigotimes^r g \otimes \bigotimes^s g\i$$
extends $g$ to all tensor spaces $T^r_s M$, and 
$g^r_s \otimes \g$ yields a metric on $T^r_s M \otimes f^*TN$. 
%
%
%

\section{Traces}\label{no:tr}

The \emph{trace} contracts pairs of vectors and co-vectors in a tensor product: 
\begin{align*}
\Tr:\; T^*M \otimes TM = L(TM,TM) \to M \x \R
\end{align*}
A special case of this is the operator
$i_X$ inserting a vector $X$ into a co-vector or into a covariant factor of a tensor product.
The inverse of the metric $g$ can be used to define a trace 
$$\Tr^g: T^*M \otimes T^*M \to M \x \R$$
contracting pairs of co-vecors.
Note that $\Tr^g$ depends on the metric whereas $\Tr$ does not. 
The following lemma will be useful in many calculations: 

\begin{lem*}
\begin{equation*}
g^0_2(B,C)= \on{Tr}(g\i B g\i C) \quad \text{for $B,C \in T^0_2M$ if $B$ or $C$ is symmetric.}
\end{equation*}
(In the expression under the trace, $B$ and $C$ are seen maps $TM \to T^*M$.)
\end{lem*}
\begin{proof}
Express everything in a local coordinate system $u^1, \ldots u^{m}$ of $M$.
\begin{align*}
g^0_2(B,C)&=g^0_2\Big(\sum_{ik} B_{ik}du^i \otimes du^k,\sum_{jl}C_{jl}du^j \otimes du^l\Big) \\ & =
\sum_{ijkl} g^{ij}B_{ik}g^{kl}C_{jl} = \sum_{ijkl} g^{ji}B_{ik}g^{kl}C_{lj} 
= \on{Tr}(g\i B g\i C)
\end{align*}
Note that only the symmetry of $C$ has been used. 
\end{proof}

\section{Volume density}\label{no:vo}

Let $\Vol(M)$ be the \emph{density bundle} over $M$, see \cite[section~10.2]{MichorH}.
The \emph{volume density} on $M$ induced by $f \in \Imm(M,N)$ is 
$$\vol(g)=\vol(f^*\g) \in \Ga\big(\Vol(M)\big).$$
The \emph{volume} of the immersion is given by
$$\Vol(f)=\int_M \vol(f^*\g)=\int_M \vol(g).$$
The integral is well-defined since $M$ is compact. If $M$ is oriented the volume 
density may be identified with a differential form.

\section{Metric on tensor fields}\label{no:me2}

A \emph{metric on a space of tensor fields} is defined by integrating the appropriate metric on the 
tensor space with respect to the volume density:
$$\widetilde{g^r_s}(B,C)=\int_M g^r_s\big(B(x),C(x)\big)\vol(g)(x)$$
for $B,C \in \Gamma(T^r_sM)$, and
$$\widetilde{g^r_s \otimes \g}(B,C) = \int_M g^r_s\otimes \g \big(B(x),C(x)\big)\vol(g)(x)$$
for $B,C \in \Gamma(T^r_sM \otimes f^*TN)$, $f \in \Imm(M,N)$. 
The integrals are well-defined because $M$ is compact.

\section{Covariant derivative}\label{no:co}

Covariant derivatives on vector bundles as explained in \cite[sections 19.12, 22.9]{MichorH}
will be used.
Let $\nabla^g, \nabla^{\g}$ be the \emph{Levi-Civita covariant derivatives} on $(M,g)$
and $(N,\g)$, respectively. 
For any manifold $Q$ and vector field $X$ on $Q$, one has
\begin{align*}
\nabla^g_X:C^\infty(Q,TM) &\to C^\infty(Q,TM), & h &\mapsto \nabla^g_X h \\
\nabla^{\g}_X: C^\infty(Q,TN) &\to C^\infty(Q,TN), & h &\mapsto \nabla^{\g}_X h.
\end{align*}
Usually the symbol $\nabla$ will be used for all covariant derivatives.
It should be kept in mind that $\nabla^g$ depends on the metric $g=f^*\g$ and therefore also on 
the immersion $f$.
The following properties hold \cite[section~22.9]{MichorH}:
\begin{enumerate}
\item \label{no:co:ba}
$\nabla_X$ respects base points, i.e. 
$\pi \o \nabla_X h = \pi \o h$, where $\pi$ is the projection 
of the tangent space onto the base manifold. 
\item
$\nabla_X h$ is $C^\infty$-linear in $X$. So for a tangent vector $X_x \in T_xQ$, 
$\nabla_{X_x}h$ makes sense and equals $(\nabla_X h)(x)$.
\item
$\nabla_X h$ is $\R$-linear in $h$.
\item
$\nabla_X (a.h) = da(X).h + a.\nabla_X h$ for $a \in C^\infty(Q)$, the derivation property of $\nabla_X$.
\item \label{no:co:prop5}
For any manifold $\widetilde Q$ and smooth mapping 
$q:\widetilde Q \to Q$ and $Y_y \in T_y \widetilde Q$ one has
$\nabla_{Tq.Y_y}h=\nabla_{Y_y}(h \o q)$. If $Y \in \X(Q_1)$ and $X \in \X(Q)$ are $q$-related, then 
$\nabla_Y(h \o q) = (\nabla_X h) \o q$.
\end{enumerate}
The two covariant derivatives $\nabla^g_X$ and $\nabla^{\g}_X$ 
can be combined to yield a covariant derivative $\nabla_X$ acting on
$C^\infty(Q,T^r_sM \otimes TN)$ by additionally requiring the following properties 
\cite[section 22.12]{MichorH}:
\begin{enumerate}
\setcounter{enumi}{5} 
\item $\nabla_X$ respects the spaces $C^\infty(Q,T^r_sM \otimes TN)$. 
\item $\nabla_X(h \otimes k) = (\nabla_X h) \otimes k + h \otimes (\nabla_X k)$, 
a derivation with respect to the tensor product.
\item $\nabla_X$ commutes with any kind of contraction (see \cite[section 8.18]{MichorH}). 
A special case of this is
$$\nabla_X\big(\alpha(Y)\big)=(\nabla_X \alpha)(Y)+\alpha(\nabla_X Y) \quad 
\text{for } \alpha\otimes Y :N \to T^1_1M.$$
\end{enumerate}
Property \eqref{no:co:ba} is important because it implies that $\nabla_X$ 
respects spaces of sections of bundles. 
For example, for $Q=M$ and $f \in C^\infty(M,N)$, one gets
$$\nabla_X : \Ga(T^r_s M \otimes f^* TN) \to \Ga(T^r_s M \otimes f^* TN). $$

\section{Swapping covariant derivatives}\label{no:sw}

Some formulas allowing to swap covariant derivatives will be used repeatedly. 
Let $f$ be an immersion, $h$ a vector field along $f$ and $X,Y$ vector fields on $M$. 
Since $\nabla$ is torsion-free, one has \cite[section~22.10]{MichorH}:
\begin{equation}\label{no:sw:to}
\nabla_X Tf.Y-\nabla_Y Tf.X -Tf.[X,Y] = \on{Tor}(Tf.X,Tf.Y) = 0.
\end{equation}
Furthermore one has \cite[section~24.5]{MichorH}:
\begin{equation}\label{no:sw:r}
\nabla_X \nabla_Y h - \nabla_Y \nabla_X h - \nabla_{[X,Y]} h 
= R^{\g} \o (Tf.X,Tf.Y) h, 
\end{equation}
where $R^{\g} \in \Om^2\big(N;L(TN,TN)\big)$ is the Riemann curvature tensor of $(N,\g)$.

These formulas also hold when $f:\R \x M \to N$ is a path of immersions, 
$h:\R \x M \to TN$ is a vector field along $f$ and
the vector fields are vector fields on $\R \x M$. 
A case of special importance is when one of the vector fields is $(\p_t,0_M)$ and the 
other $(0_\R,Y)$, where $Y$ is a vector field on $M$. 
Since the Lie bracket of these vector fields vanishes, 
\eqref{no:sw:to} and \eqref{no:sw:r} yield
\begin{equation}\label{no:sw:to2}
\nabla_{(\p_t,0_M)} Tf.(0_{\R},Y)-\nabla_{(0_{\R},Y)} Tf.{(\p_t,0_M)} = 0
\end{equation}
and
\begin{equation}\label{no:sw:r2}
\nabla_{(\p_t,0_M)} \nabla_{(0_\R,Y)} h - \nabla_{(0_\R,Y)} \nabla_{(\p_t,0_M)} h
\\= R^{\g} \big(Tf.(\p_t,0_M),Tf.(0_\R,Y)\big) h .
\end{equation}

\section{Second and higher covariant derivatives}\label{no:co2}

When the covariant derivative is seen as a mapping
$$\nabla: \Gamma(T^r_s M) \to \Gamma(T^r_{s+1}M)\quad \text{or} \quad
\nabla : \Gamma(T^r_sM \otimes f^*TN) \to \Gamma(T^r_{s+1}M \otimes f^*TN),$$
then the \emph{second covariant derivative} is simply $\nabla\nabla=\nabla^2$.
Since the covariant derivative commutes with contractions,
$\nabla^2$ can be expressed as
$$\nabla^2_{X,Y} :=\iota_Y \iota_X \nabla^2 =
\iota_Y \nabla_X \nabla =
\nabla_X\nabla_Y -\nabla_{\nabla_XY} \qquad \text{for $X,Y\in \X(M)$.}$$
Higher covariant derivates are defined accordingly as $\nabla^k$, $k \geq 0$. 

\section{Adjoint of the covariant derivative}\label{no:co*}

The covariant derivative
$$\nabla: \Gamma(T^r_sM) \to  \Gamma( T^r_{s+1}M)$$
admits an \emph{adjoint} 
$$\nabla^*:\Gamma( T^r_{s+1}M)\to \Gamma(T^r_sM)$$
with respect to the metric $\widetilde{g}$, i.e.: 
$$\widetilde{g^r_{s+1}}(\nabla B, C)= \widetilde{g^r_s}(B, \nabla^* C).$$
In the same way, $\nabla^*$ can be defined when $\nabla$ is acting on 
$\Ga(T^r_s M \otimes f^*TN)$.
In either case it is given by
$$\nabla^*B=-\on{Tr}^g(\nabla B), $$
where the trace is contracting the first two tensor slots of $\nabla B$. 
This formula will be proven now: 

\begin{proof}
The result holds for decomposable tensor fields $\be \otimes B \in \Ga(T^r_{s+1} M)$ since
\begin{align*} &
\widetilde {g^r_s}\Big(\nabla^*(\be \otimes B),C\Big) = 
\widetilde {g^r_{s+1}}\Big(\be \otimes B,\nabla C\Big) = 
\widetilde {g^r_{s}}\Big(B,\nabla_{\be^\sharp} C\Big) \\&\qquad= 
\int_M \L_{\be^\sharp} g^r_{s}(B, C) \vol(g) - \int_M g^r_s(\nabla_{\be^\sharp} B,C) \vol(g) 
\\&\qquad= 
\int_M -g^r_{s}(B, C) \L_{\be^\sharp} \vol(g) - \int_M g^r_s\big(\Tr^g(\be \otimes \nabla B),C\big) \vol(g) 
\\&\qquad= 
\widetilde {g^r_{s}}\Big(-\on{div}(\be^\sharp) B - \Tr^g(\be \otimes \nabla B),C\Big) \\&\qquad= 
\widetilde {g^r_{s}}\Big(-\on{div}(\be^\sharp) B + \Tr^g((\nabla\be) \otimes B) 
-\Tr^g(\nabla (\be \otimes B))  ,C\Big)
\\&\qquad= 
\widetilde {g^r_{s}}\Big(-\on{div}(\be^\sharp) B + \Tr^g(\nabla\be) B 
-\Tr^g(\nabla (\be \otimes B))  ,C\Big)\\&\qquad= 
\widetilde {g^r_{s}}\Big(0-\Tr^g(\nabla (\be \otimes B)), C\Big)
\end{align*}
Here it has been used that $\nabla_X g=0$, 
that $\nabla_{X}$ commutes with any kind of contraction and acts as a derivation 
on tensor products \cite[section~22.12]{MichorH}
and that $\on{div}(X) = \Tr(\nabla X)$ for all vector fields $X$ \cite[section~25.12]{MichorH}.
To prove the result for $\be \otimes B \in \Ga(T^r_{s+1} M \otimes f^*TN)$ one simply has
to replace $g^r_{s}$ by $g^r_{s} \otimes \g$. 
\end{proof}

\section{Laplacian}\label{no:la}

The definition of the Laplacian used in this work is the \emph{Bochner-Laplacian}. 
It can act on all tensor fields $B$ and is defined as
$$\Delta B = \nabla^*\nabla B = - \on{Tr}^g(\nabla^2 B).$$ 


\section{Normal bundle}\label{no:no}

The \emph{normal bundle} $\Nor(f)$ of an immersion $f$ is a sub-bundle of $f^*TN$ 
whose fibers consist of all vectors that are orthogonal to the image of $f$:
$$\Nor(f)_x = \big\{ Y \in T_{f(x)}N : \forall X \in T_xM : \g(Y,Tf.X)=0  \big\}.$$
If $\dim(M)=\dim(N)$ then the fibers of the normal bundle are but the zero vector. 
Any vector field $h$ along $f \in \Imm$ can be decomposed uniquely 
into parts {\it tangential} and {\it normal} to $f$ as
$$h=Tf.h^\top + h^\bot,$$ 
where $h^\top$ is a vector field on $M$ and $h^\bot$ is a section of the normal bundle $\Nor(f)$. 


\section{Second fundamental form and Weingarten mapping}\label{no:we}

Let $X$ and $Y$ be vector fields on $M$. 
Then the covariant derivative $\nabla_X Tf.Y$ splits into tangential and a normal parts as
$$\nabla_X Tf.Y=Tf.(\nabla_X Tf.Y)^\top + (\nabla_X Tf.Y)^\bot = Tf.\nabla_X Y + S(X,Y).$$
$S$ is the \emph{second fundamental form of $f$}. 
It is a symmetric bilinear form with values in the normal bundle of $f$. 
When $Tf$ is seen as a section of $T^*M \otimes f^*TN$ one has $S=\nabla Tf$ since
$$S(X,Y) = \nabla_X Tf.Y - Tf.\nabla_X Y = (\nabla Tf)(X,Y).$$
The trace of $S$ is the \emph{vector valued mean curvature} $\Tr^g(S) \in \Ga\big(\Nor(f)\big)$. 

\chapter{Shape space}\label{sh}

Briefly said, in this work the word shape means an \emph{unparametrized surface}.
(The term surface is used regardless of whether it has dimension two or not.)  
This section is about the infinite dimensional space of all shapes. 
First an overview of the differential calculus that is used is presented. 
Then some spaces of parametrized and unparametrized surfaces are described, 
and it is shown how to define Riemannian metrics on them. 
The geodesic equation and conserved quantities arising from symmetries are derived.

The agenda that is set out in this section
will be pursued in section~\ref{so} when the arbitrary metric is replaced
by a Sobolev-type metric involving a pseudo-differential operator and later in 
section~\ref{la} when the pseudo-differential operator is replaced by
an operator involving powers of the Laplacian.

This section is common work with Martin Bauer and can also be found in section~1.2
of his Ph.D. thesis \cite{Bauer2010}.

\section{Convenient calculus}\label{sh:co}

The differential calculus used in this work is \emph{convenient calculus} \cite{MichorG}. 
The overview of convenient calculus presented here is taken from 
\cite[Appendix~A]{Michor109}. 
Convenient calculus is a generalization of differential calculus to spaces
beyond Banach and Fr\'echet spaces. 
Its most useful property for this work is
that the \emph{exponential law} holds without any restriction: 
$$C^\infty(E \x F, G) \cong C^\infty\big(E,C^\infty(F,G)\big) $$
for convenient vector spaces $E,F,G$ 
and a natural convenient vector space structure on $C^\infty(F,G)$.
As a consequence \emph{variational calculus} simply works: 
For example, a smooth curve in $C^\infty(M,N)$ can 
equivalently be seen as a smooth mapping $M \x \R \to N$. 
The main difficulty in the convenient setting is that the composition of 
linear mappings stops being jointly continuous at the level of Banach 
spaces for any compatible topology. 

Let $E$ be a \emph{locally convex vector space}. 
A curve $c:\R\to E$ is called 
\emph{smooth} or $C^\infty$ if all derivatives exist and are 
continuous - this is a concept without problems. Let 
$C^\infty(\R,E)$ be the space of smooth functions. It can be 
shown that $C^\infty(\R,E)$ does not depend on the locally convex 
topology of $E$, but only on its associated bornology (system of bounded 
sets).

$E$ is said to be a \emph{convenient 
vector space} if one of the following equivalent
conditions is satisfied (called $c^\infty$-completeness):
\begin{enumerate}
\item For any $c\in C^\infty(\R,E)$ the (Riemann-) integral 
       $\int_0^1c(t)dt$ exists in $E$.
\item A curve $c:\R\to E$ is smooth if and only if $\la\o c$ is 
       smooth for all $\la\in E'$, where $E'$ is the dual consisting 
       of all continuous linear functionals on $E$.
\item Any Mackey-Cauchy-sequence (i.\ e.\  $t_{nm}(x_n-x_m)\to 0$  
       for some $t_{nm}\to \infty$ in $\R$) converges in $E$. 
       This is visibly a weak completeness requirement.
\end{enumerate}
The final topology with respect to all smooth curves is called the 
$c^\infty$-topology on $E$, which then is denoted by $c^\infty E$. 
For Fr\'echet spaces it coincides with 
the given locally convex topology, but on the space $\mathcal D$ of test 
functions with compact support on $\R$ it is strictly finer.

Let $E$ and $F$ be locally 
convex vector spaces, and let $U\subset E$ be $c^\infty$-open. 
A mapping $f:U\to F$ is called {\it smooth} or 
$C^\infty$, if $f\o c\in C^\infty(\R,F)$ for all 
$c\in C^\infty(\R,U)$.
The notion of smooth mappings carries over to mappings between 
\emph{convenient manifolds}, which are manifolds modeled on 
$c^\infty$-open subsets of convenient vector spaces. 

\begin{thm*}
The main properties of smooth calculus are the following.
\begin{enumerate}
\item For mappings on Fr\'echet spaces this notion of smoothness 
       coincides with all other reasonable definitions. Even on 
       $\R^2$ this is non-trivial.
\item Multilinear mappings are smooth if and only if they are 
       bounded.
\item If $f:E\supseteq U\to F$ is smooth then the derivative 
       $df:U\x E\to F$ is  
       smooth, and also $df:U\to L(E,F)$ is smooth where $L(E,F)$ 
       denotes the space of all bounded linear mappings with the 
       topology of uniform convergence on bounded subsets.
\item The chain rule holds.
\item The space $C^\infty(U,F)$ is again a convenient vector space 
       where the structure is given by the obvious injection
$$
C^\infty(U,F)\to \prod_{c\in C^\infty(\R,U)} C^\infty(\R,F)
\to \prod_{c\in C^\infty(\R,U), \la\in F'} C^\infty(\R,\R).
$$
\item The exponential law holds:
$$
C^\infty(U,C^\infty(V,G)) \cong C^\infty(U\x V, G)
$$
     is a linear diffeomorphism of convenient vector spaces. Note 
     that this is the main assumption of variational calculus.
\item A linear mapping $f:E\to C^\infty(V,G)$ is smooth (bounded) if 
       and only if $E \East{f}{} C^\infty(V,G) \East{\on{ev}_v}{} G$ is smooth 
       for each $v\in V$. This is called the smooth uniform 
       boundedness theorem and it is quite applicable.
\end{enumerate}
\end{thm*}
Proofs of these statements can be found in \cite{MichorG}.

\section{Manifolds of immersions and embeddings}\label{sh:im}

What has sloppily been called a \emph{parametrized surface} will now be turned into a rigorous definition.
Mathematically, parametrized surfaces will be modeled as immersions or embeddings 
of one manifold into another.
Immersions and embeddings are called parametrized since a change in their parametrization 
(i.e. applying a diffeomorphism on the domain of the function) 
results in a different object. 
The following sets of functions will be important:
\begin{equation}\label{sh:im:eq1}
\Emb(M,N) \subset \Imm_f(M,N) \subset \Imm(M,N) \subset C^\infty(M,N).
\end{equation}
$C^\infty(M,N)$ is the set of smooth functions from $M$ to $N$. 
$\Imm(M,N)$ is the set of all \emph{immersions} of $M$ into $N$, i.e.
all functions $f \in C^\infty(M,N)$ such that 
$T_x f$ is injective for all $x \in M$. 
$\Imm_f(M,N)$ is the set of all \emph{free immersions}.
An immersion $f$ is called free if the diffeomorphism 
group of $M$ acts freely on it, i.e. $f \o \ph = f$ implies
$\ph=\Id_M$ for all $\ph \in \Diff(M)$. 
$\Emb(M,N)$ is the set of all \emph{embeddings} of $M$ into $N$, i.e. 
all immersions $f$ that are a homeomorphism onto their image. 

The following lemma from \cite[1.3 and 1.4]{Michor40} gives sufficient 
conditions for an immersion to be free. 
In particular it implies that every embedding is free. 
\begin{lem*}
If $\ph \in \Diff(M)$ has a fixed point and if $f \o \ph = f$ for some immersion $f$, then $\ph = \Id_M$.

If for an immersion $f$ there is a point $x \in f(M)$ with only one preimage then $f$ is free. 
\end{lem*}

Since $M$ is compact by assumption (see section~\ref{no:as}) it follows that $C^\infty(M,N)$ is a 
\emph{Fr\'echet manifold} \cite[section~42.3]{MichorG}. 
All inclusions in \eqref{sh:im:eq1} are inclusions of open subsets: 
$\Imm(M,N)$ is open in $C^\infty(M,N)$ since
the condition that the differential is injective at every point 
is an open condition on the one-jet of $f$ \cite[section~5.1]{MichorC}. 
$\Imm_f(M,N)$ is open in $\Imm(M,N)$ by \cite[theorem~1.5]{Michor40}.
$\Emb(M,N)$ is open in $\Imm_f(M,N)$ by \cite[theorem~44.1]{MichorG}. 
Therefore all function spaces in \eqref{sh:im:eq1} are Fr\'echet manifolds as well. 

When it is clear that $M$ and $N$ are the domain and target of the mappings, 
the abbreviations $\Emb, \Imm_f, \Imm$ will be used. 
In most cases, immersions will be used since this is the most general 
setting. Working with free immersions instead of immersions makes a difference in 
section \ref{sh:sh}, and working with embeddings instead of immersions makes a difference in 
section \ref{ge:ge}. The tangent and cotangent space to $\Imm$ are treated in the next section.

\section{Bundles of multilinear maps over immersions}\label{sh:na}

Consider the following \emph{natural bundles of $k$-multilinear mappings}: 
\begin{equation*}\xymatrix{
L^k(T\Imm;\R) \ar[d] & L^k(T\Imm;T\Imm ) \ar[d] \\
\Imm & \Imm
}\end{equation*}
These bundles are isomorphic to the bundles
\begin{equation*}\xymatrix{
L\left(\widehat\bigotimes^k T\Imm;\R\right)\ar[d] & 
L\left(\widehat\bigotimes^k T\Imm;T\Imm\right)\ar[d]\\
\Imm & \Imm
}\end{equation*}
where $\widehat\bigotimes$ denotes the $c^\infty$-completed bornological tensor product of 
locally convex vector spaces \cite[section~5.7, section~4.29]{MichorG}.
Note that $L(T\Imm;T\Imm)$ is not isomorphic to 
$T^*\Imm \;\widehat\otimes\; T\Imm$ 
since the latter bundle corresponds to multilinear mappings with finite rank.

It is worth to write down more explicitly what some of these bundles of multilinear mappings are.  
The \emph{tangent space to $\Imm$} is given by
\begin{align*}
T_f\Imm &= C^\infty_f(M,TN) := \big\{ h \in C^\infty(M,TN): \pi_N \o h =f\big\}, \\
T\Imm &= C^\infty_{\Imm}(M,TN) := \big\{ h \in C^\infty(M,TN): \pi_N \o h \in \Imm \big\}.
\end{align*}
Thus $T_f\Imm$ is the space of vector fields along the immersion $f$.
Now the \emph{cotangent space to $\Imm$} will be described.
The symbol $\widehat\otimes_{C^\infty(M)}$ means that the tensor product
is taken over the algebra $C^\infty(M)$.
\begin{align*}
T^*_f\Imm &= L(T_f\Imm;\R) = C^\infty_f(M,TN)' = 
C^\infty(M)'\; \widehat\otimes_{C^\infty(M)} C^\infty_f(M,T^*N) \\
T^*\Imm &= L(T\Imm;\R) = C^\infty(M)'\; \widehat\otimes_{C^\infty(M)} C^\infty_{\Imm}(M,T^*N) 
\end{align*}
The bundle $L^2_{\on{sym}}(T\Imm;\R)$ is of interest for the definition of a Riemannian 
metric on $\Imm$. 
(The subscripts $_\sym$ and $_\alt$ indicate symmetric and alternating multilinear maps, respectively.) 
Letting $\otimes_S$ denotes the symmetric tensor product and 
$\widehat\otimes_S$ the $c^\infty$-completed bornological symmetric tensor product, one has
\begin{align*}
L^2_{\on{sym}}(T_f\Imm;\R) &= (T_f\Imm\; \widehat\otimes_S\; T_f\Imm)' = 
\big(C^\infty_f(M,TN) \; \widehat\otimes_S\; C^\infty_f(M,TN)\big)' \\ &=
\big(C^\infty_f(M,TN \; \otimes_S\; TN) \big)' \\&=
C^\infty(M)' \;\widehat\otimes_{C^\infty(M)} C^\infty_f(M,T^*N \; \otimes_S\; T^*N) 
\\
L^2_{\on{sym}}(T\Imm;\R) &= 
C^\infty(M)' \;\widehat\otimes_{C^\infty(M)} C^\infty_{\Imm}(M,T^*N \;\otimes_S\; T^*N) 
\end{align*}

\section{Diffeomorphism group}\label{sh:di}

The following result is taken from \cite[section~43.1]{MichorG} 
with slight simplifications due to the compactness of $M$.

\begin{thm*}
For a smooth compact manifold $M$ the group $\on{Diff}(M)$ of all 
smooth diffeomorphisms of $M$ is an open submanifold 
of $C^\infty(M,M)$. 
Composition and inversion are smooth. 
The Lie algebra of the smooth infinite dimensional Lie group $\on{Diff}(M)$ 
is the convenient vector space $\X(M)$ of all smooth vector fields 
on $M$, equipped with the negative of the usual Lie bracket. 
$\on{Diff}(M)$ is a regular Lie group in the sense that the right evolution 
$$\on{evol}^r: C^\infty\big(\R,\X(M)\big) \to \Diff(M)$$
as defined in \cite[section~38.4]{MichorG}
exists and is smooth. 
The exponential mapping 
$$\on{exp}: \X(M) \to \Diff(M)$$ 
is the flow mapping to time $1$, and it is smooth.
\end{thm*}

The diffeomorphism group $\Diff(M)$ acts smoothly on $C^\infty(M,N)$ and its subspaces
$\Imm, \Imm_f$ and $\Emb$ by composition from the right. For $\Imm$, the action is given by the mapping
$$\Imm(M,N) \x \Diff(M) \to \Imm(M,N), \qquad (f,\ph) \mapsto r(f,\ph) = r^\ph(f)= f \o \ph.$$
The tangent prolongation of this group action is given by the mapping
$$T\Imm(M,N) \x \Diff(M) \to T\Imm(M,N), \qquad (h,\ph) \mapsto Tr^\ph(h) = h \o \ph.$$

\section{Riemannian metrics on immersions}\label{sh:ri}

A \emph{Riemannian metric $G$ on $\Imm$} is a section of the bundle
$$L^2_{\on{sym}}(T\Imm;\R)$$
such that at every $f \in \Imm$, $G_f$ is a symmetric positive definite bilinear mapping 
$$G_f: T_f\Imm \x T_f\Imm \to \R.$$
Each metric is {\it weak} in the sense that $G_f$, seen as a mapping
$$G_f: T_f\Imm \to T^*_f\Imm$$
is injective. (But it can never be surjective.)

\begin{ass*}
It will always be assumed that the metric $G$ is compatible with the action of $\Diff(M)$ on $\Imm(M,N)$
in the sense that the group action is given by isometries. 
\end{ass*}
This means that $G=(r^\ph)^* G$ for all $\ph \in \Diff(M)$, where 
$r^\ph$ denotes the right action of $\ph$ on $\Imm$ that was described in section~\ref{sh:di}. 
This condition can be spelled out in more details using the definition of $r^\ph$ as follows:
\begin{align*}
G_f(h,k)=\big((r^\ph)^* G\big)(h,k)
=G_{r^\ph(f)}\big(Tr^\ph(h),Tr^\ph(k)\big)
=G_{f \o \ph}(h \o \ph,k \o \ph). 
\end{align*}

\section{Covariant derivative $\nabla^{\g}$ on immersions}\label{sh:cov}

The covariant derivative $\nabla^{\g}$ defined in section~\ref{no:co} induces a
\emph{covariant derivative over immersions} as follows.
Let $Q$ be a smooth manifold. Then one identifies
\begin{align*}
&h \in  C^\infty\big(Q,T\Imm(M,N)\big) && \text{and} && X \in \X(Q)
\intertext{with}
&h^{\wedge} \in C^\infty(Q \x M, TN) && \text{and} && (X,0_M) \in \X(Q \x M).
\end{align*}
As described in section~\ref{no:co} one has the covariant derivative
$$\nabla^{\g}_{(X,0_M)} h^{\wedge} \in C^\infty\big(Q \x M, TN).$$
Thus one can define
$$\nabla_X h = \left(\nabla^{\g}_{(X,0_M)} h^{\wedge}\right)^{\vee} \in C^\infty\big(Q,T\Imm(M,N)\big).$$
This covariant derivative is torsion-free by section~\ref{no:sw}, formula~\eqref{no:sw:to}. 
It respects the metric $\g$ but in general does not respect $G$. 

It is helpful to point out some special cases of how this construction can be used. 
The case $Q=\R$ will be important to formulate the geodesic equation. 
The expression that will be of interest in the formulation of the 
geodesic equation is $\nabla_{\p_t} f_t$, which is 
well-defined when $f:\R \to \Imm$ is a path of immersions and $f_t: \R \to T\Imm$ is its velocity. 

Another case of interest is $Q = \Imm$. Let $h, k, m \in \X(\Imm)$. Then the covariant 
derivative $\nabla_m h$ is well-defined and tensorial in $m$.
Requiring $\nabla_m$ to respect the grading of the spaces of multilinear maps, to act as a derivation 
on products and to commute with compositions of multilinear maps, one obtains 
as in section~\ref{no:co} a covariant 
derivative $\nabla_m$ acting on all mappings into 
the natural bundles of multilinear mappings over $\Imm$.
In particular, $\nabla_m P$ and $\nabla_m G$ are well-defined for 
\begin{align*}
P \in \Ga\big(L(T\Imm;T\Imm)\big), \quad
G \in \Ga\big(L^2_{\on{sym}}(T\Imm;\R)\big)
\end{align*}
by the usual formulas
\begin{align*}
(\nabla_m P)(h) &= \nabla_m\big(P(h)\big) - P(\nabla_mh), \\
(\nabla_m G)(h,k) &= \nabla_m\big(G(h,k)) - G(\nabla_m h,k) - G(h,\nabla_m k). 
\end{align*}

\section{Metric gradients}\label{sh:me}

The \emph{metric gradients} $H,K \in \Ga\big(L^2(T\Imm;T\Imm)\big)$ are uniquely defined by the equation
$$(\nabla_m G)(h,k)=G\big(K(h,m),k\big)=G\big(m,H(h,k)\big),$$
where $h,k,m$ are vector fields on $\Imm$ and 
the covariant derivative of the metric tensor $G$ is defined as in the previous section. 
(This is a generalization of the definition used in \cite{Michor107} 
that allows for a curved ambient space $N \neq \R^n$.)

Existence of $H, K$ has to proven case by case for each metric $G$, 
usually by partial integration. 
For Sobolev metrics, this will be proven in sections~\ref{la:ad} and \ref{la:ge}.

\begin{ass*}
Nevertheless it will be assumed for now that the metric gradients $H,K$ exist.
\end{ass*}

\section{Geodesic equation on immersions}\label{sh:ge}

\begin{thm*}
Given $H,K$ as defined in the previous section and $\nabla$ as defined in 
section~\ref{sh:cov}, the geodesic equation reads as
$$\nabla_{\p_t} f_t=\frac12 H_f(f_t,f_t)-K_f(f_t,f_t).$$
\end{thm*}
This is the same result as in \cite[section~2.4]{Michor107}, but in a more general setting.
\begin{proof}
Let $f: (-\ep,\ep) \x [0,1] \x M \to N$ be a one-parameter 
family of curves of immersions with fixed endpoints. 
The variational parameter will be denoted by $s \in (-\ep,\ep)$ and 
the time-parameter by $t \in [0,1]$. 
In the following calculation, let $G_f$ denote $G$ composed with $f$, i.e.
$$G_f: \R \to \Imm \to L^2_{\sym}(T\Imm;\R).$$ 
Remember that the covariant derivative on $\Imm$ that has been introduced in 
section~\ref{sh:cov} is torsion-free so that one has
$$\nabla_{\p_t}f_s - \nabla_{\p_s}f_t=Tf.[\p_t,\p_s]+\on{Tor}(f_t,f_s) = 0.$$
Thus the first variation of the energy of the curves is
\begin{align*}
\p_s \frac12 \int_0^1 G_f(f_t, f_t) dt &= 
\frac12 \int_0^1 (\nabla_{\p_s} G_f)(f_t, f_t)
+ \int_0^1 G_f(\nabla_{\p_s} f_t, f_t) dt 
\\&= 
\frac12 \int_0^1 (\nabla_{f_s} G)(f_t, f_t) 
+ \int_0^1 G_f(\nabla_{\p_t} f_s, f_t) dt 
\\&= 
\frac12 \int_0^1 (\nabla_{f_s} G)(f_t, f_t) dt
+ \int_0^1 \p_t\ G_f(f_s,f_t) dt \\&\qquad
- \int_0^1 (\nabla_{f_t} G)(f_s, f_t) dt 
- \int_0^1 G_f(f_s,\nabla_{\p_t} f_t) dt 
\\&=
\int_0^1 G\Big(f_s,\frac12 H(f_t,f_t)+0-K(f_t,f_t)-\nabla_{\p_t} f_t\Big) dt.
\end{align*}
If $f(0,\cdot,\cdot)$ is energy-minimizing, then one has at $s=0$ that
\begin{equation*}
\frac12 H(f_t,f_t)-K(f_t,f_t)-\nabla_{\p_t} f_t =0. \qedhere
\end{equation*}
\end{proof}

\section[Geodesic equation on immersions]%
{Geodesic equation on immersions in terms of the momentum}\label{sh:gemo}

In the previous section the geodesic equation for the velocity $f_t$ has been derived. 
In many applications it is more convenient to formulate the geodesic equation as an equation 
for the momentum $G(f_t,\cdot) \in T^*_f\Imm$. 
$G(f_t,\cdot)$ is an element of the \emph{smooth 
cotangent bundle}, also called \emph{smooth dual}, which is given by 
$$G(T\Imm) := \coprod_{f \in \Imm} G_f(T_f\Imm) = 
\coprod_{f \in \Imm} \{ G_f(h,\cdot): h \in T_f\Imm\} \subset T^*\Imm. $$
It is strictly smaller than $T^*\Imm$ since at every $f \in \Imm$
the metric $G_f: T_f\Imm \to T^*_f\Imm$ is injective but not surjective. It is called 
smooth since it does not contain distributional sections of $f^*TN$, whereas $T_f^*\Imm$ does. 

\begin{thm*}
The geodesic equation for the momentum $p \in T^*\Imm$ is given by
\begin{equation*}
\left\{\begin{aligned}
p &= G(f_t, \cdot) \\
\nabla_{\p_t} p &= \frac12 G_f\big( H(f_t,f_t),\cdot\big),
\end{aligned}\right.
\end{equation*}
where $H$ is the metric gradient defined in section~\ref{sh:me}
and $\nabla$ is the covariant derivative action on mappings into $T^*\Imm$ as 
defined in section~\ref{sh:cov}.
\end{thm*}

\begin{proof}
Let $G_f$ denote $G$ composed with the path $f:\R\to\Imm$, i.e.
$$G_f: \R \to \Imm \to L^2_{\sym}(T\Imm;\R).$$
Then one has
\begin{align*}
\nabla_{\p_t} p &= \nabla_{\p_t} \big(G_f(f_t,\cdot)\big) =
(\nabla_{\p_t}G_f)(f_t,\cdot) + G_f(\nabla_{\p_t} f_t,\cdot) \\&=
(\nabla_{f_t}G)(f_t,\cdot) + G_f\Big(\frac12 H(f_t,f_t) - K(f_t,f_t),\cdot\Big) \\&=
G_f\big(K(f_t,f_t),\cdot\big)+ G_f\Big(\frac12 H(f_t,f_t) - K(f_t,f_t),\cdot\Big) \qedhere
\end{align*}
\end{proof}
This equation is equivalent to \emph{Hamilton's equation} restricted to the smooth cotangent bundle:
\begin{equation*}
\left\{\begin{aligned}
p &= G(f_t, \cdot) \\
p_t &= (\grad^{\om} E)(p).
\end{aligned}\right.
\end{equation*} 
Here $\om$ denotes the restriction of the canonical symplectic form on $T^*\Imm$ to the smooth 
cotangent bundle and $E$ is the Hamiltonian
$$E: G(T\Imm) \to \R, \quad E(p) = G\i(p,p)$$
which is only defined on the smooth cotangent bundle. 

\section{Conserved momenta}\label{sh:mo}

This section describes how a group acting on $\Imm$ by isometries defines 
a \emph{momentum mapping} that is conserved along geodesics in $\Imm$.
It is similar to \cite[section~4]{Michor118}. 
A more detailed treatment and proofs can be found in \cite{Michor107}.

Consider an infinite dimensional regular Lie group with
Lie algebra $\mathfrak g$ and a right action $g\mapsto r^g$ of this group on $\on{Imm}$.
Let $\Imm$ be endowed with a Riemannian metric $G$. 
The basic assumption (assumption~\ref{sh:ri}) is that the action is by isometries: 
$$ G=(r^g)^*G, \quad \text{i.e.} \quad G_f(h,k)=G_{r^g(f)}\big(T_f(r^g)h,T_f(r^g)k\big). $$
Denote by $\X(\on{Imm})$ the set of vector fields on $\on{Imm}$. Then
the group action can be specified by the fundamental vector field mapping $\ze:\mathfrak g\to
\X(\on{Imm})$, which will be a bounded Lie algebra homomorphism.
The fundamental vector field $\ze_X, X \in \mathfrak g$ is
the infinitesimal action in the sense:
$$\ze_X(f)=\p_t|_0 r^{\exp(tX)}(f).$$
The key to the Hamiltonian approach is to write the infinitesimal action 
as a Hamiltonian vector field, i.e. as the $\om$-gradient of some function.
This function will be called the \emph{momentum map}. 
$\om$ is a two-form on $T\Imm$, 
$$\om \in \Ga\big(L^2_{\alt}(TT\Imm;\R)\big)$$ 
that is obtained as the pullback of the canonical symplectic form on $T^*\Imm$ 
via the metric
$$G: T\Imm \to T^*\Imm.$$
The $\om$-gradient is defined by the relation
$$\grad^\om f \in \X(T\Imm), \qquad \om(\on{grad}^{\om} f, \cdot)=df, $$
where $f$ is a smooth function on $T\Imm$. 
Not all functions have an $\om$-gradient because 
$$\om:TT\Imm \to T^*T\Imm$$
is injective, but not surjective. 
The set of functions that have a smooth $\om$-gradient are denoted by 
$$C^\infty_{\om}(T\Imm,\R) \subset C^\infty(T\Imm,\R).$$
The momentum map is defined as 
$$j:\mathfrak g\to C^\infty(T\Imm,\R), \quad j_X(h_f) = G_f\big(\ze_X(f),h_f\big)$$
and it is verified that it has the desired properties: 
Assuming that the metric gradients $H,K$ exist (assumption~\ref{sh:me}), 
it can be proven that 
$$ j_X \in C^\infty_{\om}(T\Imm,\R) \quad \text{and} \quad \grad^\om(j_X) = \ze_X. $$
Thus the momentum map fits into the following commutative diagram of Lie algebras:
\begin{displaymath}
\xymatrix{
H^0(T\on{Imm}) \ar[r]^{i} & 
C^\infty_{\om}(T\Imm,\R)  \ar[r]^{\on{grad}^{\om}}  &
\X(T\Imm,\om) \ar[r]^{\om} & 
H^1(T\on{Imm})
\\
& & \mathfrak g \ar[lu]^{j} \ar[u]_{\ze^{T\Imm}} &
}\end{displaymath}
Here $\X(T\Imm,\om)$ is the space of vector fields on $T\Imm$
whose flow leaves $\om$ fixed. 
All arrows in this diagram are homomorphism of Lie algebras. 
The sequence at the top is exact when it is extended by zeros on the left and right end.  

By \emph{Emmy Noether's theorem}, the momentum mapping is constant 
along any geodesic $f: \R \to \Imm$. 
Thus for any $X\in\mathfrak g$ one has that
\begin{equation*}
j_X(f_t) = G_f\big(\ze_X(f),f_t\big) 
\quad\text{  is constant in }t.
\end{equation*}
Now several group actions on $\Imm$ will be considered, 
and the corresponding conserved momenta will be calculated. 
\begin{itemize}
\item
Consider the smooth right action of the group
$\on{Diff}(M)$ on $\on{Imm}(M,N)$ given by composition from the right:
$$f\mapsto f\o \ph \quad \text{for} \quad \ph\in\on{Diff}(M).$$
This action is isometric by assumption, see section~\ref{sh:ri}. 
For $X\in\X(M)$ the fundamental vector field is given by
$$ \ze_X(f) = \p_t|_0 (f\o \on{Fl}^X_t) = Tf \o X $$
where $\on{Fl}^X_t$ denotes the flow of $X$.
The {\it reparametrization momentum}, for any vector field $X$ on $M$ is thus
$G_f(Tf \o X, h_f)$.
Assuming that the metric is reparametrization invariant, it follows that along any
geodesic $f(t,\cdot)$, the expression $G_f(Tf \o X,f_t)$ is constant for all $X$.
\end{itemize}
For a flat ambient space $N=\R^n$ the following group actions can be consider in addition:
\begin{itemize}
\item
The left action of the Euclidean motion group
$\mathbb R^n\rtimes SO(n)$ on $\on{Imm}(M,\mathbb R^n)$ given by
$$f\mapsto A+Bf \quad \text{for} \quad (A,B)\in \R^n\x SO(n).$$
The fundamental vector field mapping is 
$$\ze_{(A,X)}(f)= A+Xf \quad \text{for} \quad (A,X) \in \R^n \x \mathfrak s\mathfrak o(n).$$
The {\it linear momentum} is thus $G_f(A,h), A \in \R^n$ and if the
metric is trans\-la\-tion invariant, $G_f(A,f_t)$ will be constant along
geodesics for every $A\in \mathbb R^n$.
The {\it angular momentum} is similarly
$G_f(X.f,h), X\in \mathfrak s\mathfrak o(n)$ and if the
metric is rotation invariant, then $G_f(X.f,f_t)$ will be constant along
geodesics for each $X\in \mathfrak s\mathfrak o(n)$.

\item
The action of the scaling group of $\mathbb R$ given by $f\mapsto e^r f$,
with fundamental vector field $\ze_a(f)=a.f$.
If the metric is scale invariant, then
the {\it scaling momentum} $G_f(f,f_t)$ will be constant along
geodesics.
\end{itemize}

\section{Shape space}\label{sh:sh}

$\Diff(M)$ acts smoothly on $C^\infty(M,N)$ and its subsets $\Imm,\Imm_f$ and $\Emb$ 
by composition from the right. 
\emph{Shape space} is defined as the orbit space with respect to this action. 
That means that in shape space, two mappings that differ only in their parametrization
will be regarded the same. 

\begin{thm*}
Let $M$ be compact and of dimension $\leq n$.
Then $\Imm_f(M,N)$ is the total space of a smooth principal fiber bundle 
with structure group $\Diff(M)$, whose base manifold 
is a Hausdorff smooth Fr\'echet manifold denoted by
$$B_{i,f}(M,N) = \Imm_f(M,N)/\Diff(M).$$
The same result holds for the open subset $\Emb(M,N) \subset \Imm_f(M,N)$. 
The corresponding base space is denoted by
$$B_{e}(M,N) = \Emb(M,N)/\Diff(M).$$
However, the space
$$B_i(M,N) = \Imm(M,N)/\Diff(M)$$
is not a smooth manifold, but has singularities of orbifold type: 
Locally, it looks like a finite dimensional orbifold times an infinite dimensional 
Fr\'echet space. 
\end{thm*}

The proofs for free and non-free immersions can be found in \cite{Michor40} and 
the one for embeddings in \cite[section~44.1]{MichorG}. 

As with immersions and embeddings, the notation $B_{i,f}, B_i, B_e$ will be used
when it is clear that $M$ and $N$ are the domain and target of the mappings.

\section{Riemannian submersions and geodesics}\label{sh:sub}

The way to induce a Riemannian metric on shape space is to use the 
concept of a \emph{Riemannian submersion}. 
This section explains in general terms what a Riemannian submersion is
and how horizontal geodesics in the top space correspond nicely to geodesics in the quotient space. 
The definitions and results of this section are taken from \cite[section~26]{MichorH}.

Let $\pi:E\rightarrow B$ be a submersion of smooth manifolds, 
that is, $T\pi:TE \rightarrow TB$ is surjective. Then
$$V=V(\pi):=\on{ker}(T\pi) \subset TE$$
is called the {\it vertical subbundle}. 
If $E$ carries a Riemannian metric $G$, 
then one can go on to define the {\it horizontal subbundle}
as the $G$-orthogonal complement of $V$: 
$$\Hor=\Hor(\pi,G):=V(\pi)^\bot \subset TE.$$
Now any vector $X \in TE$ can be decomposed uniquely in vertical and horizontal components as
$$X=X^{\on{ver}}+X^{\hor}.$$
This definition extends to the cotangent bundle as follows: 
An element of $T^*E$ is called horizontal when it annihilates all vertical vectors, 
and vertical when it annihilates all horizontal vectors. 

In the setting described so far, the mapping 
\begin{equation*}\label{sh:sub:eq1}
T_x \pi|_{\Hor_x}:\Hor_x\rightarrow T_{\pi(x)}B
\end{equation*}
is an isomorphism of vector spaces for all $x\in E$. 
If both $(E,G_E)$ and $(B,G_B)$ are Riemannian manifolds and
if this mapping is an isometry for all $x\in E$, 
then $\pi$ will be called a {\it Riemannian submersion}.

\begin{thm*}
Consider a Riemannian submersion $\pi:E\rightarrow B,$ 
and let $c:[0,1]\rightarrow E$ be a geodesic in $E$.
\begin{enumerate}
\item If $c'(t)$ is horizontal at one $t$, then it is horizontal at all $t$. 
\item If $c'(t)$ is horizontal then $\pi \circ c$ is a geodesic in $B$.
\item If every curve in $B$ can be lifted to a horizontal curve in $E$, 
then there is a one-to-one correspondence between curves in $B$ and horizontal curves in $E$. 
This implies that instead of solving the geodesic equation on $B$ one can equivalently solve
the equation for horizontal geodesics in $E$.
\end{enumerate}
\end{thm*}
See \cite[section~26]{MichorH} for the proof.

\section{Riemannian metrics on shape space}\label{sh:rish}

Now the previous chapter is applied to the submersion $\pi:\Imm \to B_i$:
\begin{thm*}
Given a  $\Diff(M)$-invariant Riemannian metric on $\Imm$, 
there is a unique Riemannian metric on the quotient space $B_i$ such that the
quotient map $\pi:\Imm \to B_i$ is a Riemannian submersion. 
\end{thm*}

One also gets a description of the tangent space to shape space: When $f \in \Imm$, then
$T_{\pi(f)}B_i$ is isometric to the horizontal bundle at $f$. The horizontal bundle 
depends on the definition of the metric. For the $H^0$-metric, it consists of vector 
fields along $f$ that are everywhere normal to $f$, see section~\ref{so:ho}.

\begin{ass*}
It will always be assumed that a $\Diff(M)$-invariant Riemannian
metric on the manifold of immersions is given, and that shape space 
is endowed with the unique Riemannian
metric turning the projection into a Riemannian submersion. 
\end{ass*}

\section{Geodesic equation on shape space}\label{sh:gesh}

Theorem~\ref{sh:sub} applied to the Riemannian submersion $\pi: \Imm \to B_i$ yields: 
\begin{thm*}
Assuming that every curve in $B_i$ can be lifted to a horizontal curve in $\Imm$, 
the geodesic equation on shape space is equivalent to
\begin{equation}\label{sh:gesh:eq1}
\left\{\begin{aligned}
f_t&=f_t^{\hor}\in \Hor \\
(\nabla_{\p_t}f_t)^{\hor} &= \Big(\frac12 H(f_t,f_t)-K(f_t,f_t)\Big)^{\hor},
\end{aligned}\right.
\end{equation}
where $f$ is a horizontal curve in $\Imm$, where $H,K$ are the metric gradients
defined in section~\ref{sh:me} and where $\nabla$ is the covariant derivative defined
in section~\ref{sh:cov}.
\end{thm*}

\begin{proof}
Theorem~\ref{sh:sub} states that the geodesic equation on shape space is equivalent 
to the horizontal geodesic equation on $\Imm$ which is given by
\begin{equation}\label{sh:gesh:eq2}
\left\{\begin{aligned}
f_t &= f_t^{\hor}\\
\nabla_{\p_t} f_t&=\frac12 H_f(f_t,f_t)-K_f(f_t,f_t)
\end{aligned}\right.
\end{equation}
Clearly \eqref{sh:gesh:eq2} implies \eqref{sh:gesh:eq1}. 
To prove the converse it remains to show that 
$$(\nabla_{\p_t}f_t)^{\vert} = \Big(\frac12 H(f_t,f_t)-K(f_t,f_t)\Big)^{\vert}.$$
As the following proof shows, this is a consequence of the 
conservation of the momentum along $f$ and of 
the invariance of the metric under $\Diff(M)$. 

Recall the infinitesimal action of $\Diff(M)$ on $\Imm(M,N)$. 
For any $X \in \X(M)$ it is given by the fundamental vector field
$$\ze_X \in \X(\Imm), \qquad 
\ze_X(f) = \p_s|_0 r\big(f,\exp(sX)\big) = \p_s|_0 \big(f \o Fl_t^X\big) = Tf.X.$$
Here $r$ is the right action of $\Diff(M)$ on $\Imm(M,N)$ defined in section~\ref{sh:di}. 
When $f: \R \to \Imm$ is a curve of immersions, one obtains a two-parameter family of immersions
$$g: \R \x \R \to \Imm, \qquad g(s,t) = r\big(f(t),\exp(sX)\big) $$
that satisfies
\begin{align*}
\nabla_{\p_t} Tg.\p_s &= 
\nabla_{\p_s} Tg.\p_t + Tg.[\p_t,\p_s] + \on{Tor}(Tg.\p_t,Tg.\p_s) \\&= 
\nabla_{\p_s} T\big(r^{\exp(sX)}\big)f_t +0+0
\end{align*}
since $\nabla$ is torsion-free.
This implies 
\begin{align*}
\nabla_{\p_t} \ze_X(f) = \nabla_{\p_t} Tg.\p_s|_0 = \nabla_{\p_s|_0} T\big(r^{\exp(sX)}\big)f_t.
\end{align*}
$\ze_X(f)$ is vertical and $f_t$ is horizontal by assumption. 
Thus the momentum mapping $G_f\big(\ze_X(f),f_t\big)$ is constant and equals zero.
Its derivative is
\begin{align*}
0 &= \p_t \Big(G_f\big(\ze_X(f),f_t\big)\Big) \\&=  
(\nabla_{\p_t} G_f)\big(\ze_X(f),f_t\big) 
+ G_f\big(\nabla_{\p_t}\ze_X(f), f_t\big) 
+ G_f\big(\ze_X(f), \nabla_{\p_t} f_t\big)
\\&=
(\nabla_{f_t} G)\big(\ze_X(f),f_t\big) 
+ G_f\Big(\nabla_{\p_s|_0} T\big(r^{\exp(sX)}\big)f_t, f_t\Big)
+ G_f\big(\ze_X(f), \nabla_{\p_t} f_t\big)
\\&=
G_f\big(K_f(f_t,f_t) + \nabla_{\p_t} f_t,\ze_X(f)\big) \\&\qquad
+ \left.G_{r^{\exp(sX)}f}\Big(\nabla_{\p_s} T\big(r^{\exp(sX)}\big)f_t, 
T\big(r^{\exp(sX)}\big)f_t\Big)\right|_{s=0}
\\&=
G_f\big(K_f(f_t,f_t) + \nabla_{\p_t} f_t,\ze_X(f)\big) \\&\qquad
+ \frac12 \p_s|_0 \bigg(G_{r^{\exp(s.X)}f}\Big(T\big(r^{\exp(sX)}\big)f_t, 
T\big(r^{\exp(s.X)}\big)f_t\Big) \bigg)
\\&\qquad
- \frac12 \left. \big(\nabla_{\p_s} G_{r^{\exp(sX)}f}\big)
\Big( T\big(r^{\exp(sX)}\big)f_t, T\big(r^{\exp(sX)}\big)f_t\Big) \right|_{s=0}
\\&=
G_f\big(K_f(f_t,f_t) + \nabla_{\p_t} f_t,\ze_X(f)\big) \\&\qquad
+ \frac12 \p_s|_0 \big(G_f(f_t, f_t) \big)
- \frac12 \big(\nabla_{\ze_X(f)} G\big)(f_t, f_t)
\\&=
G_f\Big(K_f(f_t,f_t) + \nabla_{\p_t} f_t + 0 - \frac12 H_f(f_t,f_t),\ze_X(f)\Big) 
\end{align*}
Any vertical tangent vector to $f$ is of the form $\ze_X(f)$ for some $X \in \X(M)$. 
Therefore
\begin{equation*}
0=\Big(\nabla_{\p_t} f_t - \frac12 H_f(f_t,f_t) + K_f(f_t,f_t) \Big)^{\vert}. \qedhere
\end{equation*}
\end{proof}

It will be shown in section~\ref{so:ho2} that curves in $B_i$ 
can be lifted to horizontal curves in $\Imm$ for the very general class of 
Sobolev type metrics. Thus all assumptions and conclusions of the theorem hold. 

\section[Geodesic equation on shape space]%
{Geodesic equation on shape space in terms of the momentum}\label{sh:geshmo}

As in the previous section, theorem~\ref{sh:sub} will be applied to the Riemannian submersion 
$\pi: \Imm \to B_i$. 
But this time, the formulation of the geodesic equation in terms of the momentum will be used, 
see section~\ref{sh:gemo}.
As will be seen in section~\ref{so:geshmo}, this is the most convenient formulation
of the geodesic equation for Sobolev-type metrics.
\begin{thm*}
Assuming that every curve in $B_i$ can be lifted to a horizontal curve in $\Imm$, 
the geodesic equation on shape space is equivalent to the set of equations
\begin{equation*}
\left\{\begin{aligned}
p &= G_f(f_t,\cdot) \in \Hor \subset T^*\Imm, \\
(\nabla_{\p_t}p)^{\hor} &= \frac12 G_f\big(H(f_t,f_t),\cdot)^{\hor}.
\end{aligned}\right.
\end{equation*}
Here $f$ is a curve in $\Imm$, 
$H$ is the metric gradient defined in section~\ref{sh:me},
and $\nabla$ is the covariant derivative defined in section~\ref{sh:cov}.
$f$ is horizontal because $p$ is horizontal. 
\end{thm*}

\section{Inner versus outer metrics}\label{sh:inout}

There are two similar yet different approaches 
on how to define a Riemannian metric on shape space. 

The metrics on shape space presented in this work are induced by metrics on $\Imm(M,N)$. 
One might call them \emph{inner metrics} since they are defined intrinsically to $M$. 
Intuitively, these metrics can be seen as 
describing a deformable material that the shape itself is made of. 

In contrast to these metrics, there are also metrics that are induced from metrics on $\on{Diff}(N)$
by the same construction of Riemannian submersions. 
(The widely used LDDMM algorithm is based on such a metric.)
The differential operator governing these metrics is defined on all of $N$, even outside of the shape. 
When the shape is deformed, the surrounding ambient space is deformed with it. 
Intuitively, such metrics can be seen as describing some deformable material that the ambient 
space is made of. Therefore one might call them \emph{outer metrics}. 

The following diagram illustrates both approaches. Metrics are defined on one of the top 
spaces and induced on the corresponding space below by the construction of Riemannian submersions. 
\begin{equation*}\xymatrix{
\Diff(N) \ar[d] & \\
\Emb(M, N) \ar[d] \ar@{^{(}->}[r] & \Imm(M,N) \ar[d] \\
B_e(M,N) \ar@{^{(}->}[r] & B_i(M,N)
}\end{equation*}

\chapter{Variational formulas}\label{va}

Recall that many operators like
$$g=f^*\g, \quad S=S^f, \quad \vol(g), \quad \nabla=\nabla^g, \quad \Delta=\Delta^g, \quad \ldots$$ 
implicitly depend on the immersion $f$. In this section their derivative 
with respect to $f$ which is called their \emph{first variation} will be calculated . 
These formulas will be used to calculate the metric gradients that are needed 
for the geodesic equation. 

This section is based on \cite{Michor118, Michor119}. 
Some of the formulas can be found in \cite{Besse2008, Michor102, Verpoort2008}.
The presentation is similar to \cite{Bauer2010}, and some of the variational formulas 
are the same.

\section{Paths of immersions}\label{va:pa}

All of the differential-geometric concepts introduced in section \ref{no}
can be recast for a path of immersions instead of a fixed immersion. 
This allows to study variations of immersions. 
So let $f:\R \to \on{Imm}(M,N)$ be a path of immersions. By convenient calculus
\cite{MichorG}, $f$ can equivalently be seen as $f:\R \x M \to N$ 
such that $f(t,\cdot)$ is an immersion for each $t$. 
The bundles over $M$ can be replaced by bundles over $\R \x M$:
\begin{equation*}\xymatrix{
\on{pr}_2^* T^r_s M \ar[d] & 
\on{pr}_2^* T^r_s M \otimes f^*TN \ar[d] &
\Nor(f) \ar[d]\\
\R \x M & \R \x M & \R \x M
}\end{equation*}
Here $\on{pr}_2$ denotes the projection $\on{pr}_2:\R \x M \to M$.
The covariant derivative $\nabla_Z h$ is now defined for vector fields $Z$ on $\R \x M$ 
and sections $h$ of the above bundles. 
The vector fields $(\p_t, 0_M)$ and $(0_{\R}, X)$, where $X$ is a vector field on $M$, are of
special importance. 
In later sections they will be identified with $\p_t$ and $X$ 
whenever this does not pose any problems. 
Let
$$\on{ins}_t : M \to \R \x M, \qquad x \mapsto (t,x) .$$
Then by property~\ref{no:co:prop5} from section~\ref{no:co} one has for vector fields $X,Y$ on $M$
\begin{align*}
\nabla_X Tf(t,\cdot).Y &= \nabla_X T(f \o \on{ins}_t) \o Y = \nabla_X Tf \o T\on{ins}_t \o Y
\\&= \nabla_X Tf \o (0_\R,Y) \o \on{ins}_t 
= \nabla_{T\on{ins}_t \o X} Tf \o (0_\R,Y)\\&
 = \big(\nabla_{(0_\R,X)} Tf \o (0_\R,Y)\big) \o \on{ins}_t .
\end{align*}
This shows that one can recover the static situation at $t$ by using vector fields on $\R \x M$ 
with vanishing $\R$-component and evaluating at $t$.

\section{Directional derivatives of functions}

The following ways to denote directional derivatives of functions will be used, in particular in 
infinite dimensions.
Given a function $F(x,y)$ for instance,
$$ D_{(x,h)}F \text{ will be written as a shorthand for } \partial_t|_0 F(x+th,y).$$
Here $(x,h)$ in the subscript denotes the tangent vector with foot point $x$ and direction $h$. 
If $F$ takes values in some linear space, this linear space and its tangent space will be identified. 

\section{Setting for first variations}\label{va:se}

In all of this chapter, let $f$ be an immersion and $f_t \in T_f\Imm$ a tangent vector to $f$. 
The reason for calling the tangent vector $f_t$ is that in calculations  
it will often be the derivative of a curve of immersions through $f$. 
Using the same symbol $f$ for the fixed immersion
and for the path of immersions through it, one has in fact that
$$D_{(f,f_t)} F = \p_t F(f(t)).$$

\section{Variation of equivariant tensor fields}\label{va:ta}

Let the smooth mapping $F:\Imm(M,N) \to \Gamma(T^r_s M)$ 
take values in some space of tensor fields over $M$, 
or more generally in any natural bundle over $M$, see \cite{MichorF}.

\begin{lem*}
If $F$ is equivariant 
with respect to pullbacks by 
diffeomorphisms of $M$, i.e. 
$$F(f)=(\ph^* F)(f)=\ph^* \Big(F\big((\ph\i)^*f\big)\Big) $$ 
for all $\ph \in \on{Diff}(M)$ and $f \in \Imm(M,N)$,
then the tangential variation of $F$ is its Lie-derivative:
\begin{align*}
D_{(f,Tf.f_t^\top)} F&=
\p_t|_0 F\Big(f \o Fl^{f_t^\top}_t\Big)=
\p_t|_0 F\Big((Fl^{f_t^\top}_t)^* f\Big)\\&=
\p_t|_0 \Big(Fl_t^{f_t^\top}\Big)^* \big(F(f)\big) = \L_{f_t^\top}\big(F(f)\big).
\end{align*}
\end{lem*}

This allows us to calculate the tangential variation of the pullback metric and 
the volume density, for example.

\section{Variation of the metric}\label{va:me}

\begin{lem*}
The differential of the pullback metric
\begin{equation*}\left\{ \begin{array}{ccl}
\Imm &\to &\Gamma(S^2_{>0} T^*M),\\
f &\mapsto &g=f^*\g
\end{array}\right.\end{equation*}
is given by
\begin{align*}
D_{(f,f_t)} g&= 2\on{Sym}\g(\nabla f_t,Tf) = -2 \g(f_t^\bot,S)+2 \on{Sym} \nabla (f_t^\top)^\flat 
\\& = -2 \g(f_t^\bot,S)+ \L_{f_t^\top} g.
\end{align*}
\end{lem*}
Here $\on{Sym}$ denotes the symmetric part of the tensor field $C$ of type $\left(\begin{smallmatrix}0\\2\end{smallmatrix}\right)$  given by
$$\big(\on{Sym}(C)\big)(X,Y):=\frac12\big(C(X,Y)+C(Y,X)\big).$$

\begin{proof}
Let $f:\R \x M \to N$ be a path of immersions. Swapping covariant derivatives as in 
section~\ref{no:sw}, formula \eqref{no:sw:to2} one gets
\begin{align*}
\p_t\big(g(X,Y)\big) &= \p_t\big( \g( Tf.X,Tf.Y ) \big)
= \g( \nabla_{\p_t}Tf.X,Tf.Y ) + \g( Tf.X, \nabla_{\p_t}Tf.Y )\\
&=\g( \nabla_X f_t,Tf.Y ) + \g( Tf.X, \nabla_Y f_t ) = \big(2 \on{Sym}\g(\nabla f_t,Tf)\big)(X,Y).
\end{align*}
Splitting $f_t$ into its normal and tangential part yields
\begin{align*}
2 \on{Sym}\g(\nabla f_t,Tf) &=
2 \on{Sym}\g(\nabla f_t^\bot + \nabla Tf.f_t^\top,Tf) \\&=
-2 \on{Sym}\g(f_t^\bot,\nabla Tf)+2 \on{Sym} g(\nabla f_t^\top,\cdot) \\&=
-2 \g(f_t^\bot,S)+2 \on{Sym} \nabla (f_t^\top)^\flat .
\end{align*}
Finally the relation
$$D_{(f,Tf.f_t^\top)} g = 2 \on{Sym} \nabla (f_t^\top)^\flat = \L_{f_t^\top} g $$
follows either from the equivariance of $g$ 
with respect to pullbacks by diffeomorphisms (see section~\ref{va:ta}) or directly from
\begin{align*}
(\L_Xg)(Y,Z)&=
\L_X\big(g(Y,Z)\big)-g(\L_XY,Z)-g(Y,\L_XZ)\\&=
\nabla_X\big(g(Y,Z)\big)-g(\nabla_XY-\nabla_YX,Z)-g(Y,\nabla_XZ-\nabla_ZX)\\&=
g(\nabla_YX,Z)+g(Y,\nabla_ZX)=
(\nabla_YX)^\flat(Z)+(\nabla_ZX)^\flat(Y)\\&=
(\nabla_YX^\flat)(Z)+(\nabla_ZX^\flat)(Y)=2 \on{Sym} \big(\nabla(X^\flat)\big)(Y,Z).\qedhere
\end{align*}
\end{proof}

\section{Variation of the inverse of the metric}\label{va:in}

\begin{lem*}
The differential of the inverse of the pullback metric
\begin{equation*}\left\{ \begin{array}{ccl}
\Imm &\to &\Ga\big(L(T^*M,TM)\big),\\
f &\mapsto &g\i=(f^*\g)\i
\end{array}\right.\end{equation*}
is given by
\begin{align*}
D_{(f,f_t)} g\i = D_{(f,f_t)} (f^*\g)\i =2 \g(f_t^\bot, g\i S  g\i) + \mathcal L_{f_t^\top}(g\i)
\end{align*}
\end{lem*}
\begin{proof}
\begin{align*}
\p_t g\i &= - g\i (\p_t g ) g\i
 = -g\i \big(-2 \g(f_t^\bot,S)+ \L_{f_t^\top} g\big) g\i \\
& = 2 g\i \g(f_t^\bot,S) g\i -g\i (\L_{f_t^\top} g) g\i
= 2 \g(f_t^\bot,g\i S g\i)+ \L_{f_t^\top} (g\i) \qedhere
\end{align*}
\end{proof}

\section{Variation of the volume density}\label{va:vo}

\begin{lem*}
The differential of the volume density
\begin{equation*}
\left\{ \begin{array}{ccl}
\Imm &\to &\Vol(M),\\
f &\mapsto &\vol(g)=\vol(f^*\g)
\end{array}\right.\end{equation*}
is given by
\begin{equation*}
D_{(f,f_t)} \vol(g) = 
\Tr^g\big(\g(\nabla f_t,Tf)\big) \vol(g)=
\Big(\on{div}^{g}(f_t^{\top})-\g\big(f_t^{\bot},\Tr^g(S)\big)\Big) \vol(g).
\end{equation*}
\end{lem*}

\begin{proof}
Let $g(t) \in \Ga(S^2_{>0}T^*M)$ be any curve of Riemannian metrics. Then
$$\p_t \vol(g)=\frac{1}{2}\on{Tr}(g\i.\p_t g)\vol(g).$$
This follows from the formula for $\vol(g)$ in a local oriented chart
$(u^1,\ldots u^m)$ on $M$:
\begin{align*}
\p_t\vol(g)&=\p_t \sqrt{\det( (g_{ij})_{ij})}\ du^1\wedge\cdots\wedge du^{m}\\
&=\frac{1}{2\sqrt{\det ((g_{ij})_{ij})}}\on{Tr}(\on{adj}(g) \p_t g)\
du^1\wedge\cdots\wedge du^{m}\\
&=\frac{1}{2\sqrt{\det ((g_{ij})_{ij})}}\on{Tr}(\det((g_{ij})_{ij})g^{-1}\p_t g)\
du^1\wedge\cdots\wedge du^{m}\\
&=\frac{1}{2}\on{Tr}(g\i.\p_t g)\vol(g)
\end{align*}
Now one can set $g = f^*\g$ and plug in the formula
$$\p_t g=\p_t (f^*\g)=2\on{Sym}\g(\nabla f_t,Tf)$$
from \ref{va:me}. 
This immediately proves the first formula:
\begin{align*}
\p_t \vol(g)&=\frac12 \Tr\big(g\i.2\on{Sym}\g(\nabla f_t,Tf) \big) 
=\Tr^g\big(\g(\nabla f_t,Tf) \big).
\end{align*}
Expanding this further yields the second formula:
\begin{align*}
\p_t \vol(g)&=\Tr^g\Big(\nabla\g( f_t,Tf)-\g( f_t,\nabla Tf) \Big)\\&
=\Tr^g\Big(\nabla\g( f_t,Tf)-\g( f_t,S) \Big)=-\nabla^*\g( f_t,Tf)-\g\big(f_t,\Tr^g(S)\big)\\&
=-\nabla^*\big((f_t^{\top})^{\flat}\big)-\g\big(f_t^{\bot},\Tr^g(S)\big)
=\on{div}(f_t^{\top})-\g\big(f_t^{\bot},\Tr^g(S)\big). 
\end{align*}
Here it has been used that
$$\nabla Tf = S \quad \text{and} \quad 
\on{div}(f_t^\top) = \Tr(\nabla f_t^\top)= \Tr^g\big((\nabla f_t^\top)^\flat\big)
= -\nabla^*\big((f_t^\top)^\flat\big).$$
Note that by \ref{va:ta}, the formula for the tangential variation 
would have followed also from the equivariance of the volume form with respect to pullbacks by 
diffeomorphisms. 
\end{proof}

\section{Variation of the covariant derivative}\label{va:co}

In this section, let $\nabla=\nabla^g=\nabla^{f^*\g}$ be the 
Levi-Civita covariant derivative acting on vector fields on $M$. 
Since any two covariant derivatives on $M$ differ by a tensor field, 
the first variation of $\nabla^{f^*\g}$ is tensorial. It is given by the 
tensor field $D_{(f,f_t)} \nabla^{f^*\g} \in \Ga(T^1_2 M)$.

\begin{lem*}
The tensor field $D_{(f,f_t)}\nabla^{f^*\g}$ is determined by the following relation
holding for vector fields $X,Y,Z$ on $M$:
\begin{multline*}
g\big((D_{(f,f_t)} \nabla)(X, Y),Z\big) = 
\frac12 (\nabla D_{(f,f_t)} g)\big( X \otimes Y \otimes Z
+ Y \otimes X \otimes Z	- Z \otimes X \otimes Y \big)
\end{multline*}
\end{lem*}

\begin{proof}
The defining formula for the covariant derivative is
\begin{align*}
g(\nabla_X Y,Z)&= \frac12 \Big[ Xg(Y,Z)+Yg(Z,X)-Zg(X,Y)\\&\qquad
-g(X,[Y,Z])+g(Y,[Z,X])+g(Z,[X,Y]) \Big].
\end{align*}
Taking the derivative $D_{(f,f_t)}$ yields
\begin{multline*}
(D_{(f,f_t)}g)(\nabla_X Y,Z)+g\big((D_{(f,f_t)}\nabla)(X, Y),Z\big)\\ 
\begin{aligned}
=\frac12 \Big[ & X\big((D_{(f,f_t)}g)(Y,Z)\big)+Y\big((D_{(f,f_t)}g)(Z,X)\big)-Z\big((D_{(f,f_t)}g)(X,Y)\big)\\&
-(D_{(f,f_t)}g)(X,[Y,Z])+(D_{(f,f_t)}g)(Y,[Z,X])+(D_{(f,f_t)}g)(Z,[X,Y]) \Big].
\end{aligned}
\end{multline*}
Then the result follows by replacing all Lie brackets in the above formula by covariant derivatives using 
$[X,Y]=\nabla_X Y - \nabla_Y X$
and by expanding all terms of the form $X\big((D_{(f,f_t}g)(Y,Z)\big)$ using
\begin{align*}
&X\big((D_{(f,f_t)}g)(Y,Z)\big)=\\&\qquad\qquad
(\nabla_X D_{(f,f_t)}g)(Y,Z)
+(D_{(f,f_t)}g)(\nabla_X Y,Z)
+(D_{(f,f_t)}g)(Y,\nabla_X Z).
\qedhere\end{align*}
\end{proof}

\section{Variation of the Laplacian}\label{va:la}

The Laplacian as defined in section \ref{no:la} 
can be seen as a smooth section of the bundle $L(T\Imm;T\Imm)$ over $\Imm$ since 
for every $f \in \Imm$ it is a mapping 
$$\De^{f^*\g}:T_f\Imm \to T_f\Imm.$$
The right way to define a first variation is to 
use the covariant derivative defined in section~\ref{sh:cov}.

\begin{lem*}
For $\De \in \Ga\big(L(T\Imm;T\Imm)\big)$, $f \in \Imm$ and $h \in T_f\Imm$ one has
\begin{align*}
(\nabla_{f_t} \Delta)(h) &=
\on{Tr}\big(g\i.(D_{(f,f_t)}g).g\i \nabla^2 h\big) 
-\nabla_{\big(\nabla^*(D_{(f,f_t)} g)+\frac12 d\on{Tr}^g(D_{(f,f_t)}g)\big)^\sharp}h \\&\qquad
+\nabla^*\big(R^{\g}(f_t,Tf)h\big)
-\Tr^g\Big( R^{\g}(f_t,Tf)\nabla h \Big).
\end{align*}
\end{lem*}

\begin{proof}
Let $f$ be a curve of immersions and $h$ a vector field along $f$. One has
$$\De : \Imm \to L(T\Imm;T\Imm), \quad \De \o f = \De^{f^*\g} : \R \to \Imm \to L(T\Imm;T\Imm). $$
Using property~\ref{no:co}.5 one gets
\begin{align*}
(\nabla_{f_t} \De)(h) &=
\big(\nabla_{\p_t} (\De \o f)\big)(h)=
\nabla_{\p_t} \De h - \De \nabla_{\p_t} h\\&=
-\nabla_{\p_t} \Tr^g(\nabla^2 h) - \De \nabla_{\p_t} h \\&=
\Tr\big(g\i (D_{(f,f_t)} g) g\i \nabla^2 h\big)- \Tr^g(\nabla_{\p_t} \nabla^2 h) 
- \De(\nabla_{\p_t} h).
\end{align*}
The term $\Tr^g(\nabla_{\p_t} \nabla^2 h)$ will be treated further. 
Let $X,Y$ be vector fields on $M$ that are constant in time. 
When they are seen as vector fields on $\R \x M$ then $\nabla_{\p_t}X=\nabla_{\p_t}Y=0$.
Using the formulas from section~\ref{no:sw} to swap covariant derivatives one gets
\begin{align*}
&(\nabla_{\p_t}\nabla^2 h)(X,Y)=
\nabla_{\p_t}(\nabla_X\nabla_Y h-\nabla_{\nabla_X Y}h)
\\&\qquad=
\nabla_X\nabla_{\p_t}\nabla_Y h+R^{\g}(f_t,Tf.X)\nabla_Y h-\nabla_{\p_t}\nabla_{\nabla_X Y}h
\\&\qquad=
\nabla_X\nabla_Y\nabla_{\p_t} h+\nabla_X\big(R^{\g}(f_t,Tf.Y)h\big)
+R^{\g}(f_t,Tf.X)\nabla_Y h\\&\qquad\qquad
-\nabla_{\nabla_X Y}\nabla_{\p_t}h-\nabla_{[\p_t,\nabla_X Y]}h
-R^{\g}(f_t,Tf.\nabla_X Y)h.
\end{align*}
The Lie bracket is
\begin{align*}
[\p_t,\nabla^{f^*\g}_X Y] = (D_{(f,f_t)}\nabla)(X,Y)
\end{align*}
since (now without the slight abuse of notation)
\begin{align*}
[(\p_t,0_M),(0_\R,\nabla^{f^*\g}_X Y)]
&=\p_s|_0\ TFl_{-s}^{(\p_t,0_M)} \o \nabla_X Y \o Fl_s^{(\p_t,0_M)} 
\\&= \big(0_\R,(D_{(f,f_t)}\nabla)(X,Y)\big).
\end{align*}
Therefore
\begin{align*}
&(\nabla_{\p_t}\nabla^2 h)(X,Y)=\\&\qquad=
(\nabla^2\nabla_{\p_t} h)(X,Y)+\nabla_X\big(R^{\g}(f_t,Tf.Y)h\big)
+R^{\g}(f_t,Tf.X)\nabla_Y h\\&\qquad\qquad
-\nabla_{(D_{(f,f_t)}\nabla)(X,Y)}h-R^{\g}(f_t,Tf.\nabla_X Y)h
\\&\qquad=
(\nabla^2\nabla_{\p_t} h)(X,Y)
+(\nabla_{Tf.X} R^{\g})(f_t,Tf.Y)h
+R^{\g}(\nabla_X f_t,Tf.Y)h\\&\qquad\qquad
+R^{\g}(f_t,\nabla_X Tf.Y)h
+R^{\g}(f_t,Tf.Y)\nabla_X h
+R^{\g}(f_t,Tf.X)\nabla_Y h\\&\qquad\qquad
-\nabla_{(D_{(f,f_t)}\nabla)(X,Y)}h-R^{\g}(f_t,Tf.\nabla_X Y)h
\\&\qquad=
(\nabla^2\nabla_{\p_t} h)(X,Y)
+(\nabla_{Tf.X} R^{\g})(f_t,Tf.Y)h
+R^{\g}(\nabla_X f_t,Tf.Y)h\\&\qquad\qquad
+R^{\g}\big(f_t,(\nabla Tf)(X,Y)\big)h
+R^{\g}(f_t,Tf.Y)\nabla_X h
+R^{\g}(f_t,Tf.X)\nabla_Y h\\&\qquad\qquad
-\nabla_{(D_{(f,f_t)}\nabla)(X,Y)}h
\\&\qquad=
(\nabla^2\nabla_{\p_t} h)(X,Y)
+ \nabla_X\big(R^{\g}(f_t,Tf.Y)h\big)
+R^{\g}(f_t,Tf.X)\nabla_Y h \\&\qquad\qquad
-\nabla_{(D_{(f,f_t)}\nabla)(X,Y)}h
\end{align*}
Putting together all terms one obtains
\begin{align*}
(\nabla_{f_t} \De)(h) &=
\Tr\big(g\i (D_{(f,f_t)} g) g\i \nabla^2 h\big)
-\Tr^g\Big(  \nabla\big(R^{\g}(f_t,Tf)h\big) \Big)\\&\qquad
-\Tr^g\Big( R^{\g}(f_t,Tf)\nabla h \Big)
+\nabla_{\Tr^g(D_{(f,f_t)}\nabla)}h\\&=
\Tr\big(g\i (D_{(f,f_t)} g) g\i \nabla^2 h\big)
+\nabla^*\big(R^{\g}(f_t,Tf)h\big) \\&\qquad
-\Tr^g\Big( R^{\g}(f_t,Tf)\nabla h \Big)
+\nabla_{\Tr^g(D_{(f,f_t)}\nabla)}h.
\end{align*}
It remains to calculate $\Tr^g(D_{(f,f_t)}\nabla)$. 
Using the variational formula for $\nabla$ from section~\ref{va:co}
one gets for any vector field $Z$ and a $g$-orthonormal frame $s_i$
\begin{align*}
&g\big(\Tr^g(D_{(f,f_t)}\nabla),Z\big) \\&\qquad=
\frac12 \sum_i (\nabla D_{(f,f_t)} g)\big( s_i \otimes s_i \otimes Z
+ s_i \otimes s_i \otimes Z	- Z \otimes s_i \otimes s_i \big) \\&\qquad=
-\big(\nabla^*(D_{(f,f_t)}g)\big)(Z) - \frac12 \Tr^g(\nabla_Z D_{(f,f_t)}g) \\&\qquad=
-\big(\nabla^*(D_{(f,f_t)}g)\big)(Z) - \frac12\nabla_Z \Tr^g(D_{(f,f_t)}g) \\&\qquad=
-\Big(\nabla^*(D_{(f,f_t)}g) + \frac12 d \Tr^g(D_{(f,f_t)}g)\Big)(Z) \\&\qquad=
-g\Big(\big(\nabla^*(D_{(f,f_t)}g) + \frac12 d \Tr^g(D_{(f,f_t)}g)\big)^\sharp,Z\Big) .
\end{align*}
Therefore 
\begin{equation*}\label{va:la:eq2}
\Tr^g(D_{(f,f_t)}\nabla) = -\big(\nabla^*(D_{(f,f_t)}g) + \frac12 d \Tr^g(D_{(f,f_t)}g)\big)^\sharp.
\qedhere
\end{equation*}
\end{proof}

\chapter{Sobolev-type metrics}\label{so}

\begin{ass*}
Let $P$ be a smooth section of the bundle $L(T\Imm;T\Imm)$ over $\Imm$
such that at every $f \in \Imm$ the operator 
$$P_f:T_f\Imm \to T_f\Imm$$ 
is an elliptic pseudo differential operator that is positive and symmetric with respect to 
the $H^0$-metric on $\Imm$, 
$$H^0_f(h,k) = \int_M \g(h,k)\vol(g).$$
\end{ass*}

Then $P$  induces a  metric on the set of immersions, namely
$$G^P_f(h,k)=\int_M \g(P_fh,k) \vol(g) \quad \text{for} \quad f \in \Imm, \quad h,k \in T_f\Imm.$$ 
In this section, the geodesic equation on $\Imm$ and $B_i$ for the $G^P$-metric will be calculated
in terms of the operator $P$ and it will be proven that it is well-posed under some assumptions. 

This section is based on \cite[section~4]{Michor119}.

\section{Invariance of $P$ under reparametrizations}\label{so:in}

\begin{ass*}
It will be assumed that $P$ is invariant under 
the action of the reparametrization group $\Diff(M)$ acting on $\Imm(M,N)$, i.e. 
$$P=(r^{\ph})^* P \qquad \text{for all } \ph \in \Diff(M).$$ 
\end{ass*}

For any $f \in \Imm$ and $\ph \in \Diff(M)$ this means
$$P_f = (T_fr^{\ph})\i \o P_{f \o \ph} \o T_fr^{\ph}.$$
Applied to $h \in T_f\Imm$ this means 
$$P_f(h) \o \ph = P_{f \o \ph}(h \o \ph).$$

The invariance of $P$ implies that the induced metric $G^P$ is invariant 
under the action of $\Diff(M)$, too. Therefore it induces a 
unique metric on $B_i$ as explained in section~\ref{sh:rish}

\section{The adjoint of $\nabla P$}\label{so:ad}

The following construction is needed to express the metric gradient $H$ which is part of
the geodesic equation. $H_f$ arises from the metric $G_f$ by differentiating it with 
respect to its foot point $f \in \Imm$.
Since $G$ is defined via the operator $P$, one also needs to differentiate $P_f$ with respect to its 
foot point. As for the metric, this is accomplished by the covariant derivate.
For $P \in \Ga\big(L(T\Imm;T\Imm)\big)$ and $m \in T\Imm$ one has
$$\nabla_m P \in \Ga\big(L(T\Imm;T\Imm)\big), \qquad \nabla P \in \Ga\big(L(T^2\Imm;T\Imm)\big).$$
See section~\ref{sh:cov} for more details.

\begin{ass*}
It is assumed that there exists a smooth \emph{adjoint} 
$$\adj{\nabla P} \in \Ga\big(L^2(T\Imm;T\Imm)\big)$$ 
of $\nabla P$ in the following sense: 
\begin{equation*}
\int_M \g\big((\nabla_m P)h,k\big) \vol(g)=\int_M \g\big(m,\adj{\nabla P}(h,k)\big) \vol(g).
\end{equation*}
\end{ass*}

The existence of the adjoint needs to be checked in each specific example, 
usually by partial integration. 
For the operator $P=1+A\Delta^p$, the existence of the adjoint will be proven
and explicit formulas will be calculated in sections~\ref{la:ad} and \ref{la:ge}. 

\begin{lem*}
If the adjoint of $\nabla P$ exists, then its tangential part is determined 
by the invariance of $P$ with respect to reparametrizations:
\begin{align*}
\adj{\nabla P}(h,k)^\top &=\big(\g(\nabla Ph,k)-\g(\nabla h,Pk)\big)^\sharp \\
&=\grad^g \g(Ph,k)-\big(\g(Ph,\nabla k)+\g(\nabla h,Pk)\big)^\sharp
\end{align*}  
for $f \in \Imm, h,k \in T_f\Imm$. 
\end{lem*}
\begin{proof}
Let $X$ be a vector field on $M$. Then
\begin{align*}
(\nabla_{Tf.X} P)(h) &= 
(\nabla_{\p_t|_0} P_{f\o Fl_t^X})(h \o Fl_0^X) \\&= 
\nabla_{\p_t|_0}\big(P_{f\o Fl_t^X}(h \o Fl_t^X)\big) -
P_{f\o Fl_0^X} \big( \nabla_{\p_t|_0}(h \o Fl_t^X)\big) \\&=
\nabla_{\p_t|_0}\big(P_f(h) \o Fl_t^X\big) -
P_f \big( \nabla_{\p_t|_0}(h \o Fl_t^X)\big) \\&=
\nabla_X\big(P_f(h)) - P_f \big( \nabla_X h\big) 
\end{align*}
Therefore one has for $m,h,k \in T_f\Imm$ that
\begin{align*} &
\int_M g\big(m^\top,\adj{\nabla P}(h,k)^\top \big) \vol(g) =
\int_M \g\big(Tf.m^\top,\adj{\nabla P}(h,k)\big) \vol(g) \\&\qquad=
\int_M \g\big((\nabla_{Tf.m^\top} P)h,k\big) \vol(g)  \\&\qquad=
\int_M \g\big(\nabla_{m^\top}(Ph)-P(\nabla_{m^\top}h),k\big) \vol(g) \\&\qquad=
\int_M \big(\g(\nabla_{m^\top}Ph,k)-\g(\nabla_{m^\top}h,Pk)\big) \vol(g) \\&\qquad=
\int_M g\Big(m^\top,\big(\g(\nabla Ph,k)-\g(\nabla h,Pk)\big)^\sharp\Big) \vol(g).
\qedhere
\end{align*} 
\end{proof}

\section{Metric gradients}\label{so:me}

As explained in section~\ref{sh:ge}, the geodesic equation can be expressed
in terms of the metric gradients $H$ and $K$. 
These gradients will be computed now.

\begin{lem*}
If $\adj{\nabla P}$ exists, then also $H$ and $K$ exist and are given by
\begin{align*}
K_f(h,m)&=P_f\i\Big((\nabla_m P)h+\Tr^g\big(\g(\nabla m,Tf)\big).Ph\Big) \\
H_f(h,k)&=P_f\i\Big(\adj{\nabla P}(h,k)^\bot-Tf.\big(\g(Ph,\nabla k)+\g(\nabla h,Pk)\big)^\sharp
\\&\qquad -\g(Ph,k).\Tr^g(S)\Big).
\end{align*}
\end{lem*}

\begin{proof}
For vector fields $m,h,k$ on $\Imm$ one has
\begin{equation}\label{so:me:na}
\begin{aligned}
&(\nabla_m G^P)(h,k)=
D_{(f,m)} \int_M \g(Ph,k) \vol(g) \\&\qquad\qquad
- \int_M \g\big(P(\nabla_m h),k\big) \vol(g)
- \int_M \g(Ph,\nabla_m k) \vol(g)
\\&\qquad=
\int_M D_{(f,m)}\g(Ph,k) \vol(g)
+ \int_M \g(Ph,k) D_{(f,m)}\vol(g)\\&\qquad\qquad
- \int_M \g\big(P(\nabla_m h),k\big) \vol(g)
- \int_M \g(Ph,\nabla_m k) \vol(g)
\\&\qquad=
\int_M \g\big(\nabla_m(Ph),k\big) \vol(g)
+ \int_M \g(Ph,\nabla_m k) \vol(g)\\&\qquad\qquad
+ \int_M \g(Ph,k) D_{(f,m)}\vol(g)\\&\qquad\qquad
- \int_M \g\big(P(\nabla_m h),k\big) \vol(g)
- \int_M \g(Ph,\nabla_m k) \vol(g)
\\&\qquad=
\int_M \g\big((\nabla_m P)h,k\big) \vol(g)
+ \int_M \g(Ph,k) D_{(f,m)}\vol(g)
\end{aligned}
\end{equation}
One immediately gets the $K$-gradient 
by plugging in the variational formula \ref{va:vo} for the volume form:
$$K_f(h,m)=P_f\i\Big((\nabla_m P)h+\Tr^g\big(\g(\nabla m,Tf)\big).Ph\Big).$$
To calculate the $H$-gradient, one rewrites equation~\eqref{so:me:na}
using the definition of the adjoint:
\begin{equation*}
\begin{aligned}
(\nabla_m G^P)(h,k)=
\int_M \g\big(m,\adj{\nabla P}(h,k)\big) \vol(g) + \int_M \g(Ph,k) D_{(f,m)}\vol(g).
\end{aligned}
\end{equation*}
Now the second summand is treated further using again the variational formula of the volume density 
from section~\ref{va:vo}:
\begin{align*}
&\int_M \g(Ph,k) D_{(f,m)}\vol(g)=
\int_M \g(Ph,k) \Tr^g\big(\g(\nabla m,Tf)\big) \vol(g) \\
&\qquad=
\int_M \g(Ph,k) \Tr^g\big(\nabla\g(m,Tf)-\g(m,\nabla Tf)\big) \vol(g) \\
&\qquad=
\int_M \g(Ph,k) \Big(-\nabla^*\g(m,Tf)-\g\big(m,\Tr^g(S)\big)\Big) \vol(g) \\
&\qquad=
-\int_M g^0_1\big(\nabla\g(Ph,k),\g(m,Tf)\big)\vol(g)-\int_M \g(Ph,k) \g\big(m,\Tr^g(S)\big) \vol(g) \\
&\qquad=
\int_M \g\big(m,-Tf.\grad^g\g(Ph,k)-\g(Ph,k) \Tr^g(S)\big) \vol(g) 
\end{align*} 
Collecting terms one gets that 
\begin{align*}
&G_f^P(H_f(h,k),m)=(\nabla_m G^P)(h,k)\\&\qquad=
\int_M \g\big(m,\adj{\nabla P}(h,k)-Tf.\grad^g\g(Ph,k)-\g(Ph,k) \Tr^g(S)\big) \vol(g) 
\end{align*}
Thus the $H$-gradient is given by
\begin{align*}
H_f(h,k)=P\i\Big(\adj{\nabla P}(h,k)-Tf.\grad^g\g(Ph,k) -\g(Ph,k).\Tr^g(S)\Big)
\end{align*}
The highest order term $\grad^g\g(Ph,k)$ cancels out when taking into 
account the formula for the tangential part of the adjoint from section~\ref{so:ad}:
\begin{align*}
H_f(h,k)=P\i\Big(&\adj{\nabla P}(h,k)^\bot-Tf.\big(\g(Ph,\nabla k)+\g(\nabla h,Pk)\big)^\sharp
\\& -\g(Ph,k).\Tr^g(S)\Big). 
\qedhere
\end{align*}
\end{proof}

\section[Geodesic equation]{Geodesic equation on immersions}\label{so:ge}

The geodesic equation for a general metric on $\Imm(M,N)$ has been calculated
in section~\ref{sh:ge} and reads as 
$$\nabla_{\p_t} f_t = \frac12 H_f(f_t,f_t) - K_f(f_t,f_t). $$
Plugging in the formulas for $H,K$ derived in the last section yields the following theorem. 

\begin{thm*}
The geodesic equation for a Sobolev-type metric $G^P$ on immersions is given by
\begin{align*}
\nabla_{\p_t} f_t= &\frac12P\i\Big(\adj{\nabla P}(f_t,f_t)^\bot-2.Tf.\g(Pf_t,\nabla f_t)^\sharp
-\g(Pf_t,f_t).\Tr^g(S)\Big)
\\&
-P\i\Big((\nabla_{f_t}P)f_t+\Tr^g\big(\g(\nabla f_t,Tf)\big) Pf_t\Big).
\end{align*}
\end{thm*}

\section[Geodesic equation]{Geodesic equation on immersions in terms of the momentum}\label{so:gemo}

The geodesic equation in terms of the momentum has been calculated in section~\ref{sh:gemo}
for a general metric on immersions. For a Sobolev-type metric $G^P$, the momentum 
$G^P(f_t,\cdot)$ takes the form 
$$p=Pf_t\otimes\vol(g): \R \to T^*\Imm$$ 
since all other parts of the metric (namely the integral and $\g$) 
are constant and can be neglected. 

\begin{thm*}
The geodesic equation written in terms of the momentum for a Sobolev-type metric $G^P$
on $\Imm$ is given by:
\begin{equation*}
\left\{\begin{aligned}
p&=Pf_t\otimes\vol(g) 
\\
\nabla_{\p_t}p &= \frac12\big(\adj{\nabla P}(f_t,f_t)^\bot-2Tf.\g(Pf_t,\nabla f_t)^\sharp
-\g(Pf_t,f_t)\Tr^g(S)\big)\otimes\vol(g)
\end{aligned}\right.
\end{equation*}
\end{thm*}

\section{Well-posedness of the geodesic equation}\label{so:we}

It will be proven that the geodesic equation for a Sobolev-type metric $G^P$ on $\Imm$ is well-posed 
under some assumptions on $P$. 
It will also be shown that $(\pi,\exp)$ is a diffeomorphism from a neighbourhood 
of the zero section in $T\Imm$ to a neighbourhood of the diagonal in $\Imm \x \Imm$. 

First, Sobolev sections of vector bundles are introduced. 
More information can be found in \cite{Shubin1987} and in \cite{EichhornFricke1998}. 
Let $E\to M$ be a vector bundle over a compact Riemannian manifold $(M,g)$. 
$E$ is given a fiber Riemannian metric $\hat g^E$ 
and a compatible covariant derivative $\hat \nabla^E$ on $\Ga(E)$ is chosen. 
Let $\hat \nabla^M$ be the Levi-Civita covariant derivative on $\Ga(TM)$ for a fixed 
background metric $\hat g$ on $M$.
Then the {\it Sobolev space} $H^k(E)$ is the Hilbert space completion 
of the space of smooth sections $\Ga(E)$ in the Sobolev norm
$$
\|s\|_k^2 = \sum_{j=0}^k \int_M  (\hat g^E\otimes \hat g^0_j)((\hat \nabla^E)^j s, (\hat \nabla^E)^j s)\vol(\hat g).
$$
The Sobolev space does not depend on the choices of $\hat g,\hat g^E,\hat \nabla^E$; 
the resulting norms are equivalent, see \cite{Shubin1987}. 
The following results hold (see \cite{Eichhorn2007}):
\newtheorem*{SL}{Sobolev lemma}
\newtheorem*{MP}{Module property of Sobolev spaces}
\begin{SL}
If $k>\dim(M)/2$ then the identy on $\Ga(E)$ extends to a
injective bounded linear mapping $H^{k+p}(E)\to C^p(E)$ where $C^p(E)$ carries the 
supremum norm of all derivatives up to order $p$.
\end{SL}
\begin{MP}
If $k>\dim(M)/2$ then pointwise evaluation $H^k(L(E,E))\x H^k(E)\to H^k(E)$ is
bounded bilinear. Likewise all other pointwise contraction operations are multilinear bounded
operations.
\end{MP}
The following notation shall be used: 
\begin{align*}
T\Imm(M,N) &= C^\infty_{\Imm}(M,TN):= \{s\in C^\infty(M,TN): \pi_N\o s\in\Imm(M,N)\},
\\
T_f\Imm(M,N) &= C^\infty_f(M,TN):= \{s\in C^\infty(M,TN): \pi_N\o s = f\} \cong \Ga(f^*TN).
\end{align*}
The smooth Sobolev manifolds (for $k>\dim(M)/2+1$) 
$$
\Imm^{k+1}(M,N) \subset \Imm^k(M,N),\qquad \bigcap_{k}\Imm^k(M,N) = \Imm(M,N) 
$$
shall also be considered. 
They are constructed from the Sobolev completions in each canonical chart separately and then 
glued together, always with respect to the the background metrics. See \cite{Eichhorn2007} for a 
detailed treatment in a more general situation ($M$ does not need to be compact there).

It is assumed that the operator $P$ satisfies the following properties, for 
$k>\frac{\dim(M)}{2}+1$. (Some of the assumptions have already been stated earlier.)

\begin{ass} 
$P \in \Ga\big(L(T\Imm;T\Imm)\big)$ is smooth and 
invariant under the action of $\Diff(M)$. 
(See section~\ref{so:in} for the definition of invariance.)
\end{ass}

\begin{ass}
For each $f\in \Imm(M,N)$ the operator 
$$P_f:\Ga(f^*TN) \to \Ga(f^*TN)$$ 
is an elliptic pseudo-differential operator of order $2p$ for $p>0$ of classical type
which is positive and symmetric with respect to the $H^0$-metric on $\Imm$, 
$$H^0_f(h,k) = \int_M \g(h,k)\vol(g)\qquad \text{for } f \in \Imm \text{ and } h,k \in T_f\Imm.$$
\end{ass}

Since $P_f$ is elliptic, it is unbounded selfadjoint on the Hilbert completion of $T_f\Imm$
with respect to $H^0$, see \cite[theorem 26.2]{Shubin1987}.
Furthermore $P_f$ extends to a bounded injective (since it is positive) linear operator 
$$
P_f: H^{k+2p}(f^*TN)\to H^k(f^*TN)
$$
which is also surjective (since it is Fredholm as an elliptic operator, with vanishing index 
as a selfadjoint operator). 

\begin{ass} 
$P$ extends to a smooth section of the smooth Sobolev bundle
\begin{equation*}\xymatrix{
L\big(H^{k+2p}_{\Imm^{k+2p}}(M,TN),H^{k}_{\Imm^k}(M,TN)|_{\Imm^{k+2p}(M,N)}\big) \ar[d]  \\
\Imm^{k+2p}(M,N)
}\end{equation*}
with fiber 
$$L\big(H^{k+2p}(f^*TN),H^k(f^*TN)\big)\quad\text{over}\quad f\in \Imm^{k+2p}(M,N)$$ 
such that 
$$P_f:H^{k+2p}(f^*TN)\to H^k(f^*TN)$$ 
is invertible for each $f\in \Imm^{k+2p}(M,N)$.
\end{ass}

By the implicit function theorem on Banach spaces, 
$f\mapsto P_f\i$ is then a smooth section of the smooth Sobolev bundle 
\begin{equation*}\xymatrix{
L\big(H^{k}_{\Imm^k}(M,N)|_{\Imm^{k+2p}(M,N)}, H^{k+2p}_{\Imm^{k+2p}}(M,N)\big) \ar[d]  \\
\Imm^{k+2p}(M,N)
}\end{equation*}
Moreover, $P_f\i$ is also an elliptic pseudo-differential operator of order $-2p$. 

\begin{ass} 
The normal part of the adjoint $f\mapsto \adj{\nabla P}_f^\bot$ from 
section~\thetag{\ref{so:ad}} extends to a smooth section of the smooth Sobolev bundle 
\begin{equation*}\xymatrix{
L^2_{\text{sym}}\big(H^{k+2p}(f^*TN);H^{k}(f^*TN)\big) \ar[d]  \\
\Imm^{k+2p}(M,N)
}\end{equation*}
\end{ass}

\begin{ass}\label{so:we:ass5}
All mappings $P_fu$, $P_f\i u$, and $\adj{\nabla P}_f^\bot(u,v)$ are linear pseudo differential 
operators in $u$ and $v$ of order $2p$, $-2p$, and $2p$, respectively. 
As mappings in the footpoint $f$, viewed locally in a trivialisation of the bundle 
$C^\infty_{\Imm}(M,TN)\to \Imm(M,N)$, 
they are nonlinear, and it is assumed that they are a composition of operators of the following type:
(a) Local operators of order $\le 2p$, i.e., 
$A(f)(x)=A(x,\hat \nabla^{2p}f(x),\hat \nabla^{2p-1}f(x),\dots,\hat \nabla f(x), f(x))$. 
(b) Linear pseudo-differential operators.
\end{ass}

These properties hold for the operator $1+A\De^p$ considered in section~\ref{la}.

\begin{thm*}
Let $p\ge 1$ and $k>\dim(M)/2+1$, and let
$P$ satisfy assumptions 1--5.

Then the initial value problem for the geodesic equation 
\thetag{\ref{so:ge}}
has unique local solutions in the Sobolev manifold $\Imm^{k+2p}$ of
$H^{k+2p}$-immersions. The solutions depend smoothly on $t$ and on the initial
conditions $f(0,\;.\;)$ and $f_t(0,\;.\;)$.
The domain of existence (in $t$) is uniform in $k$ and thus this
also holds in $\Imm(M,N)$.

Moreover, in each Sobolev completion $\Imm^{k+2p}$, the Riemannian exponential mapping $\exp^{P}$ exists 
and it smooth on a $H^{k_0}$-open neighborhood of the zero section in the tangent bundle, 
and $(\pi,\exp^{P})$ is a diffeomorphism from a (smaller)  $H^{k_0}$-open neigbourhood of the zero 
section to an $H^{k_0}$-open neighborhood of the diagonal in 
$\Imm^{k+2p}\x \Imm^{k+2p}$, where $k_0$ is the smallest integer 
$>\dim(M)/2+2p$. 
All these neighborhoods are uniform in $k>\dim(M)/2$, and thus both properties of the exponential 
mapping continue to hold in $\Imm(M,N)$.
\end{thm*}

This proof is partly an adaptation of \cite[section 4.3]{Michor107}. 
It works in three steps: First, the geodesic equation 
is formulated as the flow equation of a smooth vector field on a Sobolev completion of $T\Imm$. 
Thus one gets local existence and uniqueness of solutions. Second, it is shown that 
the time-interval where a solution exists does not depend on the order of the Sobolev space
of immersions. Thus one gets solutions on the intersection of all Sobolev spaces, which 
is the space of smooth immersions. Third, a general argument proves the claims 
about the exponential map.

\begin{proof}
The geodesic equation is considered as the flow equation of a smooth
($C^\infty$) autonomous vector field $X$ on 
$$
T\Imm^k(M,N)|_{\Imm^{k+2p}(M,N)} = H^k_{\Imm^{k}}(M,TN)|_{\Imm^{k+2p}(M,N)}
$$
Locally near a fixed smooth immersion $f_0\in\Imm(M,N)$ this looks like 
$U^{k+2p}\x H^k(f_0^*TN)$ where $U^{k+2p}$ is $H^{k}$-open in 
$\Imm^{k+2p}(M,N)\subset H^{k+2p}(M,N)$.
It suffices to prove the theorem in each open subset $U^{k+2p}\x H^k(f_0^*TN)$.

On this subset one can write $X=(X_1,X_2)$ as follows (using \thetag{\ref{so:ge}}):
\begin{align*}
f_t &= P_f\i u =:X_1(f,u)
\\
u_t&= 
\frac12\adj{\nabla P}(P_f\i u,P_f\i u)^\bot
\\&\qquad
-Tf.\g(u,\nabla P_f\i u\big)^\sharp
-\frac12\g(u,P_f\i u).\Tr^{f^*\g}(S^f)
\tag{5}\\&\qquad
-\Tr^{f^*\g}\big(\g(\nabla P_f\i u,Tf)\big)u
\\&
=: X_2(f,u)
\end{align*}
For $(f,u)\in U^{k+2p}\x H^k$ one has $\adj{\nabla P}_f(P_f\i u,P_f\i u) \in H^k$.
When $Y\in H^k$ then also $Y^\bot \in H^k$ since $g=f^*\g \in H^{k+2p-1} \subset H^{k+1}$.
Similarly, when $\alpha \in H^k$ then also $\alpha^\sharp \in H^k$.
$S^f\in H^{k+2p-2} \subset H^k$.  
Thus a term by term investigation using \thetag{1} -- \thetag{3} shows that 
the right hand side of \thetag{5} is smooth in  $(f,u)\in U^{k+2p}\x H^k$ with values in $H^{k+2p} \x H^{k}$.
Thus by the theory of smooth ODE's on Banach spaces,
the flow $\on{Fl}^k$ exists on $U^{k+2p}\x H^k$ and is smooth in $t$ and the initial 
conditions for fixed $k>\frac{\dim(M)}{2}$.

Consider $C^\infty$ initial conditions $f_0=f(0,\quad)$ and
$u_0=P_{f_0} f_t(0,\quad)=u(0,\quad)$ for the flow equation \thetag{5} in 
$U^{\infty}\x \Ga(f_0^*TN)$. Suppose the 
trajectory $\on{Fl}^k_t(f_0,u_0)$ of
$X$ through these initial conditions in $U^{k+2p}\x H^k$ maximally exists for $t\in (-a_k,b_k)$, 
and the trajectory $\on{Fl}^{k+1}_t(f_0,u_0)$ in $U^{k+1+2p}\x H^{k+1}$ 
maximally exists for $t\in(-a_{k+1},b_{k+1})$ with $a_{k+1}<a_k$ and $b_{k+1}<b_k$, say. By
uniqueness of solutions one has $\on{Fl}^{k+1}_t(f_0,u_0)=\on{Fl}^{k}_t(f_0,u_0)$ for
$t\in (-a_{k+1,}b_{k+1})$. 
Now $\hat \nabla$ is applied to both equations \thetag{5}:
\begin{align*}
(\hat \nabla f)_t &= \hat \nabla f_t =  \hat \nabla X_1(f,u) 
\\
(\hat \nabla u)_t &=\hat \nabla u_t =  \hat \nabla X_2(f,u) 
\end{align*}
It is claimed that the highest derivatives of $f$ and $u$ appear only linearly
in $\hat \nabla X_i(f,u)$ for $i=1,2$, i.e.
$$
\hat \nabla X_i(f,u) 
= X_{i.1}(f,u)(\hat \nabla^{2p+1}f) + X_{i,2}(f,u)(\hat \nabla^{2p+1}u) + X_{i,3}(f,u)
$$
where all $X_{i,j}(f,u)(v)$ and $X_{i,3}(f,u)$ ($i,j=1,2$) are smooth in all variables, 
of highest order $2p$ in $f$ and $u$, 
linear and algebraic (i.e., of order 0) in $v$.
This claim follows from assumption~\ref{so:we:ass5}: 
(a) For a local operator we can apply the chain rule: 
The derivative of order $2p+1$ of $f$ appears only linearly.
(b) For a linear pseudo differential operator $A$ of order $k$ 
the commutator $[\hat \nabla,A]$ is a pseudo-differential operator of order $k$ again. 

Then one writes $\hat \nabla^{2p+1}f = \hat \nabla^{2p}\tilde f$ and 
$\hat \nabla^{2p+1}u = \hat \nabla^{2p}\tilde u$ for the highest derivatives only.
The last system now becomes 
\begin{align*}
\tilde f_t &= X_{1,1}(f,u)(\hat \nabla^{2p}\tilde f) + 
X_{1,2}(f,u)(\hat \nabla^{2p}\tilde u) + X_{1,3}(f,u)
\\
\tilde u_t &= X_{2,1}(f,u)(\hat \nabla^{2p}\tilde f) + 
X_{2,2}(f,u)(\hat \nabla^{2p}\tilde u) + X_{2,3}(f,u)
\end{align*}
which is inhomogeneous bounded linear
in $(\tilde f,\tilde u)\in U^{k+2p}\x H^k$ with
coefficients bounded linear operators on $H^{k+2p}$ and $H^k$, respectively.  
These coefficients are $C^\infty$ functions of $(f,u) \in U^{k+2p}\x H^k$ 
which are already known on the interval $(-a_k,b_k)$. 
This equation 
therefore has a solution $(\tilde f(t,\quad),\tilde u(t,\quad))$ for all $t$ for which the
coefficients exists, thus for all $t\in (a_k,b_k)$. The limit 
$\lim_{t\nearrow b_{k+1}} (\tilde f(t,\quad),\tilde u(t,\quad))$ exists in $U^{k+2p}\x H^k$ 
and by continuity it equals $(\hat \nabla f,\hat \nabla u)$ in $U^{k+2p}\x H^k$
for some $t>b_{k+1}$. Thus the $H^{k+1}$-flow was not
maximal and can be continued. So $(-a_{k+1},b_{k+1})=(-a_k,b_k)$.  
Iterating this procedure one concludes that the flow of $X$ exists in
$\bigcap_{m\ge k} U^{m+2p}\x H^{m}= \Imm\x C^\infty$. 

It remains to check the properties of the Riemannian exponential mapping $\exp^P$.
It is given by $\exp^P_{f}(h)= c(1)$ where $c(t)$ is the geodesic emanating from  
value $f$ with initial velocity $h$. 
The properties claimed follow
from local existence and uniqueness of solutions to 
the geodesic equation on each space $\Imm^{k+2p}(M,N)$
and from the form of the geodesic equation $f_{tt}=\Ga_f(f_t,f_t)$
when it is written down in a chart using the Christoffel symbols,
namely linearity in $f_{tt}$ and bilinearity in $f_t$.
See for example \cite[22.6 and 22.7,]{MichorH} for a detailed proof in terms of the spray
vector field $S(f,h)=(f,h;h,\Ga_f(h,h))$ which works on each 
$T\Imm^{k+2p}(M,N)=H^{k+2p}_{\Imm^{k+2p}}(M,TN)$
without any change in notation.
So one checks this on the largest of these spaces $\Imm^{k_0}(M,N)$ (i.e. with the smallest $k$). 
Since the spray on $\Imm^{k_0}$ restricts to the spray on each $\Imm^{k+2p}$,
the exponential mapping $\exp^P$ and the inverse $(\pi,\exp^P)\i$ on $\Imm^{k_0}$ restrict 
to the corresponding mappings on each $\Imm^{k+2p}$. Thus the neighborhoods of existence are 
uniform in $k$.
\end{proof}

\section{Momentum mappings}\label{so:mo}

Recall that by assumption, the operator $P$ is invariant under the action of the 
reparametrization group $\on{Diff}(M)$. Therefore
the induced metric $G^P$ is invariant under this group action, too.  
According to \cite[section~2.5]{Michor107} one gets:

\begin{thm*}
The reparametrization momentum, which is the momentum mapping 
corresponding to the action of $\Diff(M)$ on $\Imm(M,N)$, 
is conserved along any geodesic $f$ in $\Imm(M,N)$: 
\begin{align*}
&\forall X\in\X(M): \int_M  \g( Tf.X,Pf_t ) \vol(g) \\
\intertext{or equivalently}
&g\bigl((Pf_t )^\top\bigr) \vol(g) \in\Ga(T^*M\otimes_M\vol(M))
\end{align*}
is constant along $f$.
\end{thm*}

\section{Horizontal bundle}\label{so:ho}

The splitting of $T\Imm$ into horizontal and vertical subspaces will be calculated
for Sobolev-type metrics $G^P$. See section~\ref{sh:sub} for the general theory. 
By definition, a tangent vector $h$ to $f \in \Imm(M,N)$ is horizontal if
and only if it is $G^P$-perpendicular to
the $\on{Diff}(M)$-orbits. This is the case if and only if 
$\g( P_f h(x), T_x f .X_x) = 0$ at every point $x \in M$. 
Therefore the horizontal bundle at the point $f$ equals 
\begin{align*}
&\big\{h\in T_f\Imm: P_fh(x)\,\bot\, T_x f(T_xM)\text{ for all 
}x\in M\}
=\big\{h : (P_fh)^\top = 0\big\}. 
\end{align*}
Note that the horizontal bundle consists of vector fields that are normal to $f$ 
when $P=\Id$, i.e. for the $H^0$-metric on $\Imm$.

Let us work out the $G^P$-decomposition of $h$ into vertical and 
horizontal parts. This decomposition is written as
\begin{equation*}
h= Tf.h^{\text{ver}} + h^{\text{hor}}.
\end{equation*}
Then 
\begin{align*}
P_fh = P_f (Tf.h^{\text{ver}}) + P_f h^{\text{hor}} \quad \text{with} \quad
(P_fh)^\top = (P_f (Tf.h^{\text{ver}}))^\top + 0. 
\end{align*}
Thus one considers the operators
\begin{align*}
&P_f^\top :\X(M) \to \X(M), &\qquad  P_f^\top (X) &= \big(P_f(Tf.X)\big)^\top,
\\
&P_{f,\bot}:\X(M)\to \Ga\big(\Nor(f)\big) \subset C^\infty(M,TN), &  P_{f,\bot}(X) &=  
\big(P_f(Tf.X)\big)^\bot.
\end{align*}
The operator $P_f^\top $ is unbounded, positive and symmetric on the Hilbert completion
of $T_f\Imm$ with respect to the $H^0$-metric since one has 
\begin{align*}
\int_M g(P_f^\top X,Y)\vol(g) &= 
\int_M \g(Tf.P_f^\top X,Tf.Y)\vol(g) 
\\&
= \int_M \g(P_{f}(Tf.X),Tf.Y)\vol(g) 
\\&
= \int_M g(P_f^\top Y,X)\vol(g), 
\\
\int_M g(P_f^\top X,X)\vol(g) &= 
\int_M \g(P_{f}(Tf.X),Tf.X)\vol(g) \; > 0 \quad\text{  if }X\ne 0.
\end{align*}
Let $\si^{P_f}$ and $\si^{P_f^\top}$ denote the principal symbols of $P_f$ and $P_f^\top$, 
respectively. Take any $x \in M$ and $\xi\in T^*_xM\setminus\{0\}$.
Then $\si^{P_f}(\xi)$ is symmetric, positive definite on $(T_{f(x)}N,\g)$. 
This means that one has for any $h,k \in T_{f(x)}N$ that
\begin{equation*}
\g\big(\si^{P_f}(\xi)h,k\big) = \g\big(h,\si^{P_f}(\xi)k\big), \qquad
\g\big(\si^{P_f}(\xi)h,h\big) > 0 \text{ for } h \neq 0.
\end{equation*}
The principal symbols $\si^{P_f}$ and $\si^{P_f^\top}$ are related by
\begin{equation*}
g\big(\si^{P_f^\top}(\xi)X,Y\big) = \g\big(Tf.\si^{P_f^\top}(\xi)X,Tf.Y\big)
= \g\big(\si^{P_f}(\xi)Tf.X,Tf.Y\big),
\end{equation*}
where $X,Y \in T_xM$.
Thus $\si^{P_f^\top}(\xi)$ is symmetric, positive definite on $(T_xM,g)$.
Therefore $P_f^\top $ is again elliptic, thus it is selfadjoint, 
so its index (as operator $H^{k+2p}\to H^{k}$) vanishes. It 
is injective (since positive), hence it is bijective and thus invertible. 
Thus it has been proven: 

\begin{lem*}
The decomposition of $h \in T\Imm$ into its vertical and horizontal components 
is given by
\begin{align*}
h^{\text{ver}} &= (P_f^\top )\i\big((P_fh)^\top\big),
\\
h^{\text{hor}} &= h - Tf.h^{\text{ver}} =  h - Tf.(P_f^\top )\i\big((P_fh)^\top\big).
\end{align*}
\end{lem*}

\section{Horizontal curves}\label{so:ho2}

To establish the one-to-one correspondence between curves in shape space and horizontal 
curves in $\Imm$ that has been described in theorem~\ref{sh:sub}, 
one needs the following property: 

\begin{lem*}
For any smooth path $f$ in $\Imm(M,N)$ there exists a
smooth path $\ph$ in $\on{Diff}(M)$ with $\ph(0,\;.\;)=\on{Id}_M$ 
depending smoothly on $f$ such that
the path $\tilde f$ given by $\tilde f(t,x)=f(t,\ph(t,x))$ is horizontal:
$$\g \big( P_{\tilde f}(\p_t\tilde f),T\tilde f.TM \big) =0.$$
Thus any path in shape space can be lifted to a horizontal path of immersions.
\end{lem*}

The basic idea is to write the path $\ph$ as the integral curve of a time dependent vector field. 
This method is called the Moser-trick (see \cite[Section 2.5]{Michor102}).

\begin{demo}{Proof}
Since $P$ is invariant, one has $(r^\ph)^* P = P$ or 
$P_{f\o \ph}(u\o \ph)=(P_fu)\o\ph$ for $\ph\in\on{Diff}(M)$. 
In the following $f\circ\varphi$ will denote the map $f(t, \varphi(t,x))$, etc.
One looks for $\ph$ as the integral curve of a time dependent vector field $\xi(t,x)$ on
$M$, given by $\ph_t=\xi\o \ph$.
The following expression must vanish for all $x \in M$ and $X_x \in T_x M$:
\begin{align*}
0&=\g\Big( P_{f\o\ph}\big(\p_t(f\circ \varphi)\big)(x), T(f\circ \varphi).X_x \Big) \\&=
\g\Big( P_{f\o\ph}\big((\p_t f)\circ\varphi +Tf.(\p_t\varphi)\big)(x), T(f\circ \varphi).X_x \Big)
\\ &=
\g\Big( \big((P_{f}(\p_tf)) +P_{f}(Tf.\xi)\big)\big(\ph(x)\big),Tf\circ T\varphi.X_x \Big) 
\end{align*}
Since $T\varphi$ is surjective, $T\varphi.X$ exhausts the tangent space $T_{\varphi(x)}M$, and one has
$$\big((P_{f}(\p_tf)) +P_{f}(Tf.\xi)\big)\big(\ph(x)\big)\quad \perp \quad f.$$
This holds for all $x \in M$, and by the surjectivity of $\varphi$, one also has that
$$\big((P_{f}(\p_tf)) +P_{f}(Tf.\xi)\big)(x)\quad \perp \quad f$$
at all $x \in M$. 
This means that the tangential part $\big(P_{f}(\p_tf) +  P_f(Tf.\xi)\big)^\top$ vanishes.
Using the time dependent vector field
$$\xi=-(P_f^\top )\i \big((P_f \p_t f)^\top\big)$$
and its flow $\ph$ achieves this.
\qed\end{demo}

\section[Geodesic equation]{Geodesic equation on shape space}\label{so:gesh}

By the previous section and theorem~\ref{sh:sub}, 
geodesics in $B_i$ correspond exactly to horizontal geodesics in $\Imm$. 
The equations for horizontal geodesics in the space of immersions have been written down in 
section~\ref{sh:gesh}. Here they are specialized to Sobolev-type metrics: 
\begin{thm*} 
The geodesic equation on shape space for a Sobolev-type metric $G^P$ 
is equivalent to the set of equations
\begin{equation*}\left\{\begin{aligned}
f_t &= f_t^\hor \in \Hor, \\
(\nabla_{\p_t} f_t)^\hor &= 
\frac12P\i\Big(\adj{\nabla P}(f_t,f_t)^\bot -\g(Pf_t,f_t).\Tr^g(S)\Big)\\
&\qquad-P\i\Big(\big((\nabla_{f_t}P)f_t\big)^\bot-\Tr^g\big(\g(\nabla f_t,Tf)\big) Pf_t\Big),
\end{aligned}\right.\end{equation*}
where $f$ is a horizontal path of immersions.
\end{thm*}
These equations are not handable very well since taking the horizontal part
of a vector to $\Imm$ involves inverting an elliptic pseudo-differential operator, 
see section~\ref{so:ho}. 
However, the formulation in the next section is much better.

\section[Geodesic equation]{Geodesic equation on shape space in terms of the momentum}\label{so:geshmo}

The geodesic equation in terms of the momentum has been derived in section~\ref{sh:geshmo} 
for a general metric on shape space. Now it is specialized to Sobolev-type metrics
using the formula for the $H$-gradient from section~\ref{so:me}. 

As in section~\ref{so:gemo} the momentum $G^P(f_t,\cdot)$ is identified with
$Pf_t \otimes \vol(g)$. By definition, the momentum is horizontal if it annihilates all 
vertical vectors. This  is the case if and only if $Pf_t$ is normal to $f$. 
Thus the splitting of the momentum in horizontal and vertical parts is given by
$$Pf_t \otimes \vol(g) = (Pf_t)^\bot \otimes \vol(g) + Tf.(Pf_t)^\top \otimes \vol(g).$$
This is much simpler than the splitting of the velocity in horizontal and vertical parts where 
a pseudo-differential operator has to be inverted, see section~\ref{so:ho}.
Thus the following version of the geodesic equation on shape space is the easiest to solve.

\begin{thm*}
The geodesic equation on shape space is equivalent to the set of equations
for a path of immersions $f$:
\begin{equation*}
\left\{\begin{aligned}
p &= Pf_t \otimes \vol(g), \qquad Pf_t = (Pf_t)^\bot, \\
(\nabla_{\p_t}p)^\hor &= \frac12 \Big(\adj{\nabla P}(f_t,f_t)^\bot-\g(Pf_t,f_t).\Tr^g(S)\Big) \otimes \vol(g).
\end{aligned}\right.
\end{equation*}
\end{thm*}

The equation for geodesics on $\Imm$ without the horizontality condition is
$$\nabla_{\p_t}p = \frac12\big(\adj{\nabla P}(f_t,f_t)^\bot-2Tf.\g(Pf_t,\nabla f_t)^\sharp
-\g(Pf_t,f_t).\Tr^g(S)\big)\otimes\vol(g),
$$
see section~\ref{so:gemo}. 
It has been proven in section~\ref{sh:gesh} that the vertical part of this equation 
is satisfied automatically when the geodesic is horizontal. 
Nevertheless this will be checked by hand because the proof is much simpler here
than in the general case. 

If $f_t$ is horizontal then by definition $Pf_t$ is normal to $f$. Thus one has for any $X \in \X(M)$ that
\begin{align*}
g\big((\nabla_{\p_t}Pf_t)^\top, X \big) &= 
\g(\nabla_{\p_t}Pf_t, Tf.X ) = 
0-\g(Pf_t,\nabla_{\p_t}Tf.X)\\&=
-\g(Pf_t,\nabla_X f_t)=
-g\big(\g(Pf_t,\nabla f_t)^\sharp,X).
\end{align*}
Thus 
\begin{align*}
\big(\nabla_{\p_t} p\big)^\vert &=
\big((\nabla_{\p_t} Pf_t) \otimes \vol(g) + Pf_t \otimes D_{(f,f_t)} \vol(g) \big)^\vert \\&=
Tf.(\nabla_{\p_t} Pf_t)^\top \otimes \vol(g) + Tf.(Pf_t)^\top \otimes D_{(f,f_t)} \vol(g) \\&=
-Tf.\g(Pf_t,\nabla f_t)^\sharp \otimes \vol(g) + 0 ,
\end{align*}
which is exactly the vertical part of the geodesic equation.

\chapter{Geodesic distance on shape space}\label{ge}

It came as a big surprise when it was discovered in \cite{Michor98} 
that the Sobolev metric of order zero induces vanishing geodesic distance on shape space $B_i$. 
It will be shown that this problem can be overcome by using higher order Sobolev metrics. 
The proof of this result is based on bounding the $G^P$-length of a path from below by its area swept out. 
The main result is in section~\ref{ge:no}.

This section is based on \cite[section~5]{Michor119}. The same ideas can also be found
in \cite[section~2.4]{Bauer2010}, \cite[section~7]{Michor118} and \cite[section~3]{Michor102}.

\section{Geodesic distance on shape space}\label{ge:ge}

\emph{Geodesic distance} on $B_i$ is given by
$$\dist_{G^P}^{B_i}(F_0,F_1) = \inf_F L_{G^P}^{B_i}(F),$$
where the infimum is taken over all $F :[0,1] \to B_i$ with $F(0)=F_0$ and $F(1)=F_1$.
$L_{G^P}^{B_i}$ is the length of paths in $B_i$ given by
$$L_{G^P}^{B_i}(F) = \int_0^1 \sqrt{G^P_F(F_t,F_t)} dt \quad \text{for $F:[0,1] \to B_i$.}$$
Letting $\pi:\Imm \to B_i$ denote the projection, one has 
$$L_{G^P}^{B_i}(\pi \o f) =  L_{G^P}^{\Imm}(f) =\int_0^1 \sqrt{G^P_f(f_t,f_t)} dt$$
when $f:[0,1]\to \Imm$ is horizontal.
In the following sections, conditions on the metric $G^P$ ensuring that $\dist_{G^P}^{B_i}$ separates
points in $B_i$ will be developed.

\section{Vanishing geodesic distance}\label{ge:va}

\begin{thm*}
The distance $\dist_{H^0}^{B_i}$ induced by the Sobolev $L^2$ metric of order zero vanishes. 
Indeed it is possible to connect any two distinct shapes by a path of arbitrarily short length. 
\end{thm*}

This result was first established by Michor and Mumford 
for the case of planar curves in \cite{Michor98}.
Here a more general version from \cite{Michor102} is quoted. 

\begin{proof}
Take a path $f(t,x)$ in $\on{Imm}(M,N)$ from $f_0$ to $f_1$ and
make it horizontal by the same method that was used in \ref{so:ho2}. 
Horizontality for the $H^0$-metric simply means $\g(f_t,\D f)=0$. 
This forces a reparametrization on $f_1$. 

Let $\al:M\to [0,1]$ be a surjective Morse function whose
singular values are all contained in the set
$\{\frac{k}{2N}:0\le k\le 2N\}$ for some integer $N$. We shall
use integers $n$ below and we shall use only multiples of $N$. 

Then the level sets
$M_r:=\{x\in M:\al(x)=r\}$ are of Lebesgue measure 0.
We shall also need the slices 
$M_{r_1,r_2}:=\{x\in M:r_1\le\al(x)\le r_2\}$.
Since $M$ is compact there exists a constant $C$ such that the
following estimate holds uniformly in $t$:
$$
\int_{M_{r_1,r_2}} \on{vol}(f(t,\quad)^*\g) 
\le C(r_2-r_1)\int_{M} \on{vol}(f(t,\quad)^*\g) 
$$

Let $\tilde f(t,x)=f(\ph(t,\al(x)),x)$ where 
$\ph:[0,1]\x[0,1]\to[0,1]$ is given as in \cite{Michor98},~3.10  by
\begin{displaymath}
\ph(t,\al)=
\begin{cases}
  2t(2n\al-2k), &\; 0 \le t\le 1/2,\;
    \tfrac{2k}{2n}\le \al \le\tfrac{2k+1}{2n}
  \\
  2t(2k+2-2n\al), &\; 0 \le t\le 1/2,\;
    \tfrac{2k+1}{2n}\le \al \le\tfrac{2k+2}{2n}
  \\
  2t-1+2(1-t)(2n\al-2k), &\; 1/2 \le t\le 1,\;
    \tfrac{2k}{2n}\le \al \le\tfrac{2k+1}{2n}
  \\
  2t-1+2(1-t)(2k+2-2n\al), &\; 1/2 \le t\le 1,\;
    \tfrac{2k+1}{2n}\le \al \le\tfrac{2k+2}{2n}.
\end{cases}
\end{displaymath}
See figure~\ref{ge:va:zi} for an illustration of the construction.

\begin{figure}
\centering
\includegraphics[width=.7\textwidth]{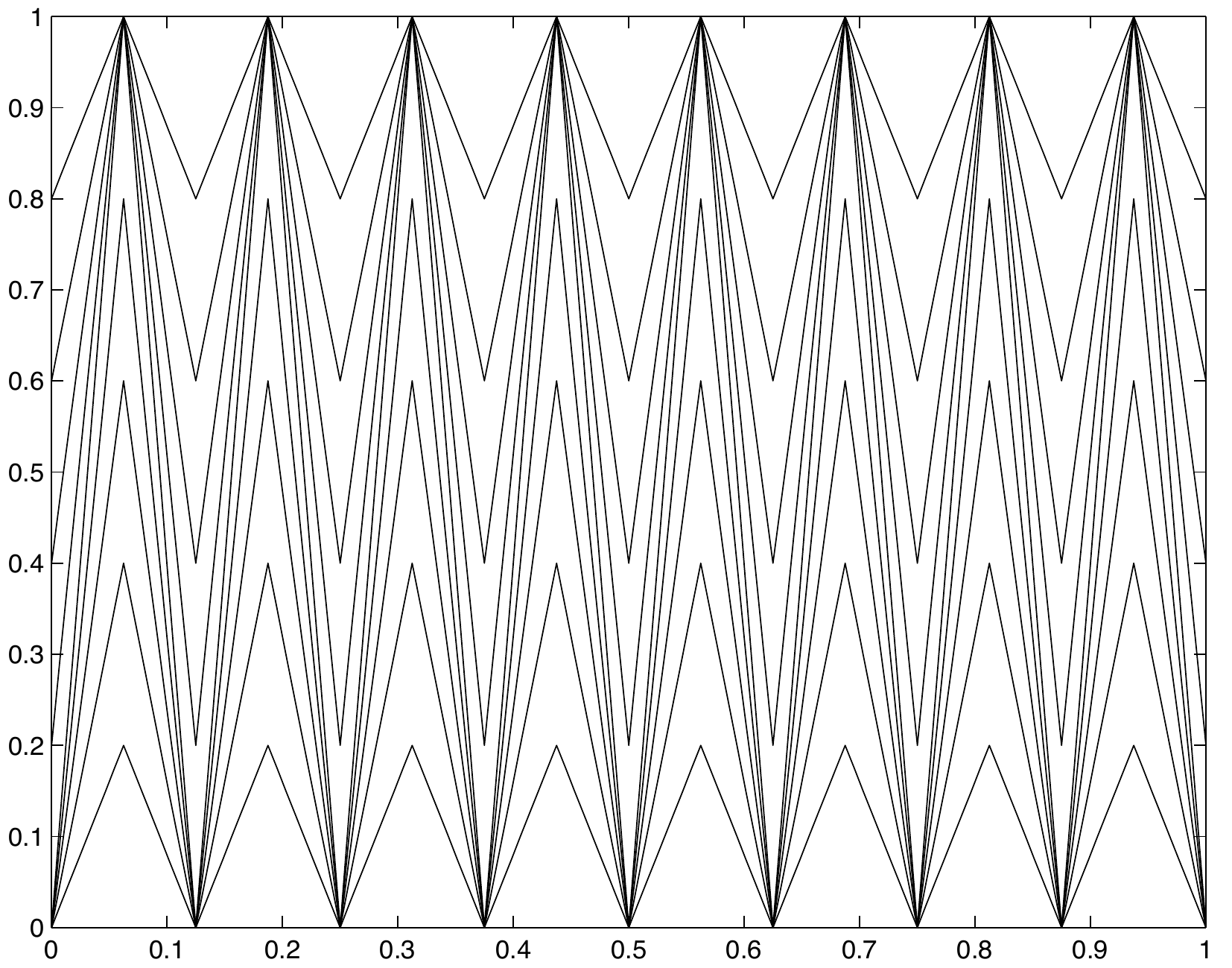}
\caption{Plot of the function $\ph$. Each zig-zagged line corresponds to 
$\ph(t,\cdot)$ for some fixed values of $t$, namely 
$t=\frac{1}{10}, \frac{2}{10}, \ldots ,\frac{9}{10}$. 
The figure is taken from \cite{Michor98}. }
\label{ge:va:zi}
\end{figure}

Then we get $\D \tilde f = \ph_\al.d\al.f_t + \D f$ and 
$\tilde f_t = \ph_t. f_t$ where
\begin{displaymath}
\ph_\al=\begin{cases} +4nt \\ -4nt \\+4n(1-t) \\ -4n(1-t) \end{cases},\qquad 
\ph_t=\begin{cases} 4n\al-4k \\ 4k+4-4n\al \\2-4n\al+4k\\-(2-4n\al+4k)\end{cases}.
\end{displaymath}
We use horizontality $\g(f_t,\D f)=0$ 
to determine $\tilde f_t^\bot=\tilde f_t + \D\tilde f(X)$ where
$X\in TM$ satisfies
$0=\g(\tilde f_t + \D\tilde f(X),\D\tilde f(\xi))$ for all
$\xi\in TM$. We also use
$$
d\al(\xi)=g(\on{grad}^{g}\al,\xi)
  =\g\big(\D f(\on{grad}^{g}\al),\D f(\xi)\big)
$$
and get
\begin{align*}
0 &= \g\big(\tilde f_t + \D\tilde f(X),\D\tilde f(\xi)\big)
\\&
= g\Bigl(\ph_tf_t + \ph_\al d\al(X) f_t + \D f(X), 
                \ph_\al d\al(\xi) f_t +\D f(\xi)\Bigr)
\\&
=\ph_t.\ph_\al.g(\on{grad}^{g}\al,\xi)\|f_t\|^2 
+\\&\quad
+\ph_\al^2.g(\on{grad}^{g}\al,X).g(\on{grad}^{g}\al,\xi)\|f_t\|^2
+\g(\D f(X),\D f(\xi))
\\&
=(\ph_t.\ph_\al+\ph_\al^2.g(\on{grad}^{g}\al,X))\|f_t\|^2
g(\on{grad}^{g}\al,\xi) + g(X,\xi)
\end{align*}
This implies that $X=\la\on{grad}^{g}\al$ for a function $\la$
and in fact we get  
\begin{align*}
\tilde f_t^\bot 
= \frac{\ph_t}
   {1+\ph_\al^2\|d\al\|_{g}^2\|f_t\|^2}
   f_t
-\frac{\ph_t \ph_\al\|f_t\|_g^2}
   {1+\ph_\al^2\|d\al\|_{f^*\g}^2\|f_t\|^2}
  \D f(\on{grad}^{g}\al)
\end{align*}
and 
$$
\|\tilde f_t\|^2 
= \frac{\ph_t^2\|f_t\|^2}
      {1+\ph_\al^2\|d\al\|_{f^*\g}^2\|f_t\|^2}
$$
From $\D \tilde f = \ph_\al.d\al.f_t + \D f$ and 
$\g(f_t,\D f)=0$ we get for the volume form 
$$
\on{vol}(\tilde f^*\g) =
\sqrt{1+\ph_\al^2\,\|d\al\|_{g}^2\|f_t\|^2}\,\on{vol}(g).
$$
For the horizontal length we get
\begin{align*}
&L^{\text{hor}}(\tilde f)  
=\int_0^1\Bigl(\int_M
  \|\tilde f_t^\bot\|^2 \on{vol}(\tilde f^*\g)\Bigr)^{\tfrac12}dt
=\\&
=\int_0^1\Bigl(\int_M
  \frac{\ph_t^2\|f_t\|^2}
      {\sqrt{1+\ph_\al^2\|d\al\|_{g}^2\|f_t\|^2}}
      \on{vol}(g)\Bigr)^{\tfrac12}dt
=\\&
=\int_0^{\frac12}\Biggl(\sum_{k=0}^{n-1}\Bigl(
\int_{M_{\tfrac{2k}{2n},\tfrac{2k+1}{2n}}} 
  \frac{(4n\al-4k)^2\|f_t\|^2}
      {\sqrt{1+(4nt)^2\|d\al\|_{g}^2\|f_t\|^2}}
  \on{vol}(g)
+\\&\qquad\qquad
+\int_{M_{\tfrac{2k+1}{2n},\tfrac{2k+2}{2n}}} 
  \frac{(4k+4-4n\al)^2\|f_t\|^2}
      {\sqrt{1+(4nt)^2\|d\al\|_{g}^2\|f_t\|^2}}
      \on{vol}(g)
\Bigr)\Biggr)^{\tfrac12}dt
+\\&\quad
+\int_{\frac12}^1\Biggl(\sum_{k=0}^{n-1}\Bigl(
\int_{M_{\tfrac{2k}{2n},\tfrac{2k+1}{2n}}} 
  \frac{(2-4n\al+4k)^2\|f_t\|^2}
      {\sqrt{1+(4n(1-t))^2\|d\al\|_{g}^2\|f_t\|^2}}
  \on{vol}(g)
+\\&\qquad\qquad\quad
+\int_{M_{\tfrac{2k+1}{2n},\tfrac{2k+2}{2n}}} 
  \frac{(2-4n\al+4k)^2\|f_t\|^2}
      {\sqrt{1+(4n(1-t))^2\|d\al\|_{g}^2\|f_t\|^2}}
  \on{vol}(g)
\Bigr)\Biggr)^{\tfrac12}dt
\end{align*}
Let $\ep>0$.
The function 
$(t,x)\mapsto \|f_t(\ph(t,\al(x)),x)\|^2$ is 
uniformly bounded. 
On $M_{\tfrac{2k}{2n},\tfrac{2k+1}{2n}}$ the function
$4n\al-4k$ has values in $[0,2]$. Choose disjoint geodesic
balls centered at the finitely many singular values of the
Morse function $\al$ of total $g$-volume $<\ep$. Restricted
to the union $M_{\text{sing}}$ of these balls the integral
above is $O(1)\ep$. So we have to estimate the integrals on the
complement $\tilde M=M\setminus M_{\text{sing}}$ where the
function $\|d\al\|_{g}$ is uniformly bounded from below by
$\et>0$.

Let us estimate one of the sums above. We use the fact that
the singular points of the Morse function $\al$ lie all on the
boundaries of the sets $\tilde M_{\tfrac{2k}{2n},\tfrac{2k+1}{2n}}$
so that we can transform the integrals as follows:
\begin{align*}
&\sum_{k=0}^{n-1}
\int_{\tilde M_{\tfrac{2k}{2n},\tfrac{2k+1}{2n}}} 
  \frac{(4n\al-4k)^2\|f_t\|^2}
      {\sqrt{1+(4nt)^2\|d\al\|_{g}^2\|f_t\|^2}}
  \on{vol}(g)=
\\&=\sum_{k=0}^{n-1}
\int_{\tfrac{2k}{2n}}^{\tfrac{2k+1}{2n}}\int_{\tilde M_{r}} 
  \frac{(4nr-4k)^2\|f_t\|^2}
      {\sqrt{1+(4nt)^2\|d\al\|_{g}^2\|f_t\|^2}}
  \frac{\on{vol}(i_r^*f^*\g)}{\|d\al\|_{g}}\;dr
\end{align*}
We estimate this sum of integrals:
Consider first the set of all $(t,r,x)\in M_r$ such that 
$|f_t(\ph(t,r),x)|<\ep$. There we estimate by 
$$
O(1).n.16n^2.\ep^2.(r^3/3)|_{r=0}^{r=1/2n}=O(\ep).
$$
On the complementary set where $|f_t(\ph(t,r),x)|\ge\ep$ we
estimate by
$$
O(1).n.16n^2.\frac1{4nt\et^2\ep}(r^3/3)|_{r=0}^{r=1/2n}
  =O(\frac1{nt\et^2\ep})
$$
which goes to 0 if $n$ is large enough.
The other sums of integrals can be estimated similarly, thus 
$L^{\text{hor}}(\tilde f)$ goes to 0 for $n\to \infty$. It is clear that
one can approximate $\ph$ by a smooth function without changing the
estimates essentially. 
\end{proof}

\section{Area swept out}\label{ge:ar}

For a path of immersions $f$ seen as a mapping $f:[0,1] \x M \to N$ one has
$$(\text{area swept out by $f$})=\int_{[0,1]\x M} \vol(f(\cdot,\cdot)^* \g) 
=\int_0^1 \int_M \norm{f_t^\bot} \vol(g) dt.$$

\section{Area swept out bound}\label{ge:ar1}

\begin{lem*}
Let $G^P$ be a Sobolev type metric that is at least as strong as the $H^0$-metric, i.e. 
there is a constant $C_1 > 0$ such that
\begin{align*}
\norm{h}_{G^P} \geq C_1 \norm{h}_{H^0} = C_1 \sqrt{\int_M \g(h,h) \vol(g)} \qquad \text{for all $h \in T\Imm$.} 
\end{align*}
Then one has the area swept out bound for any path of immersions $f$:
\begin{align*}
C_1 \ (\text{area swept out by $f$})
\leq \max_t \sqrt{\Vol\big(f(t)\big)} . L_{G^P}^{\Imm}(f).
\end{align*}
\end{lem*}

The proof is an adaptation of the one given in \cite[section~7.3]{Michor118} for almost local metrics.

\begin{proof}
\begin{align*}
L_{G^P}^{\Imm}(f)&=\int_0^1 \norm{f_t}_{G^P} dt \geq
C_1 \int_0^1 \norm{f_t}_{H^0} dt \\&\geq 
C_1 \int_0^1 \norm{f_t^\bot}_{H^0} dt =
C_1 \int_0^1 \Big(\int_M \norm{f_t^\bot}^2 \vol(g) \Big)^{\frac12} dt \\&\geq
C_1 \int_0^1 \Big(\int_M \vol(g) \Big)^{-\frac12} \int_M 1.\norm{f_t^\bot} \vol(g) dt\\&\geq
C_1 \min_t \Big(\int_M \vol(g) \Big)^{-\frac12} \int_{[0,1]\x M} \vol(f(\cdot,\cdot)^* \g)  \\&=
C_1 \Big(\max_t \int_M \vol(g) \Big)^{-\frac12}\ (\text{area swept out by $f$}) 
\qedhere
\end{align*}
\end{proof}

\section{Lipschitz continuity of $\sqrt{\Vol}$}\label{ge:li}

\begin{lem*}
Let $G^P$ be a Sobolev type metric that is at least as strong as the $H^1$-metric, i.e. 
there is a constant $C_2 > 0$ such that
\begin{align*}
\norm{h}_{G^P} \geq C_2 \norm{h}_{H^1} = 
C_2 \sqrt{\int_M \g\big( (1+\Delta) h,h \big) \vol(g)} \qquad \text{for all $h \in T\Imm$.} 
\end{align*}
Then the mapping
$$\sqrt{\Vol}:(B_i,\dist_{G^P}^{B_i}) \to \R_{\geq 0}$$
is Lipschitz continuous, i.e. for all $F_0$ and $F_1$ in $B_i$ one has:
$$
\sqrt{\Vol(F_1)}-\sqrt{\Vol(F_0)} 
\leq \frac{1}{2 C_2} \dist_{G^P}^{B_i}(F_0,F_1).
$$
\end{lem*}

For the case of planar curves, this has been proven in  \cite[section~4.7]{Michor107}.

\begin{proof}
\begin{align*}
\p_t \Vol &= \int_M \Big(\on{div}^g(f_t^\top)-\g\big(f_t^\bot, \Tr^g(S)\big)\Big) \vol(g) \\&=
0+\int_M \g(f_t, \nabla^* Tf) \vol(g) =
\int_M (g^0_1  \otimes \g)(\nabla f_t, Tf) \vol(g) \\&\leq
\sqrt{\int_M  \norm{\nabla f_t}_{g^0_1  \otimes \g}^2 \vol(g)} 
\sqrt{\int_M  \norm{Tf}_{g^0_1  \otimes \g}^2 \vol(g)} \\&\leq
\norm{f_t}_{H^1}\ \sqrt{\Vol}  \leq \frac{1}{C_2}  \norm{f_t}_{G^P}\ \sqrt{\Vol} .
\end{align*}
Thus
\begin{align*}
\p_t \sqrt{\Vol(f)}=\frac{\p_t \Vol(f)}{2 \sqrt{\Vol(f)}}\leq
\frac{1}{2 C_2} \norm{f_t}_{G^P}.
\end{align*}
By integration one gets
\begin{align*}
\sqrt{\Vol(f_1)}-\sqrt{\Vol(f_0)} &= 
\int_0^1 \p_t \sqrt{\Vol(f)}dt \leq 
\int_0^1 \frac{1}{2 C_2} \norm{f_t}_{G^P} = 
\frac{1}{2 C_2}\ L_{G^P}^{\Imm}(f).
\end{align*}
Now the infimum over all paths $f:[0,1] \rightarrow \Imm$ with $\pi(f(0))=F_0$ and $\pi(f(1))=F_1$ is taken. 
\end{proof}


\section{Non-vanishing geodesic distance}\label{ge:no}

Using the estimates proven above and the fact that the area swept out separates points at least on $B_e$, 
one gets the following result:

\begin{thm*}
The Sobolev type metric $G^P$ induces non-vanishing geodesic distance on $B_e$ if it is
stronger or as strong as the $H^1$-metric, i.e. if there is a constant $C > 0$ such that
\begin{align*}
\norm{h}_{G^P} \geq C \norm{h}_{H^1} = 
C \sqrt{\int_M \g\big( (1+\Delta) h,h \big) \vol(g)} \qquad \text{for all $h \in T\Imm$.} 
\end{align*}
\end{thm*}

\chapter[Sobolev metrics induced by the Laplacian]{Sobolev metrics induced by the Laplace operator}\label{la}

The results on non-vanishing geodesic distance from the previous section lead us to consider 
operators $P$ that are induced by the Laplacian operator: 
$$P=1+A \Delta^p, \quad P \in \Ga\big(L(T\Imm;T\Imm)\big)$$
for a constant $A>0$. 
(See section~\ref{no:la} for the definition of the Laplacian that is used in this work.)
At every $f \in \Imm$, $P_f$ is a positive, selfadjoint and bijective operator
of order $2p$ acting on $T_f\Imm = \Ga(f^*TN)$.
Note that $\Delta$ depends smoothly on the immersion $f$ via the pullback-metric $f^*\g$, 
so that the same is true of $P$.
$P$ is invariant under the action of the reparametrization group $\on{Diff}(M)$.
It induces the Sobolev metric
$$G_f^P(h,k)=\int_M \g\big(P_f (h),k\big) \vol(g)
=\int_M \g\Big(\big(1+A (\De^{f^*\g})^p \big)h,k\Big) \vol(f^*\g). $$
When $A=1$ we write $H^p := G^{1+\De^p}$.

In this section we will calculate explicitly for $P=1+A \De^p$
the geodesic equation and conserved momenta that have been deduced 
in section~\ref{so} for a general operator $P$. The hardest part will be the 
partial integration needed for the adjoint of $\nabla P$. 
As a result we will get explicit formulas that are ready to be implemented numerically.

This section is based on \cite[section~6]{Michor119}. 

\section{Other choices for $P$}\label{la:ot}

Other choices for $P$ are the operator $P=1+A (\nabla^*)^p \nabla^p$ 
corresponding to the metric
$$G_f^P(h,k)=\int_M \big(\g(h,k)+A \g(\nabla^p h,\nabla^p k) \big) \vol(g),$$
and other operators that differ only in lower order terms. Since these operators all 
have the same principal symbol, they induce equivalent metrics on each tangent space $T_f \Imm$.
It would be interesting to know if the induced geodesic distances on $B_i$ are equivalent as well.

\section{Adjoint of $\nabla P$}\label{la:ad}

To find a formula for the geodesic equation one has to calculate the adjoint of $\nabla P$, 
see section~\ref{so:ge}.
The following calculations at the same time show the existence of the adjoint.
It has been shown in section~\ref{so:ad}
that the invariance of the operator $P$ with respect to reparametrizations
determines the tangential part of the adjoint:
\begin{align*} 
\adj{\nabla P}(h,k)\big)^\top &=\grad^g \g(Ph,k)-\big(\g(Ph,\nabla k)+\g(\nabla h,Pk)\big)^\sharp.
\end{align*} 
It remains to calculate its normal part using the variational formulas from section~\ref{va}.

In the following calculations there will be terms of the form $\Tr(g\i s_1g\i s_2)$, 
where $s_1,s_2$ are two-forms on $M$. 
When the two-forms are seen as mappings $TM \to T^*M$, they can be composed with $g\i:T^*M \to TM$. 
Thus the expression under the trace is a mapping $TM \to TM$ to which the trace can be applied. 
When one of the two-forms is vector valued, the same tensor components as before are contracted.
For example when $h \in \Ga(f^*TN)$ then $s_2=\nabla^2 h$ is a two-form on $M$ with values in $f^*TN$. 
Then in the expression $\Tr(g\i.s_1.g\i.s_2)$ only $TM$ and $T^*M$ components are contracted, 
whereas the $f^*TN$ component remains unaffected.

{\allowdisplaybreaks
\begin{align*}
&\int_M \g\big(m^\bot,\adj{\nabla P}(h,k)\big) \vol(g)=
\int_M \g\big((\nabla_{m^\bot} P)h,k\big) \vol(g)\\
&\quad=A\sum_{i=0}^{p-1}\int_M\g((\nabla_{m^\bot}\Delta)\Delta^{p-i-1}h ,\Delta^{i}k )\vol(g)\\
&\quad =A\sum_{i=0}^{p-1}\int_M\g\Big(\Tr\Big(g\i.D_{(f,m^\bot)}g.g\i\nabla^2\Delta^{p-i-1}h\Big) ,\Delta^{i}k \Big)\vol(g)\\
&\qquad\qquad -A\sum_{i=0}^{p-1}\int_M\g\Big(\nabla_{\big(\nabla^*(D_{(f,m^\bot)}g)+\frac12 
d\Tr^g(D_{(f,m^\bot)}g)\big)^\sharp}\Delta^{p-i-1}h ,\Delta^{i}k \Big)\vol(g)\\
&\qquad\qquad+A\sum_{i=0}^{p-1}\int_M\g\Big(\nabla^*R^{\g}(m^\bot,Tf)\Delta^{p-i-1}h ,\Delta^{i}k \Big)\vol(g)\\
&\qquad\qquad-A\sum_{i=0}^{p-1}\int_M\g\Big(\Tr^g\big(R^{\g}(m^\bot,Tf)\nabla\Delta^{p-i-1}h\big) ,\Delta^{i}k \Big)\vol(g)\\
&=
A\sum_{i=0}^{p-1}\int_M\Tr\Big(g\i.D_{(f,m^\bot)}g.g\i \g(\nabla^2\Delta^{p-i-1}h,\Delta^{i}k )\Big) \vol(g)\\
&\qquad\qquad -A\sum_{i=0}^{p-1}\int_M (g^0_1\otimes \g)\Big(\nabla\Delta^{p-i-1}h ,(\nabla^*D_{(f,m^\bot)}g)\otimes\Delta^{i}k \Big)\vol(g)\\
&\qquad\qquad -A\sum_{i=0}^{p-1}\int_M (g^0_1\otimes \g)\Big(\nabla\Delta^{p-i-1}h ,\frac12 d\Tr^g(D_{(f,m^\bot)}g)\otimes\Delta^{i}k \Big)\vol(g)\\
&\qquad\qquad+A\sum_{i=0}^{p-1}\int_M(g^0_1\otimes \g)\Big(R^{\g}(m^\bot,Tf)\Delta^{p-i-1}h ,\nabla\Delta^{i}k \Big)\vol(g)\\
&\qquad\qquad-A\sum_{i=0}^{p-1}\int_M\g\Big(\Tr^g\big(R^{\g}(m^\bot,Tf)\nabla\Delta^{p-i-1}h\big) ,\Delta^{i}k \Big)\vol(g)
\end{align*}
Using the following  symmetry property of the curvature tensor (see \cite[24.4.4]{MichorH}):
$$\g(R^{\g}(X,Y)Z,U)=-\g(R^{\g}(Y,X)Z,U)=-\g(R^{\g}(Z,U)Y,X)$$
yields:
\begin{align*}
&\int_M \g\big(m^\bot,\adj{\nabla P}(h,k)\big) \vol(g)=\\
&\qquad=
A\sum_{i=0}^{p-1}\int_Mg^0_2\Big(D_{(f,m^\bot)}g,\g(\nabla^2\Delta^{p-i-1}h,\Delta^{i}k )\Big) \vol(g)\\
&\qquad\qquad -A\sum_{i=0}^{p-1}\int_M g^0_1\Big(\g(\nabla\Delta^{p-i-1}h,\Delta^{i}k),\nabla^*D_{(f,m^\bot)}g \Big)\vol(g)\\
&\qquad\qquad -A\sum_{i=0}^{p-1}\int_M g^0_1\Big(\g(\nabla\Delta^{p-i-1}h,\Delta^{i}k) ,\frac12 \nabla\Tr^g(D_{(f,m^\bot)}g) \Big)\vol(g)\\
&\qquad\qquad+A\sum_{i=0}^{p-1}\int_M\g\Big(\Tr^g\big(R^{\g}(\Delta^{p-i-1}h,\nabla\Delta^{i}k)Tf\big) ,m^\bot \Big)\vol(g)\\
&\qquad\qquad-A\sum_{i=0}^{p-1}\int_M\g\Big(\Tr^g\big(R^{\g}(\nabla\Delta^{p-i-1}h,\Delta^{i}k)Tf\big) , m^\bot\Big)\vol(g)\\
&\qquad=
A\sum_{i=0}^{p-1}\int_Mg^0_2\Big(D_{(f,m^\bot)}g,\g(\nabla^2\Delta^{p-i-1}h,\Delta^{i}k )\Big) \vol(g)\\
&\qquad\qquad -A\sum_{i=0}^{p-1}\int_M g^0_2\Big(\nabla\g(\nabla\Delta^{p-i-1}h,\Delta^{i}k),D_{(f,m^\bot)}g \Big)\vol(g)\\
&\qquad\qquad -\frac{A}{2}\sum_{i=0}^{p-1}\int_M \Big(\nabla^*\g(\nabla\Delta^{p-i-1}h,\Delta^{i}k) \Big)\Tr^g(D_{(f,m^\bot)}g) \vol(g)\\
&\qquad\qquad+A\sum_{i=0}^{p-1}\int_M\g\Big(\Tr^g\big(R^{\g}(\Delta^{p-i-1}h,\nabla\Delta^{i}k)Tf \big),m^\bot \Big)\vol(g)\\
&\qquad\qquad-A\sum_{i=0}^{p-1}\int_M\g\Big(\Tr^g\big(R^{\g}(\nabla\Delta^{p-i-1}h,\Delta^{i}k)Tf\big) , m^\bot\Big)\vol(g)\\
&\qquad=
A\sum_{i=0}^{p-1}\int_Mg^0_2\Big(D_{(f,m^\bot)}g,\g(\nabla^2\Delta^{p-i-1}h,\Delta^{i}k )\Big) \vol(g)\\
&\qquad\qquad -A\sum_{i=0}^{p-1}\int_M g^0_2\Big(\g(\nabla^2\Delta^{p-i-1}h,\Delta^{i}k),D_{(f,m^\bot)}g \Big)\vol(g)\\
&\qquad\qquad -A\sum_{i=0}^{p-1}\int_M g^0_2\Big(\g(\nabla\Delta^{p-i-1}h,\nabla\Delta^{i}k),D_{(f,m^\bot)}g \Big)\vol(g)\\
&\qquad\qquad -\frac{A}{2}\sum_{i=0}^{p-1}\int_M \Big(\nabla^*\g(\nabla\Delta^{p-i-1}h,\Delta^{i}k) \Big)\Tr^g(D_{(f,m^\bot)}g) \vol(g)\\
&\qquad\qquad+A\sum_{i=0}^{p-1}\int_M\g\Big(\Tr^g\big(R^{\g}(\Delta^{p-i-1}h,\nabla\Delta^{i}k)Tf \big), m^\bot \Big)\vol(g)\\
&\qquad\qquad-A\sum_{i=0}^{p-1}\int_M\g\Big(\Tr^g\big(R^{\g}(\nabla\Delta^{p-i-1}h,\Delta^{i}k)Tf\big) , m^\bot\Big)\vol(g)\\
&\qquad=
-A\sum_{i=0}^{p-1}\int_Mg^0_2\Big(D_{(f,m^\bot)}g,\g(\nabla\Delta^{p-i-1}h,\nabla\Delta^{i}k )\Big) \vol(g)\\
&\qquad\quad -\frac{A}{2}\sum_{i=0}^{p-1}\int_M 
\Big(\nabla^*\g(\nabla\Delta^{p-i-1}h,\Delta^{i}k) \Big)\Tr^g(D_{(f,m^\bot)}g) \vol(g)\\
&\qquad\qquad+A\sum_{i=0}^{p-1}\int_M\g\Big(\Tr^g\big(R^{\g}(\Delta^{p-i-1}h,\nabla\Delta^{i}k)Tf \big), m^\bot \Big)\vol(g)\\
&\qquad\qquad-A\sum_{i=0}^{p-1}\int_M\g\Big(\Tr^g\big(R^{\g}(\nabla\Delta^{p-i-1}h,\Delta^{i}k)Tf\big) , m^\bot\Big)\vol(g)\\
&\qquad=
-A\sum_{i=0}^{p-1}\int_Mg^0_2\Big(-2.\g(m^\bot,S),\g(\nabla\Delta^{p-i-1}h,\nabla\Delta^{i}k )\Big) \vol(g)\\
&\qquad\qquad -\frac{A}{2}\sum_{i=0}^{p-1}\int_M \Big(\nabla^*\g(\nabla\Delta^{p-i-1}h,\Delta^{i}k) \Big)\Tr^g\big(-2.\g(m^\bot,S)\big) \vol(g)\\
&\qquad\qquad+A\sum_{i=0}^{p-1}\int_M\g\Big(\Tr^g\big(R^{\g}(\Delta^{p-i-1}h,\nabla\Delta^{i}k)Tf \big),m^\bot \Big)\vol(g)\\
&\qquad\qquad-A\sum_{i=0}^{p-1}\int_M\g\Big(\Tr^g\big(R^{\g}(\nabla\Delta^{p-i-1}h,\Delta^{i}k)Tf\big) , m^\bot\Big)\vol(g)
\\&\qquad=
\int_M \g\Big(m^\bot,2A\sum_{i=0}^{p-1}\Tr\big(g\i S g\i \g(\nabla\Delta^{p-i-1}h,\nabla\Delta^{i}k ) \big)\Big)\\&\qquad\qquad
+\int_M \g\Big(m^\bot,A\sum_{i=0}^{p-1} \big(\nabla^*\g(\nabla\Delta^{p-i-1}h,\Delta^{i}k) \big) \Tr^g(S)\Big) \vol(g)\\
&\qquad\qquad+A\sum_{i=0}^{p-1}\int_M\g\Big(\Tr^g\big(R^{\g}(\Delta^{p-i-1}h,\nabla\Delta^{i}k)Tf \big), m^\bot \Big)\vol(g)\\
&\qquad\qquad-A\sum_{i=0}^{p-1}\int_M\g\Big(\Tr^g\big(R^{\g}(\nabla\Delta^{p-i-1}h,\Delta^{i}k)Tf\big) , m^\bot\Big)\vol(g).
\end{align*} 
} 
From this, one can read off the normal part of the adjoint. 
Thus one gets:
\begin{lem*}
The adjoint of $\nabla P$ defined in section~\ref{so:ad} for the operator $P=1+A\De^p$ is
\begin{align*}
\adj{\nabla P}(h,k)&=
2A\sum_{i=0}^{p-1}\Tr\big(g\i S g\i \g(\nabla\Delta^{p-i-1}h,\nabla\Delta^{i}k ) \big)
\\&\qquad
+A\sum_{i=0}^{p-1} \big(\nabla^*\g(\nabla\Delta^{p-i-1}h,\Delta^{i}k) \big) \Tr^g(S)\\
&\qquad+A\sum_{i=0}^{p-1}\Tr^g\big(R^{\g}(\Delta^{p-i-1}h,\nabla\Delta^{i}k)Tf \big)\\
&\qquad-A\sum_{i=0}^{p-1}\Tr^g\big(R^{\g}(\nabla\Delta^{p-i-1}h,\Delta^{i}k)Tf \big)
\\&\qquad
+Tf.\Big[\grad^g \g(Ph,k)-\big(\g(Ph,\nabla k)+\g(\nabla h,Pk)\big)^\sharp\Big].
\end{align*}
\end{lem*}
\section{Geodesic equations and conserved momentum}\label{la:ge}

The shortest and most convenient formulation of the geodesic equation is in terms 
of the momentum $p=(1+A\De^p)f_t \otimes \vol(g)$, see sections~\ref{so:gemo} and \ref{so:geshmo}.

\begin{thm*}
The geodesic equation on $\Imm(M,N)$ for the $G^P$-metric 
with $P=1+A \Delta^p$ is given by:
$$\left\{\begin{aligned}
p &= (1+A \De^p)f_t \otimes \vol(g), \\
\nabla_{\p_t}p&=\Bigg(
A\sum_{i=0}^{p-1}\Tr\big(g\i S g\i \g(\nabla(\Delta^{p-i-1}f_t),\nabla\Delta^{i}f_t ) \big)\\&\quad
+\frac{A}{2}\sum_{i=0}^{p-1}\big(\nabla^*\g(\nabla(\Delta^{p-i-1}f_t),\Delta^{i}f_t) \big).\Tr^g(S)\\&\quad
+2A\sum_{i=0}^{p-1}\Tr^g\big(R^{\g}(\Delta^{p-i-1}f_t,\nabla\Delta^{i}f_t)Tf\big)\\&\quad
-\frac12\g(Pf_t,f_t) \Tr^g(S)
-Tf.\g(Pf_t,\nabla f_t)^\sharp\Bigg) \otimes \vol(g).
\end{aligned}\right.$$
This equation is well-posed by theorem \ref{so:we} since all 
conditions are satisfied.
For the special case of plane curves, this agrees with the geodesic equation calculated in 
\cite[section~4.2]{Michor107}.
\end{thm*}

$P=1+A\De^p$ and consequently $G^P$ are invariant under the action of the reparametrization group
$\Diff(M)$. According to sections~\ref{sh:mo} and \ref{so:mo} one gets:
\begin{thm*}
The momentum mapping for the action of $\Diff(M)$ on $\Imm(M,N)$
$$g\Big(\big((1+A\De^p)f_t \big)^\top\Big) \otimes \vol(g)\in \Ga(T^*M\otimes_M\vol(M))$$
is constant along any geodesic $f$ in $\Imm(M,N)$.
\end{thm*}

The horizontal geodesic equation for a general metric on $\Imm$ has been derived in 
section~\ref{sh:geshmo}. In section~\ref{so:geshmo} it has been shown that this equation 
takes a very simple form. Now it is possible to write down this equation specifically for 
the operator $P=1+A\De^p$:
\begin{thm*}
The geodesic equation on shape space for the Sobolev-metric $G^P$ with $P=1+A\De^p$ is equivalent 
to the set of equations
\begin{equation*}
\left\{\begin{aligned}
p &= Pf_t \otimes \vol(g), \qquad Pf_t = (Pf_t)^\bot, \\
(\nabla_{\p_t}p)^\hor &= 
\Bigg(A\sum_{i=0}^{p-1}\Tr\big(g\i S g\i \g(\nabla\Delta^{p-i-1}f_t,\nabla\Delta^{i}f_t ) \big)
\\&\qquad+ \frac{A}{2}\sum_{i=0}^{p-1} \big(\nabla^*\g(\nabla\Delta^{p-i-1}f_t,\Delta^{i}f_t) \big) \Tr^g(S)
\\&\qquad+2A\sum_{i=0}^{p-1}\Tr^g\big(R^{\g}(\Delta^{p-i-1}f_t,\nabla\Delta^{i}f_t)Tf\big)
\\&\qquad-\frac12\g(Pf_t,f_t).\Tr^g(S) \Bigg) \otimes \vol(g),
\end{aligned}\right.
\end{equation*}
where $f$ is a path of immersions. 
For the special case of plane curves, this agrees with the geodesic equation calculated in 
\cite[section~4.6]{Michor107}.
\end{thm*}


\chapter{Surfaces in $n$-space}\label{su}
 
This section is about the special case where the ambient space $N$ is $\R^n$. 
The flatness of $\R^n$ leads to a simplification of the geodesic equation, and the 
Euclidean motion group acting on $\R^n$ induces additional conserved quantities. 
The vector space structure of $\R^n$ allows to define a 
Fr\'echet metric. 
This metric will be compared to Sobolev metrics. 
Finally in section~\ref{su:co} the space of concentric hyper-spheres in $\R^n$
is briefly investigated. 

Most of the material presented here can also be found in \cite{Michor119}.

\section{Geodesic equation}\label{su:ge}

The covariant derivative $\nabla^{\g}$ on $\R^n$ is but the usual derivative.
Therefore the covariant derivatives $\nabla_{\p_t} f_t$ and $\nabla_{\p_t}p$ in the geodesic 
equation can be replaced by $f_{tt}$ and $p_t$, respectively.
(Note that $\Imm(M,\R^n)$ is an open subset of the Fr\'echet vector space 
$C^\infty(M,\R^n)$.)
Also, the curvature terms disappear because $\R^n$ is flat.
Any of the formulations of the geodesic equation presented so far can thus be
adapted to the case $N=\R^n$. 

We want to show how the geodesic equation simplifies further under the additional assumptions 
that $\dim(M)=\dim(N)-1$ and that $M$ is orientable. Then it is possible define a unit
vector field $\nu$ to $M$. The condition that $f_t$ is horizontal then simplifies 
to $Pf_t = a.\nu$ for $a \in C^\infty(M)$. 
The geodesic equation can then be written as an equation for $a$. 
However, the equation is slightly simpler 
when it is written as an equation for $a.\vol(g)$. 
In practise, $\vol(g)$ can be treated as a function on $M$ because
one can identify $\vol(g)$ with its density with respect to $du^1 \wedge \ldots \wedge du^{n-1}$, 
where $(u^1, \ldots, u^{n-1})$ is a chart on $M$. Thus multiplication by $\vol(g)$ 
does not pose a problem. 

\begin{thm*}
The geodesic equation for a Sobolev-type metric $G^P$ on shape space 
$B_i(M,\R^n)$ with $\dim(M)=n-1$
is equivalent to the set of equations
\begin{equation*}
\left\{\begin{aligned}
Pf_t  &= a.\nu \\
\p_t\big(a.\vol(g)\big) &= \frac12 \g\big(\adj{\nabla P}(f_t,f_t),\nu\big) 
-\frac12 \g(Pf_t,f_t) \g\big(\Tr^g(S),\nu\big),
\end{aligned}\right.
\end{equation*}
where $f$ is a path in $\Imm(M,\R^n)$ and $a$ is a time-dependent function on $M$. 
\end{thm*}
\begin{proof}
Applying $\g(\cdot,\nu)$ to the geodesic equation \ref{so:geshmo} 
on shape space in terms of the momentum one gets
\begin{align*}
\p_t\big(a.\vol(g)\big) &= 
\p_t\ \g\big(Pf_t \otimes \vol(g),\nu\big) \\&=
\g\Big(\nabla_{\p_t}\big(Pf_t \otimes \vol(g)\big),\nu\Big) 
+ \g\big(Pf_t \otimes \vol(g),\nabla_{\p_t} \nu\big)  \\&=
\frac12 \g\big(\adj{\nabla P}(f_t,f_t),\nu\big) -\frac12 \g(Pf_t,f_t) \g\big(\Tr^g(S),\nu\big)+ 0.
\qedhere
\end{align*}
\end{proof}
Let us spell this equation out in even more details for the $H^1$-metric.
This is the case of interest for the numerical examples in section~\ref{nu}. 
\begin{thm*}
The geodesic equation on shape space $B_i(M,\R^n)$ 
for the Sobolev-metric $G^P$ with $P=1+A\De$ is equivalent 
to the set of equations
\begin{equation*}
\left\{\begin{aligned}
Pf_t &=  a.\nu \\
\p_t \big(a.\vol(g)\big) &= 
\Big(A g^0_2\big(s, \g(\nabla f_t,\nabla f_t ) \big) 
-\frac{\Tr(L)}{2} \big(\norm{f_t}_{\g}^2 + A \norm{\nabla f_t}_{g^0_1\otimes\g}^2 \big) 
\Big) \vol(g),
\end{aligned}\right.
\end{equation*}
where $f$ is a path of immersions, 
$a$ is a time-dependent function on $M$, 
$s=\g(S,\nu) \in \Ga(T^0_2 M)$ is the shape operator,
$L=g\i s \in \Ga(T^1_1 M)$ is the Weingarten mapping, 
and $\Tr(L)$ is the mean curvature.
\end{thm*}
\begin{proof}
The fastest way to get to this equation is to apply $\g(\cdot,\nu)$ to
the geodesic equation on $\Imm$ from section~\ref{la:ge}.
This yields
\begin{align*}
\p_t \big(a.\vol(g)\big) &= 
\Big(A \Tr\big(g\i.s.g\i \g(\nabla f_t,\nabla f_t ) \big)
+ \frac{A}{2} \big(\nabla^*\g(\nabla f_t,f_t) \big) \Tr(L)\\&\qquad
-\frac12\g(Pf_t,f_t).\Tr(L) \Big) \vol(g)\\&=
\Big(A g^0_2\big(s, \g(\nabla f_t,\nabla f_t ) \big)
- \frac{A}{2} \Tr^g\big(\g(\nabla^2 f_t,f_t) \big) \Tr(L)\\&\qquad
- \frac{A}{2} \Tr^g\big(\g(\nabla f_t,\nabla f_t) \big) \Tr(L)
-\frac12\g\big((1+A\De)f_t,f_t\big)\Tr(L) \Big) \vol(g)\\&=
\Big(A g^0_2\big(s, \g(\nabla f_t,\nabla f_t ) \big)
- \frac{A}{2} \Tr^g\big(\g(\nabla f_t,\nabla f_t) \big) \Tr(L)\\&\qquad
-\frac12\g\big(f_t,f_t\big).\Tr(L) \Big) \vol(g).
\end{align*}
Notice that the second order derivatives of $f_t$ have canceled out.
\end{proof}

\section{Additional conserved momenta}
 
If $P$ is invariant under the action of the Euclidean motion group $\R^n\rtimes\on{SO}(n)$, 
then also the metric $G^P$ is in invariant under this group action 
and one gets additional conserved quantities as described in section~\ref{sh:mo}:

\begin{thm*}
For an operator $P$ that is invariant under the action of the 
Euclidean motion group $\R^n\rtimes\on{SO}(n)$, 
the linear momentum
$$ \int_M  Pf_t \vol(g)\in(\R^n)^* $$
and the angular momentum
\begin{align*}
\forall X\in \mathfrak{so}(n): \int_M \g( X.f,Pf_t ) \vol(g) \\
\text{or equivalently } \int_M (f\wedge Pf_t ) 
\vol(g)\in\textstyle{\bigwedge^2}\R^n\cong \mathfrak{so}(n)^* 
\end{align*}
are constant along any geodesic $f$ in $\Imm(M,\R^n)$.
The operator $P=1+A \De^p$ satisfies this property. 
\end{thm*}

These momenta have also been calculated in \cite[section~4.3]{Michor119}. 

\section{Fr\'echet distance and Finsler metric}\label{su:fr}

This section can also be found in \cite[section~5.6]{Michor119}.

The Fr\'echet distance on shape space $B_i(M,\R^n)$ is defined as
\begin{align*}
\dist_\infty^{B_i}(F_0,F_1) = \inf_{f_0,f_1} \norm{f_0 - f_1}_{L^\infty},
\end{align*}
where the infimum is taken over all $f_0, f_1$ with $\pi(f_0)=F_0, \pi(f_1)=F_1$.
As before, $\pi$ denotes the projection $\pi:\Imm \to B_i$. 
Fixing $f_0$ and $f_1$, one has
\begin{align*}
\dist_\infty^{B_i}\big(\pi(f_0),\pi(f_1)\big) = \inf_{\varphi} \norm{f_0 \o \varphi - f_1}_{L^\infty},
\end{align*}
where the infimum is taken over all $\varphi \in \on{Diff}(M)$.
The Fr\'echet distance is related to the Finsler metric
\begin{align*}
G^\infty : T \Imm(M,\R^n) \rightarrow \R, \qquad h \mapsto \norm{h^\bot}_{L^\infty}.
\end{align*}

\begin{lem*}
The pathlength distance induced by the Finsler metric $G^\infty$ 
provides an upper bound for the Fr\'echet distance:
\begin{align*}
\dist_\infty^{B_i}(F_0,F_1) \leq \dist_{G^\infty}^{B_i}(F_0,F_1) = \inf_f \int_0^1 \norm{f_t}_{G^\infty} dt,
\end{align*}
where the infimum is taken over all paths 
$$f:[0,1] \to \Imm(M,\R^n) \quad \text{with} \quad \pi(f(0))=F_0, \pi(f(1))=F_1.$$
\end{lem*}

\begin{proof}
Since any path $f$ can be reparametrized such that $f_t$ is normal to $f$, one has
$$\inf_f \int_0^1 \norm{f_t^\bot}_{L^\infty} dt = \inf_f \int_0^1 \norm{f_t}_{L^\infty} dt, $$
where the infimum is taken over the same class of paths $f$ as described above. Therefore
\begin{align*}
\dist_\infty^{B_i}(F_0,F_1) &= \inf_f \norm{f(1)-f(0)}_{L^\infty} 
= \inf_f \norm{ \int_0^1 f_t dt}_{L^\infty} 
\leq \inf_f \int_0^1 \norm{f_t}_{L^\infty} dt \\ &
= \inf_f \int_0^1 \norm{f_t^\bot}_{L^\infty} dt
= \dist_{G^\infty}^{B_i}(F_0,F_1). \qedhere
\end{align*}
\end{proof}


It is claimed in \cite[theorem~13]{MennucciYezzi2008} that $d_\infty=\dist_{G^\infty}$. However, 
the proof given there only works on the vector space $C^\infty(M,\R^n)$ and not on $B_i(M,\R^n)$. 
The reason is that convex combinations of immersions are used in the proof, 
but that the space of immersions is not convex.


\section{Sobolev versus Fr\'echet distance}\label{su:fr2}

This section can also be found in \cite[section~5.7]{Michor119}. 

It is a desirable property of any distance on shape space to be stronger than the Fr\'echet
distance. Otherwise, singular points of a shape could move arbitrarily far away
without increasing the distance much. 

As the following result shows, Sobolev metrics of low order do not have this property. 
The author and the authors of \cite{Michor119} believe that they have this property 
when the order is high enough, but were not able to prove this.

\begin{lem*}
Let $G^P$ be a metric on $B_i(M,\R^n)$ that is weaker than or at least as weak as a Sobolev $H^p$-metric with 
$p < \frac{\dim(M)}2+1$, i.e.
\begin{align*}
\norm{h}_{G^P} \leq C \norm{h}_{H^p} = C \sqrt{\int_M \g\big( (1+\Delta^p) h,h \big) \vol(g)} 
\qquad \text{for all $h \in T\Imm$.}
\end{align*}
Then the Fr\'echet distance can not be bounded by the $G^P$-distance. 
\end{lem*}

\begin{proof}
It is sufficient to prove the claim for $P=1+\Delta^p$.
Let $f_0$ be a fixed immersion of $M$ into $\R^n$, 
and let $f_1$ be a translation of $f_0$ by a vector $h$ of length $\ell$. 
It will be shown that the $H^p$-distance between $\pi(f_0)$ and $\pi(f_1)$ 
is bounded by a constant $2L$ that does 
not depend on $\ell$, where $\pi$ denotes the projection of $\Imm$ onto $B_i$. 
Then it follows that the $H^p$-distance can not be bounded from below by the Fr\'echet distance, 
and this proves the claim. 

For small $r_0$, one calculates the $H^p$-length of the following path of immersions:
First scale  $f_0$ by a factor $r_0$, then translate it by $h$, and then
scale it again until it has reached $f_1$. 
The following calculation  shows that under the assumption $p<m/2+1$ 
the immersion $f_0$ can be scaled down to  zero in finite $H^p$-pathlength $L$.
Let $r: [0,1] \to [0,1]$ be a function of time with $r(0)=1$ and $r(1)=0$. 
\begin{align*}
L_{\Imm}^{G^P}\big(r.f_0\big)&=\int_0^1\sqrt{\int_M\g\Big(r_t.\big(1+(\Delta^{(r.f_0)^*\g})^p\big)(f_0),r_t.f_0\Big)\vol\big((r.f_0)^*\g\big)}dt\\
&=\int_0^1\sqrt{\int_M r^2_t.\g\Big(\big(1+\frac{1}{r^{2p}}(\Delta^{f_0^*\g})^p\big)(f_0),f_0\Big)r^m\vol\big(f_0^*\g\big)}dt\\
&=\int_1^0\sqrt{\int_M\g\Big(\big(1+\frac{1}{r^{2p}}(\Delta^{f_0^*\g})^p\big)(f_0),f_0\Big)r^m\vol\big(f_0^*\g\big)}dr =: L
\end{align*}
The last integral converges if $\frac{m-2p}{2}<-1$, which holds by assumption. 
Scaling down to $r_0>0$ needs even less effort. 
So one sees that the length of the shrinking and growing part of the path is bounded by $2L$. 

The length of the translation is simply $\ell \sqrt{r_0^m \Vol(f_0)}=O(r^{m/2})$ 
since the Laplacian of the constant vector field vanishes. Therefore 
\begin{equation*}
\dist_{B_i}^{G^P}\big(\pi(f_0),\pi(f_1)\big) \leq \dist_{\Imm}^{G^P}(f_0,f_1) \leq 2L. \qedhere
\end{equation*}
\end{proof}

\section{Concentric spheres}\label{su:co}

This section can also be found in \cite[section~6.6]{Michor119}.

For a Sobolev type metric $G^P$ that is invariant under the action of the $SO(n)$ on $\R^n$, 
the set of hyper-spheres in $\R^n$ with common center $0$
is a totally geodesic subspace of $B_i(S^{n-1},\R^n)$. 
The reason is that it is the fixed point set 
of the group $SO(n)$ acting on $B_i$ isometrically. 
(One also needs uniqueness of solutions to the geodesic equation to prove that 
the concentric spheres are totally geodesic.)
This section mainly deals with the case $P=1+\De^p$. 

First we want to determine under what conditions the set of concentric spheres is geodesically complete
under the $G^P$-metric. 

\begin{lem*}
The space of concentric spheres is complete with respect to the $G^P$ 
metric with $P=1+A\Delta^p$ iff $p\geq(n+1)/2$. 
\end{lem*}

\begin{proof}
The space is complete
if and only if it is impossible to scale a sphere down to zero or up to 
infinity in finite $G^P$ path-length. 
So let $f$ be a path of concentric spheres. 
It is uniquely described by its radius $r$. Its velocity is $f_t=r_t.\nu$, where $\nu$ 
designates the unit normal vector field. 
One has
$$\g\big(g\i.S,\nu\big)=:L=-\tfrac{1}{r}\on{Id}_{TM}, 
\quad \Tr(L^k)=(-1)^k\tfrac{n-1}{r^k},\quad \Vol=r^{n-1}\tfrac{n\pi^{n/2}}{\Gamma(n/2+1)}.$$
Keep in mind that $r$ and $r_t$ are constant functions on the sphere, so that all 
derivatives of them vanish.
Therefore 
\begin{align*}
\Delta \nu&=\nabla^*(\nabla \nu)=\nabla^*(-Tf.L)=\Tr^g\Big(\nabla(Tf.L)\Big)
\\&=\Tr^g\Big(\nabla(Tf).L\Big)+\Tr^g\Big(Tf.(\nabla L)\Big)\\&=
\Tr(L^2).\nu+ \Tr^g\Big(Tf.\nabla(-\tfrac1r\on{Id}_{TM})\Big)=\frac{n-1}{r^2}.\nu+0
\end{align*}
and
\begin{align*}
Pf_t&=(1+A\Delta^p)(r_t.\nu)=r_t.\left(1+A\frac{(n-1)^p}{r^{2p}}\right).\nu.
\end{align*}
From this it is clear that the path $f$ is horizontal. Therefore its length as a path in $B_i$ 
is the same as its length as a path in $\Imm$. 
One calculates its length as in the proof of \ref{su:fr2}:
\begin{align*}
L^{G^P}_{B_i}(f)&=\int_0^1\sqrt{G_f^P(f_t,f_t)}dt=
\int_0^1\sqrt{\int_M r_t^2.\left(1+A\frac{(n-1)^p}{r^{2p}}\right)\vol(g)}dt\\&
=\int_0^1 |r_t|\sqrt{ \left(1+A\frac{(n-1)^p}{r^{2p}}\right)\frac{n.\pi^{n/2}}{\Gamma(n/2+1)}r^{n-1}}dt\\&
=\sqrt{\frac{n.\pi^{n/2}}{\Gamma(n/2+1)}}\int_{r_0}^{r_1} \sqrt{\left(1+A\frac{(n-1)^p}{r^{2p}}\right)r^{n-1}}dr.
\end{align*}
The integral diverges for $r_1 \to \infty$ since the integrand is greater than $r^{(n-1)/2}$. 
It diverges for $r_0 \to 0$ iff $(n-1-2p)/2 \leq -1$, which is equivalent to $p \geq (n+1)/2$. 
\end{proof}

The geodesic equation within the space of concentric spheres 
reduces to an ODE for the radius that can be read off 
the geodesic equation in section \ref{la:ge}:
\begin{align*}
r_{tt}=-r_t^2\Big(\frac{n-1}{2r}-\frac{p.A.(n-1)^p}{r\big(r^{2p}+A(n-1)^p\big)}\Big).
\end{align*}

\chapter{Diffeomorphism groups}\label{di}

For $M=N$ the space $\Emb(M,M)$ equals the \emph{diffeomorphism group of $M$}. 
An operator $P \in \Ga\big(L(T\Emb;T\Emb)\big)$ that is invariant under reparametrizations
induces a right-invariant Riemannian metric on this space. 
Thus one gets the geodesic equation for 
right-invariant Sobolev metrics on diffeomorphism groups and well-posedness of this equation.
To the authors knowledge, well-posedness has so far only been shown for the special case 
$M=S^1$ in \cite{Constantin2003}
and for the special case of Sobolev order one metrics in \cite{GayBalmaz2009}.
Theorem~\ref{so:we} establishes this result for arbitrary compact $M$ and
Sobolev metrics of arbitrary order.

In the decomposition of a vector $h \in T_f\Emb$ into its tangential and normal components 
$h=Tf.h^\top + h^\bot$, the normal part $h^\bot$ vanishes.  
By the naturality of the covariant derivative, $S=\nabla Tf$ vanishes as well. Thus on gets: 

\begin{thm*}
The geodesic equation on $\Diff(M)$ is given by
\begin{equation*}
\left\{\begin{aligned}
p &= Pf_t \otimes \vol(g), \\
\nabla_{\p_t}p &=-Tf.\g(Pf_t,\nabla f_t)^\sharp \otimes \vol(g).
\end{aligned}\right.
\end{equation*}
\end{thm*}

Note that this equation is not right-trivialized, in contrast to 
the equation given in \cite{Arnold1966,Michor102,Michor109}, for example.




\chapter{Numerical results}\label{nu}

This section can also be found in \cite[section~7]{Michor119}.

It is of great interest for shape comparison to solve the \emph{boundary value problem} 
for geodesics in shape space. 
When the boundary value problem can be solved, 
then any shape can be encoded as the initial momentum of a geodesic 
starting at a fixed reference shape. 
Since the initial momenta all lie in the same vector space, 
this also opens the way to statistics on shape space. 

There are two approaches to solving the boundary value problem.
In \cite{Michor118} the first approach of minimizing \emph{horizontal path energy} 
over the set of curves in $\Imm$ 
connecting two fixed boundary shapes has been taken. 
This has been done for several almost local metrics. 
For these metrics it is straightforward to calculate the horizontal energy because 
the horizontal bundle equals the normal bundle.  
However, in the case of Sobolev type metrics 
the horizontal energy involves the inverse of a differential operator (see section~\ref{so:ho}), 
which makes this approach much harder.

The second approach is the method of \emph{geodesic shooting}. 
This method is based on iteratively solving the initial value problem 
while suitably adapting the initial conditions. 
The theoretical requirements of existence of solutions to the geodesic equation and smooth dependence on 
initial conditions are met for Sobolev type metrics, see section \ref{so:we}.
This makes geodesic shooting a promising approach.

{\it The first step towards this aim is to numerically solve the initial value problem for geodesics, 
at least for the $H^1$-metric and the case of surfaces in $\R^3$, 
and that is what will be presented in this work. 
}

The geodesic equation on shape space is equivalent to the horizontal geodesic equation
on the space of immersions. For the case of surfaces in $\R^3$, 
it takes the form given in section~\ref{su:ge}.
This equation can be conveniently set up
using the DifferentialGeometry package incorporated in the computer algebra system Maple
as demonstrated in figure~\ref{nu:maple}. 
(The equations that have actually been solved were simplified 
by multiplying intermediate terms with suitable powers of 
$\sqrt{\vol(g)}$, but for the sake of clearness this has not been included in the Maple code
in figure~\ref{nu:maple}.)

\begin{figure}[p]
\lstset{basicstyle=\small\ttfamily,
	backgroundcolor=\color[gray]{0.9},
	numbers=left, numberstyle=\scriptsize, stepnumber=2, numbersep=5pt,
  	commentstyle=\small,columns=flexible,
  	showstringspaces=false}
\begin{lstlisting}
with(DifferentialGeometry);with(Tensor);with(Tools);
DGsetup([u,v],[x,y,z],E);
f := evalDG(f1(t,u,v)*D_x+f2(t,u,v)*D_y+f3(t,u,v)*D_z);
G := evalDG(dx &t dx + dy &t dy + dz &t dz);
Gamma_vrt := 0 &mult Connection(dx &t D_x &t du);
Tf := CovariantDerivative(f,Gamma_vrt);
g:=ContractIndices(G &t Tf &t Tf,[[1,3],[2,5]]);
g_inv:=InverseMetric(g);
Gamma_bas:=Christoffel(g);
det:=Hook([D_u,D_u],g)*Hook([D_v,D_v],g)-Hook([D_u,D_v],g)^2;
ft := evalDG(diff(f1(t,u,v),t)*D_x+diff(f2(t,u,v),t)*D_y
	+diff(f3(t,u,v),t)*D_z);
ft := convert(ft,DGtensor);
S := CovariantDerivative(Tf,Gamma_vrt,Gamma_bas);
cross:=evalDG((dy &w dz) &t D_x + (dz &w dx) &t D_y + (dx &w dy) &t D_z);
N:=Hook([ContractIndices(Tf &t D_u,[[2,3]]),
	ContractIndices(Tf &t D_v,[[2,3]])],cross);
nu:=convert(N/sqrt(Hook([N,N],G)),DGvector);
s := ContractIndices(G &t S &t nu, [[1,3],[2,6]]);
L := ContractIndices(g_inv &t s,[[2,3]]);
Gftft := ContractIndices(G &t ft &t ft,[[1,3],[2,4]]);
Cft := CovariantDerivative(ft,Gamma_vrt);
CCft := CovariantDerivative(Cft,Gamma_vrt,Gamma_bas);
Dft := ContractIndices(-g_inv &t CCft,[[1,4],[2,5]]);
Pft := evalDG(ft+A*Dft);
GCftCft := ContractIndices(Cft &t Cft &t G,[[1,5],[3,6]]);
gGCftCft := ContractIndices(g_inv &t GCftCft,[[2,3]]);
TrLgGCftCft := ContractIndices(L &t gGCftCft,[[1,4],[2,3]]);
TrgGCftCft := ContractIndices(gGCftCft,[[1,2]]);
TrL := ContractIndices(L,[[1,2]]);
# b(t,u,v) := a(t,u,v)*sqrt(det);
eq1 := ContractIndices(Pft &t dx,[[1,2]])*sqrt(det) = 
	b(t,u,v)*ContractIndices(nu &t dx,[[1,2]]);
eq2 := ContractIndices(Pft &t dy,[[1,2]])*sqrt(det) = 
	b(t,u,v)*ContractIndices(nu &t dy,[[1,2]]);
eq3 := ContractIndices(Pft &t dz,[[1,2]])*sqrt(det) = 
	b(t,u,v)*ContractIndices(nu &t dz,[[1,2]]);
eq4 := diff(b(t,u,v),t) = 
	(A*TrLgGCftCft - TrL/2*(Gftft+A*TrgGCftCft))*sqrt(det);
\end{lstlisting}
\caption{Maple source code to set up the geodesic equation.}
\label{nu:maple}
\end{figure}

Unfortunately, Maple (as of version 14) is not able to solve 
PDEs with more than one space variable numerically. 
Thus the equations were translated into Mathematica.
The PDE was solved using the method of lines. Spatial discretization was done 
using an equidistant grid, and spatial derivatives were 
replaced by finite differences.
The time-derivative $f_t$ appears implicitly in the equation 
$P_f(f_t)=a.\nu$, and this remains so 
when the operator $P_f$ is replaced by finite differences.

The solver that has been used is 
the Implicit Differential-Algebraic 
(IDA) solver that is part of the SUNDIALS suite and is integrated into Mathematica. 
IDA uses backward differentiation of order 1 to 5 with 
variable step widths. Order 5 is standard and has also been used here. 
At each time step, the new value of $f_t$ is computed using some previous values of $f$, and then 
the new value of $f$ is calculated from the equation $P_f( f_t)=a.\nu$.  
The dependence on $f$ in this equation is of course highly nonlinear.
A Newton method is used to solve it. 
This operation is quite costly and has to be done at every step, 
which is a main disadvantage of backward differentiation algorithms.
Explicit methods are probably much better adapted
to the problem. The implementation of an explicit solver 
is ongoing work of Martin Bauer, Martins Bruveris and the author.
The results are too preliminary to be included in this thesis. 


In the examples that follow, $f$ at time zero is a square $[0,\pi]\x[0,\pi]$ flatly embedded in $\R^3$. 
This is a manifold with boundary, but it can be seen as a part of a bigger closed manifold. 
Zero boundary conditions are used for both $f$ and $a$. 
It remains to specify an initial condition for $a$. 
As a first example, let us assume that $a$ at time zero equals $\sin(x)\sin(y)$, 
where $x,y$ are the Euclidean coordinates on the square.
The resulting geodesic is depicted in figure~\ref{nu:bump_1}. 
In the absence of a closed-form solution of the geodesic equation, one way to 
check if the solution is correct is to see if the energy $G^P(f_t,f_t)$ 
is conserved. Figure~\ref{nu:energyplot} shows 
this for the geodesic from figure~\ref{nu:bump_1}
with various space and time discretizations. 

A more complicated example of a geodesic is shown in figure \ref{nu:A_ge} and \ref{nu:A_in}. 
There, the initial velocity was chosen to be a smoothened version of a black and white
picture of the letter A. The initial momentum was computed from it using a discrete 
Fourier transform. 

Finally, it is shown in figure \ref{nu:se} that self-intersections of the surface can occur. 
This is not due to a numerical error but part of the theory, 
and can be an advantage or a disadvantage depending 
on the application.

\begin{figure}[p]
\centering
\includegraphics[width=\textwidth-10pt]{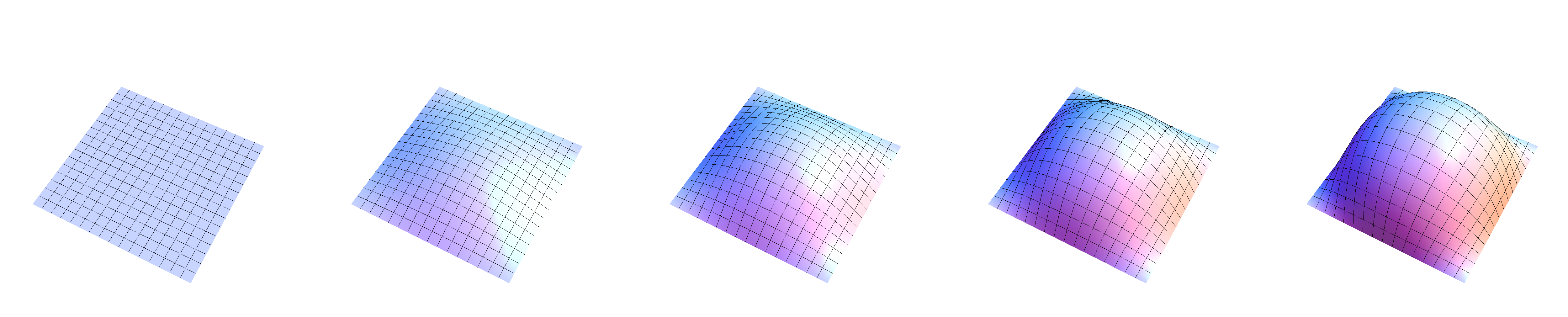}
\caption{Geodesic where a bump is formed out a flat plane. 
The initial momentum is $a=\sin(x)\sin(y)$.
Time increases linearly from left to right. The final time is $t=5$. }
\label{nu:bump_1}
\end{figure}

\begin{figure}[p]
\centering
\begin{psfrags}
\def\PFGstripminus-#1{#1}%
\def\PFGshift(#1,#2)#3{\raisebox{#2}[\height][\depth]{\hbox{%
  \ifdim#1<0pt\kern#1 #3\kern\PFGstripminus#1\else\kern#1 #3\kern-#1\fi}}}%
\providecommand{\PFGstyle}{}%
%
\psfrag{a08220A}[cr][cr]{\PFGstyle $0.8220$}%
\psfrag{a08220B}[Bc][Bc]{\PFGstyle $0.8220$}%
\psfrag{a08220}[Bc][Bc][1.][0.]{\PFGstyle $0.8220$}%
\psfrag{a08225A}[cr][cr]{\PFGstyle $0.8225$}%
\psfrag{a08225B}[Bc][Bc]{\PFGstyle $0.8225$}%
\psfrag{a08225}[Bc][Bc][1.][0.]{\PFGstyle $0.8225$}%
\psfrag{a08230A}[cr][cr]{\PFGstyle $0.8230$}%
\psfrag{a08230B}[Bc][Bc]{\PFGstyle $0.8230$}%
\psfrag{a08230}[Bc][Bc][1.][0.]{\PFGstyle $0.8230$}%
\psfrag{a08235A}[cr][cr]{\PFGstyle $0.8235$}%
\psfrag{a08235B}[Bc][Bc]{\PFGstyle $0.8235$}%
\psfrag{a08235}[Bc][Bc][1.][0.]{\PFGstyle $0.8235$}%
\psfrag{a08240A}[cr][cr]{\PFGstyle $0.8240$}%
\psfrag{a08240B}[Bc][Bc]{\PFGstyle $0.8240$}%
\psfrag{a08240}[Bc][Bc][1.][0.]{\PFGstyle $0.8240$}%
\psfrag{a08245A}[cr][cr]{\PFGstyle $0.8245$}%
\psfrag{a08245B}[Bc][Bc]{\PFGstyle $0.8245$}%
\psfrag{a08245}[Bc][Bc][1.][0.]{\PFGstyle $0.8245$}%
\psfrag{a08250A}[cr][cr]{\PFGstyle $0.8250$}%
\psfrag{a08250B}[Bc][Bc]{\PFGstyle $0.8250$}%
\psfrag{a08250}[Bc][Bc][1.][0.]{\PFGstyle $0.8250$}%
\psfrag{a0A}[Bc][Bc]{\PFGstyle $0$}%
\psfrag{a0}[Bc][Bc][1.][0.]{\PFGstyle $0$}%
\psfrag{a1A}[tc][tc]{\PFGstyle $1$}%
\psfrag{a1B}[Bc][Bc]{\PFGstyle $1$}%
\psfrag{a1}[Bc][Bc][1.][0.]{\PFGstyle $1$}%
\psfrag{a2A}[tc][tc]{\PFGstyle $2$}%
\psfrag{a2B}[Bc][Bc]{\PFGstyle $2$}%
\psfrag{a2}[Bc][Bc][1.][0.]{\PFGstyle $2$}%
\psfrag{a3A}[tc][tc]{\PFGstyle $3$}%
\psfrag{a3B}[Bc][Bc]{\PFGstyle $3$}%
\psfrag{a3}[Bc][Bc][1.][0.]{\PFGstyle $3$}%
\psfrag{a4A}[tc][tc]{\PFGstyle $4$}%
\psfrag{a4B}[Bc][Bc]{\PFGstyle $4$}%
\psfrag{a4}[Bc][Bc][1.][0.]{\PFGstyle $4$}%
\psfrag{a5A}[tc][tc]{\PFGstyle $5$}%
\psfrag{a5B}[Bc][Bc]{\PFGstyle $5$}%
\psfrag{a5}[Bc][Bc][1.][0.]{\PFGstyle $5$}%
\psfrag{eA}[Bc][Bc]{\PFGstyle $G^P(f_t,f_t)$}%
\psfrag{e}[bc][bc]{\PFGstyle $G^P(f_t,f_t)$}%
\psfrag{tA}[Bc][Bc]{\PFGstyle $t$}%
\psfrag{t}[cl][cl]{\PFGstyle $t$}%
\includegraphics[width=\textwidth-10pt]{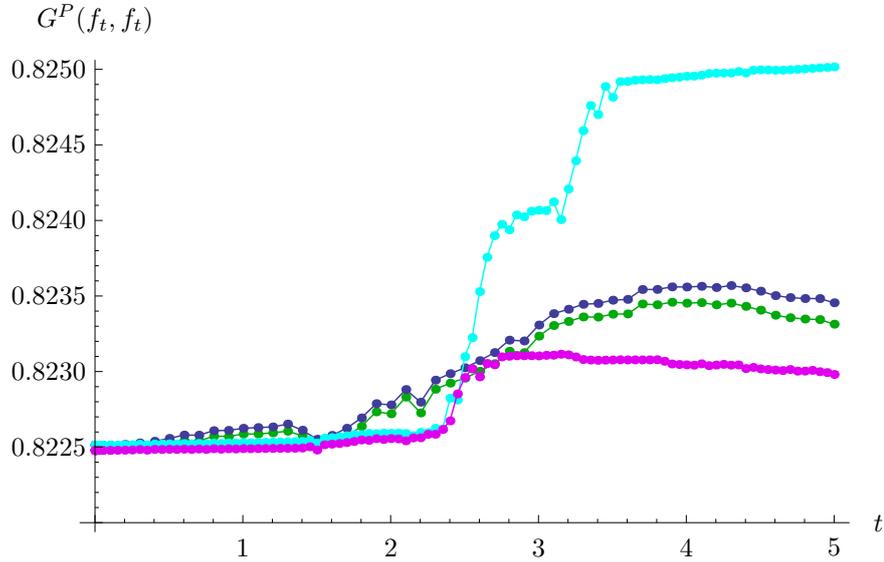}
\end{psfrags}
\caption{Conservation of the energy $G^P(f_t,f_t)$ along the geodesic in figure~\ref{nu:bump_1}. 
The true value of $G^P(f_t,f_t)$ is $\tfrac{\pi^2}{4(1+2A)} \approx 0.822467$ for $A=1$. 
The maximum time step used in blue and green is 0.1. For purple and cyan it is 0.05. The number
of grid points used in blue and cyan is 100 times 100. For green and purple it is 200 times 200.}
\label{nu:energyplot}
\end{figure}

\begin{figure}[p]
\centering
\includegraphics[width=\textwidth-10pt]{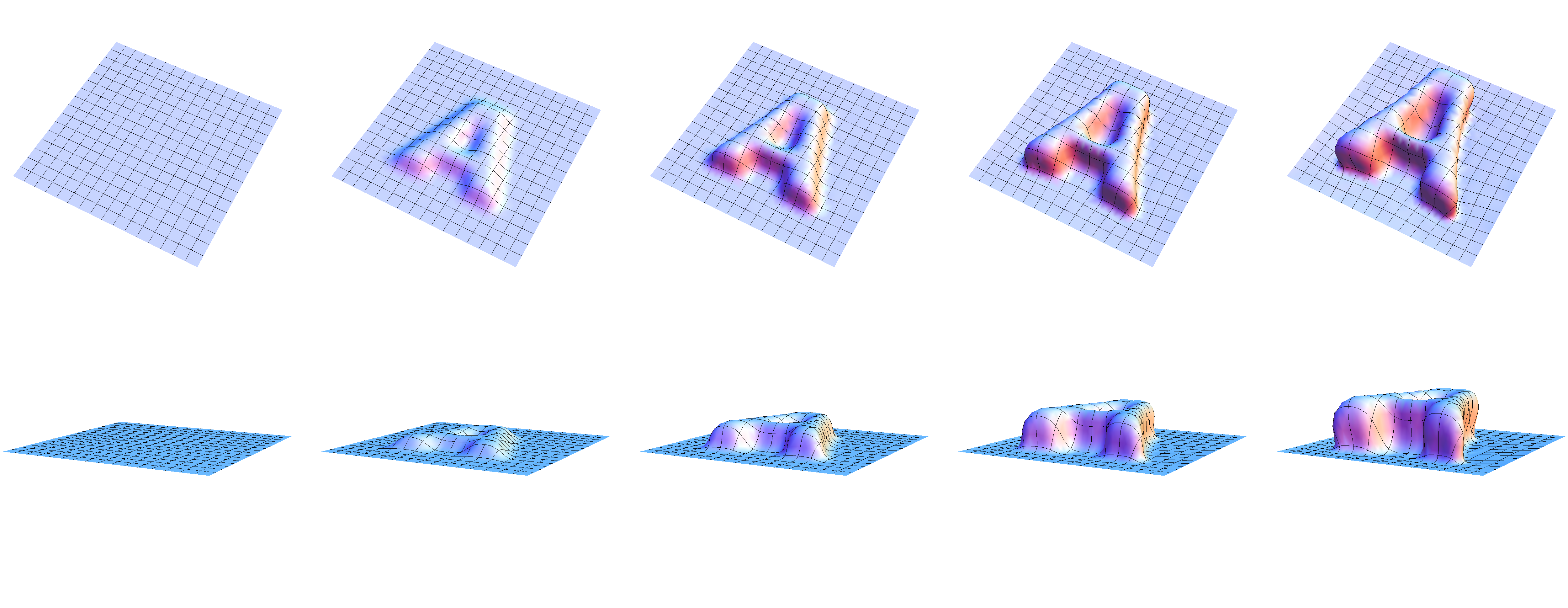}
\caption{Letter A forming along a geodesic path. Time increases linearly from left to right. 
The final time is $t=0.8$.
Top and bottom row are different views of the same geodesic. }
\label{nu:A_ge}
\end{figure}

\begin{figure}[p]
\centering
\includegraphics[width=\textwidth-10pt]{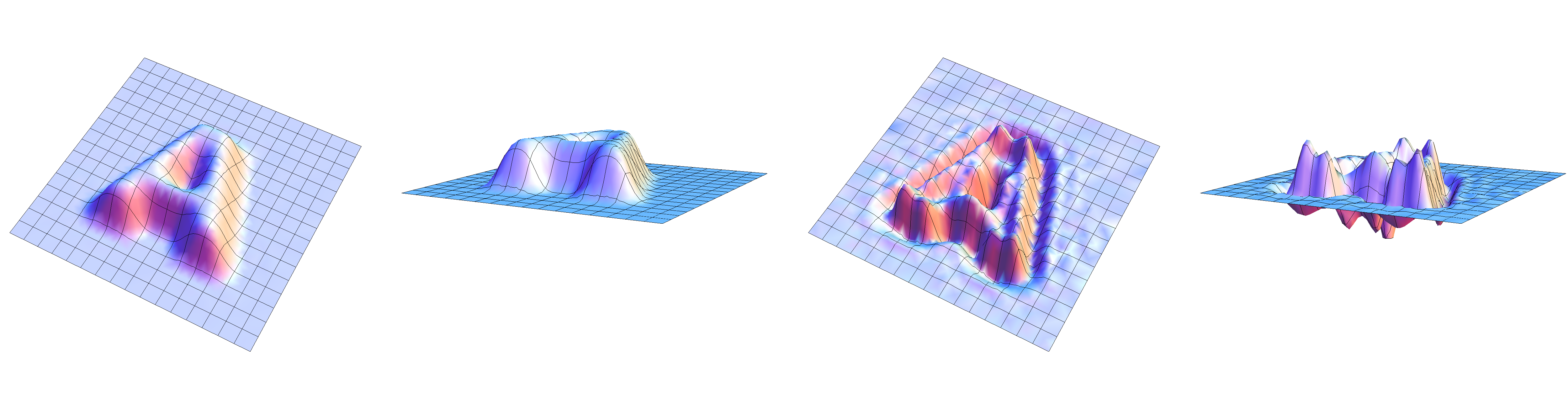}
\caption{Initial velocity $f_t(0,\cdot)$ and momentum $a(0,\cdot)$ of the geodesic in figure \ref{nu:A_ge}. 
Both are shown first from above, then from the side. }
\label{nu:A_in}
\end{figure}

\begin{figure}[p]
\centering
\includegraphics[width=\textwidth-10pt]{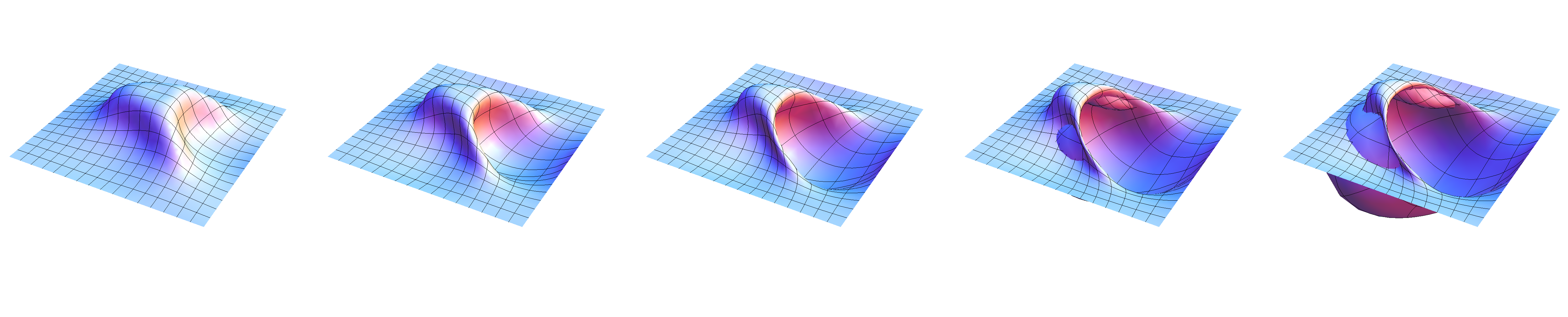}
\caption{A self-intersection forming along a geodesic. Time increases linearly from left to right. }
\label{nu:se}
\end{figure}


\backmatter
\appendix
\addcontentsline{toc}{chapter}{Appendix}

\addcontentsline{toc}{section}{Bibliography}
\markboth{APPENDIX}{BIBLIOGRAPHY}
\bibliographystyle{plainnat}
\bibliography{bibliography}

\begin{thebibliography}{52}
\providecommand{\natexlab}[1]{#1}
\providecommand{\url}[1]{\texttt{#1}}
\expandafter\ifx\csname urlstyle\endcsname\relax
  \providecommand{\doi}[1]{doi: #1}\else
  \providecommand{\doi}{doi: \begingroup \urlstyle{rm}\Url}\fi

\bibitem[Alt and Guibas(1996)]{Alt1996}
H.~Alt and L.~J. Guibas.
\newblock Discrete geometric shapes: Matching, interpolation, and
  approximation: A survey.
\newblock Technical Report B 96-11, EVL-1996-142, Institute of Computer
  Science, Freie Universit\"at Berlin, 1996.

\bibitem[Ambrosio et~al.(2004)Ambrosio, Gigli, and Savar{\'e}]{Ambrosio2004}
L.~Ambrosio, N.~Gigli, and G.~Savar{\'e}.
\newblock Gradient flows with metric and differentiable structures, and
  applications to the wasserstein space.
\newblock \emph{Atti Accad. Naz. Lincei Cl. Sci. Fis. Mat. Natur. Rend. Lincei
  (9) Mat. Appl.}, 15\penalty0 (3-4), 2004.

\bibitem[Ambrosio et~al.(2008)Ambrosio, Gigli, and Savar{\'e}]{Ambrosio2008}
L.~Ambrosio, N.~Gigli, and G.~Savar{\'e}.
\newblock \emph{Gradient flows in metric spaces and in the space of probability
  measures}.
\newblock Lectures in Mathematics ETH Z\"urich. Birkh\"auser Verlag, Basel,
  second edition, 2008.

\bibitem[Arman and Aggarwal(1993)]{Arman1993}
F.~Arman and J.~Aggarwal.
\newblock Model-based object recognition in dense-range images---a review.
\newblock \emph{ACM Comput. Surv.}, 25\penalty0 (1):\penalty0 5--43, 1993.

\bibitem[Arnold(1966)]{Arnold1966}
V.~I. Arnold.
\newblock Sur la g\'eometrie diff\'erentielle des groupes de lie de dimension
  infinie et ses applications \`a l'hydrodynamique des fluides parfaits.
\newblock \emph{Ann. Inst. Fourier}, 16:\penalty0 319--361, 1966.

\bibitem[Bajcsy et~al.(1983)Bajcsy, Lieberson, and Reivich]{Bajcsy1983}
R.~Bajcsy, R.~Lieberson, and M.~Reivich.
\newblock A computerized system for the elastic matching of deformed
  radiographic images to idealized atlas images.
\newblock \emph{J. Comput. Assisted Tomogr.}, 7:\penalty0 618--625, 1983.

\bibitem[Bauer et~al.(2010{\natexlab{a}})Bauer, Harms, and Michor]{Michor118}
M.~Bauer, P.~Harms, and P.~W. Michor.
\newblock Almost local metrics on shape space of hypersurfaces in n-space,
  2010{\natexlab{a}}.
\newblock URL \url{arXiv:math/1001.0717}.

\bibitem[Bauer et~al.(2010{\natexlab{b}})Bauer, Harms, and Michor]{Michor119}
M.~Bauer, P.~Harms, and P.~W. Michor.
\newblock Sobolev metrics on shape space of surfaces in n-space,
  2010{\natexlab{b}}.
\newblock URL \url{arXiv:math/1009.3616}.

\bibitem[Bauer et~al.(2010{\natexlab{c}})Bauer, Harms, and Michor]{Michor120}
M.~Bauer, P.~Harms, and P.~W. Michor.
\newblock Curvature weighted metrics on shape space of hypersurfaces in
  n-space.
\newblock \url{http://www.mat.univie.ac.at/~michor/gauss-surfaces.pdf},
  2010{\natexlab{c}}.

\bibitem[Bauer(2010)]{Bauer2010}
Martin Bauer.
\newblock \emph{Almost local metrics on shape space of surfaces}.
\newblock PhD thesis, University of Vienna, 2010.

\bibitem[Beg et~al.(2005)Beg, Miller, Trouv\'{e}, and Younes]{Beg2005}
M.~F. Beg, M.~I. Miller, A.~Trouv\'{e}, and L.~Younes.
\newblock Computing large deformation metric mappings via geodesic flows of
  diffeomorphisms.
\newblock \emph{Int. J. Comput. Vision}, 61\penalty0 (2):\penalty0 139--157,
  2005.

\bibitem[Benamou et~al.(2002)Benamou, Brenier, and Guittet]{Benamou2002}
J.-D. Benamou, Y.~Brenier, and K.~Guittet.
\newblock The {M}onge-{K}antorovitch mass transfer and its computational fluid
  mechanics formulation.
\newblock \emph{Internat. J. Numer. Methods Fluids}, 40\penalty0
  (1-2):\penalty0 21--30, 2002.

\bibitem[Benamou and Brenier(2000)]{Benamou2000}
Jean-David Benamou and Yann Brenier.
\newblock A computational fluid mechanics solution to the {M}onge-{K}antorovich
  mass transfer problem.
\newblock \emph{Numer. Math.}, 84\penalty0 (3):\penalty0 375--393, 2000.

\bibitem[Besl and Jain(1985)]{Besl1985}
P.~J. Besl and R.~C. Jain.
\newblock Three-dimensional object recognition.
\newblock \emph{Comput. Surv.}, 17\penalty0 (1):\penalty0 75--145, March 1985.

\bibitem[Besse(2008)]{Besse2008}
Arthur~L. Besse.
\newblock \emph{Einstein manifolds}.
\newblock Classics in Mathematics. Springer-Verlag, Berlin, 2008.

\bibitem[Bookstein(1997)]{Bookstein1997}
Fred~L. Bookstein.
\newblock \emph{Morphometric tools for landmark data : geometry and biology}.
\newblock Cambridge University Press, 1997.

\bibitem[Cervera et~al.(1991)Cervera, Mascar{\'o}, and Michor]{Michor40}
V.~Cervera, F.~Mascar{\'o}, and P.~W. Michor.
\newblock The action of the diffeomorphism group on the space of immersions.
\newblock \emph{Differential Geom. Appl.}, 1\penalty0 (4):\penalty0 391--401,
  1991.

\bibitem[Constantin and Kolev(2003)]{Constantin2003}
Adrian Constantin and Boris Kolev.
\newblock Geodesic flow on the diffeomorphism group of the circle.
\newblock \emph{Comment. Math. Helv.}, 78\penalty0 (4):\penalty0 787--804,
  2003.

\bibitem[Durrleman et~al.(2008)Durrleman, Pennec, Trouve, Thompson, and
  Ayache]{Trouve2008b}
S.~Durrleman, X.~Pennec, A.~Trouve, P.~Thompson, and N.~Ayache.
\newblock Inferring brain variability from diffeomorphic deformations of
  currents: an integrative approach.
\newblock \emph{Med Image Anal}, 12\penalty0 (5):\penalty0 626--637, 2008.

\bibitem[Durrleman et~al.(2009)Durrleman, Pennec, Trouve, and
  Ayache]{Trouve2009}
S.~Durrleman, X.~Pennec, A.~Trouve, and N.~Ayache.
\newblock Statistical models of sets of curves and surfaces based on currents.
\newblock \emph{Med Image Anal}, 13\penalty0 (5):\penalty0 793--808, 2009.

\bibitem[Eichhorn(2007)]{Eichhorn2007}
J{\"u}rgen Eichhorn.
\newblock \emph{Global analysis on open manifolds}.
\newblock Nova Science Publishers Inc., New York, 2007.

\bibitem[Eichhorn and Fricke(1998)]{EichhornFricke1998}
J{\"u}rgen Eichhorn and Jan Fricke.
\newblock The module structure theorem for {S}obolev spaces on open manifolds.
\newblock \emph{Math. Nachr.}, 194:\penalty0 35--47, 1998.

\bibitem[Gay-Balmaz(2009)]{GayBalmaz2009}
Fran{\c{c}}ois Gay-Balmaz.
\newblock Well-posedness of higher dimensional {C}amassa-{H}olm equations.
\newblock \emph{Bull. Transilv. Univ. Bra\c sov Ser. III}, 2(51):\penalty0
  55--58, 2009.

\bibitem[Glaun\`{e}s et~al.(2008)Glaun\`{e}s, Qiu, Miller, and
  Younes]{Glaunes2008}
J.~Glaun\`{e}s, A.~Qiu, M.~I. Miller, and L.~Younes.
\newblock Large deformation diffeomorphic metric curve mapping.
\newblock \emph{Int. J. Comput. Vision}, 80\penalty0 (3), 2008.

\bibitem[Holm et~al.(2009)Holm, Trouv\'{e}, and Younes]{Holm2009}
D.~Holm, A.~Trouv\'{e}, and L.~Younes.
\newblock The {E}uler-{P}oincar\'e theory of metamorphosis.
\newblock \emph{Quart. Appl. Math.}, 67\penalty0 (4):\penalty0 661--685, 2009.

\bibitem[Kendall(1984)]{Kendall1984}
David~G. Kendall.
\newblock Shape-manifolds, procrustean metrics, and complex projective spaces.
\newblock \emph{Bull. London Mathematical Society}, 16:\penalty0 81--121, 1984.

\bibitem[Kobayashi and Nomizu(1996)]{Kobayashi1996a}
Shoshichi Kobayashi and Katsumi Nomizu.
\newblock \emph{Foundations of differential geometry. {V}ol. {I}}.
\newblock Wiley Classics Library. John Wiley \& Sons Inc., New York, 1996.

\bibitem[Kol{\'a}{\v{r}} et~al.(1993)Kol{\'a}{\v{r}}, Michor, and
  Slov{\'a}k]{MichorF}
I.~Kol{\'a}{\v{r}}, P.~W. Michor, and J.~Slov{\'a}k.
\newblock \emph{Natural operations in differential geometry}.
\newblock Springer-Verlag, Berlin, 1993.

\bibitem[Kriegl and Michor(1997)]{MichorG}
Andreas Kriegl and Peter~W. Michor.
\newblock \emph{The convenient setting of global analysis}, volume~53 of
  \emph{Mathematical Surveys and Monographs}.
\newblock American Mathematical Society, Providence, RI, 1997.

\bibitem[Kushnarev(2009)]{Kushnarev2009}
Sergey Kushnarev.
\newblock Teichons: solitonlike geodesics on universal {T}eichm\"uller space.
\newblock \emph{Experiment. Math.}, 18\penalty0 (3):\penalty0 325--336, 2009.

\bibitem[Loncaric(1998)]{Loncaric1998}
S.~Loncaric.
\newblock A survey of shape analysis techniques.
\newblock \emph{Pattern Recognition}, 31\penalty0 (8):\penalty0 983--1001,
  1998.

\bibitem[Mennucci et~al.(2008)Mennucci, Yezzi, and
  Sundaramoorthi]{MennucciYezzi2008}
A.~Mennucci, A.~Yezzi, and G.~Sundaramoorthi.
\newblock Properties of {S}obolev-type metrics in the space of curves.
\newblock \emph{Interfaces Free Bound.}, 10\penalty0 (4):\penalty0 423--445,
  2008.

\bibitem[Micheli et~al.(2010)Micheli, Michor, and Mumford]{Michor121}
M.~Micheli, P.~W. Michor, and D.~Mumford.
\newblock Sectional curvature in terms of the cometric, with applications to
  the riemannian manifolds of landmarks., 2010.
\newblock URL \url{arXiv:1009.2637}.

\bibitem[Michor(1980)]{MichorC}
Peter~W. Michor.
\newblock \emph{Manifolds of differentiable mappings}.
\newblock Shiva Publ., 1980.

\bibitem[Michor(2006)]{Michor109}
Peter~W. Michor.
\newblock Some geometric evolution equations arising as geodesic equations on
  groups of diffeomorphisms including the {H}amiltonian approach.
\newblock In \emph{Phase space analysis of partial differential equations},
  volume~69 of \emph{Progr. Nonlinear Differential Equations Appl.}, pages
  133--215. Birkh\"auser Boston, 2006.

\bibitem[Michor(2008)]{MichorH}
Peter~W. Michor.
\newblock \emph{Topics in differential geometry}, volume~93 of \emph{Graduate
  Studies in Mathematics}.
\newblock American Mathematical Society, Providence, RI, 2008.

\bibitem[Michor and Mumford(2005)]{Michor102}
Peter~W. Michor and David Mumford.
\newblock Vanishing geodesic distance on spaces of submanifolds and
  diffeomorphisms.
\newblock \emph{Doc. Math.}, 10:\penalty0 217--245 (electronic), 2005.

\bibitem[Michor and Mumford(2006)]{Michor98}
Peter~W. Michor and David Mumford.
\newblock Riemannian geometries on spaces of plane curves.
\newblock \emph{J. Eur. Math. Soc. (JEMS) 8 (2006), 1-48}, 2006.
\newblock URL \url{arxiv:math/0312384}.

\bibitem[Michor and Mumford(2007)]{Michor107}
Peter~W. Michor and David Mumford.
\newblock An overview of the {R}iemannian metrics on spaces of curves using the
  {H}amiltonian approach.
\newblock \emph{Appl. Comput. Harmon. Anal.}, 23\penalty0 (1):\penalty0
  74--113, 2007.

\bibitem[Miller et~al.(2002)Miller, Trouve, and Younes]{Trouve2002}
M.~I. Miller, A.~Trouve, and L.~Younes.
\newblock On the metrics and euler-lagrange equations of computational anatomy.
\newblock \emph{Annu Rev Biomed Eng}, 4:\penalty0 375--405, 2002.

\bibitem[Pope(1994)]{Pope1994}
A.~R. Pope.
\newblock Model-based object recognition: A survey of recent research.
\newblock Technical Report TR-94-04, University of British Columbia, January
  1994.

\bibitem[Sharon and Mumford(2004)]{Mumford2004}
E.~Sharon and D.~Mumford.
\newblock 2d-shape analysis using conformal mapping.
\newblock In \emph{Proc. IEEE Conf. Computer Vision and Pattern Recognition},
  pages 350--357, 2004.

\bibitem[Sharon and Mumford(2006)]{Mumford2006}
E.~Sharon and D.~Mumford.
\newblock 2d-shape analysis using conformal mapping.
\newblock \emph{International Journal of Computer Vision}, 70:\penalty0 55--75,
  2006.

\bibitem[Shubin(1987)]{Shubin1987}
M.~A. Shubin.
\newblock \emph{Pseudodifferential operators and spectral theory}.
\newblock Springer Series in Soviet Mathematics. Springer-Verlag, Berlin, 1987.

\bibitem[Thompson(1942)]{Thompson1942}
D'Arcy Thompson.
\newblock \emph{On Growth and Form}.
\newblock Cambridge University Press, 1942.

\bibitem[Trouv{\'e} and Younes(2000)]{TrouveYounes2000}
Alain Trouv{\'e} and Laurent Younes.
\newblock Diffeomorphic matching problems in one dimension: Designing and
  minimizing matching functionals.
\newblock In David Vernon, editor, \emph{Computer Vision}, volume 1842. ECCV,
  2000.

\bibitem[Trouv{\'e} and Younes(2005)]{Trouve2005}
Alain Trouv{\'e} and Laurent Younes.
\newblock Metamorphoses through {L}ie group action.
\newblock \emph{Found. Comput. Math.}, 5\penalty0 (2):\penalty0 173--198, 2005.

\bibitem[Vaillant and Glaun{\`e}s(2005)]{Vaillant2005}
Marc Vaillant and Joan Glaun{\`e}s.
\newblock Surface matching via currents.
\newblock In \emph{Information Processing in Medical Imaging}, volume 3565 of
  \emph{Lecture Notes in Computer Science}, pages 381--392. Springer Berlin /
  Heidelberg, 2005.

\bibitem[Veltkamp and Hagedoorn(1999)]{Veltkamp1999}
R.~C. Veltkamp and M.~Hagedoorn.
\newblock State-of-the-art in shape matching.
\newblock Technical Report UU-CS-1999-27, Utrecht University, 1999.

\bibitem[Verpoort(2008)]{Verpoort2008}
Steven Verpoort.
\newblock \emph{The geometry of the second fundamental form: Curvature
  properties and variational aspects}.
\newblock PhD thesis, Katholieke Universiteit Leuven, 2008.

\bibitem[Younes et~al.(2008)Younes, Michor, Shah, and Mumford]{Michor111}
L.~Younes, P.~W. Michor, J.~Shah, and D.~Mumford.
\newblock A metric on shape space with explicit geodesics.
\newblock \emph{Atti Accad. Naz. Lincei Cl. Sci. Fis. Mat. Natur. Rend. Lincei
  (9) Mat. Appl.}, 19\penalty0 (1):\penalty0 25--57, 2008.

\bibitem[Younes(1998)]{Younes1998}
Laurent Younes.
\newblock Computable elastic distances between shapes.
\newblock \emph{SIAM J. Appl. Math.}, 58\penalty0 (2):\penalty0 565--586
  (electronic), 1998.

\end{thebibliography}
\chapter*{Zusammenfassung}
\markboth{APPENDIX}{ZUSAMMENFASSUNG}
\addcontentsline{toc}{section}{Zusammenfassung}

\selectlanguage{ngerman}
\noindent
{\sc Zusammenfassung.} 
Geometrische Figuren spielen eine wichtige Rolle in vielen Bereichen der Wissenschaft, Technik und
Medizin. Mathematisch kann eine Figur als unparametrisierte, immersive Teil-Mannigfaltigkeit
modelliert werden, was auch der dieser Arbeit zugrunde liegende Figurenbegriff ist. Wenn der Raum
der Figuren mit einer Riemannschen Metrik ausgestattet wird, er\"offnet sich die Welt der Riemannschen
Differential-Geometrie mit Geod\"aten, Gradienten-Fl\"ussen und Kr\"ummung. Leider induziert die
einfachste solche Metrik verschwindende Kurvenl\"angen-Distanz am Raum der Figuren. Diese
Entdeckung von Michor und Mumford war der Ausgangspunkt f\"ur die Suche nach st\"arkeren Metriken,
die es erm\"oglichen sollten, Figuren anhand von f\"ur die jeweilige Anwendung wichtigen Merkmalen zu
unterscheiden. Sobolev-Metriken sind ein vielversprechender Ansatz dazu. Es gibt sie in zwei
Varianten: Sogenannte \"au{\ss}ere Metriken, die von Metriken auf der Diffeomorphismen-Gruppe des
umgebenden Raums induziert werden, und innere Metriken, die intrinsisch zur Figur definiert sind. In
dieser Arbeit werden innere Sobolev-Metriken in einem sehr allgemeinen Rahmen entwickelt und
pr\"asentiert. Es gibt keine Einschr\"ankung auf die Dimension des eingebetteten oder umgebenden
Raums, und der umgebende Raum muss nicht flach sein. Es wird gezeigt, dass die
Kurvenl\"angen-Distanz von inneren Sobolev Metriken am Raum der Figuren nicht verschwindet. Die
Geod\"atengleichung und die durch Symmetrien erzeugten Erhaltungsgr\"o{\ss}en werden hergeleitet, und
die Wohldefiniertheit der Geod\"atengleichung wird gezeigt. Beispiele von numerischen L\"osungen der
Geod\"atengleichung schlie{\ss}en die Arbeit ab.
\selectlanguage{english}

\end{document}